\documentclass[10pt]{article} 

\usepackage{algorithm2e}
\usepackage{authblk}
\usepackage{blindtext}
\usepackage[utf8]{inputenc}
\usepackage[margin=2cm]{geometry}
\usepackage{enumerate}
\usepackage{amsmath,amsthm,amssymb,amsfonts}
\usepackage{todonotes}
\usepackage{graphicx}
\usepackage{caption}
\usepackage{subcaption}
\usepackage[rightcaption]{sidecap}
\usepackage{stmaryrd}
\newtheorem{theorem}{Theorem}[section]
\newtheorem{lemma}[theorem]{Lemma}
\newtheorem{proposition}[theorem]{Proposition}
\newtheorem{corollary}[theorem]{Corollary}

\newtheorem{remark}[theorem]{Remark}
\newtheorem{example}[theorem]{Example}
\newtheorem{assumption}[theorem]{Assumption}
\usepackage{multirow}
\usepackage{hyperref}
\hypersetup{
    colorlinks=true,
    linkcolor=blue,
    filecolor=magenta,      
    citecolor=blue,
    urlcolor=blue
    }

\usepackage{todonotes}

\DeclareMathOperator{\Id}{Id}
\DeclareMathOperator{\Fix}{Fix}
\DeclareMathOperator{\prox}{prox}
\DeclareMathOperator{\gra}{gra}
\DeclareMathOperator{\zer}{zer}
\DeclareMathOperator{\dom}{dom}

\DeclareMathOperator{\sri}{sri}

\DeclareMathOperator*{\argmin}{argmin}
\DeclareMathOperator{\cone}{cone}
\DeclareMathOperator{\iso}{iso}
\newcommand{\setto}{\rightrightarrows}

\title{Linear Convergence of Resolvent Splitting with Minimal Lifting and its Application to a Primal-Dual Algorithm}

\author[*]{Farhana A. Simi} 
\author[*]{Matthew K. Tam}
\affil[*]{School of Mathematics and Statistics, University of Melbourne, Parkville VIC 3010, Australia. Email: \href{mailto:fsimi@student.unimelb.edu.au}{fsimi@student.unimelb.edu.au}, \href{mailto:matthew.tam@unimelb.edu.au}{matthew.tam@unimelb.edu.au}}

\begin{document}

\maketitle

\begin{abstract}
We consider resolvent splitting algorithms for finding a zero of the sum of finitely many maximally monotone operators. The standard approach to solving this type of problem involves reformulating as a two-operator problem in the product-space and applying the Douglas--Rachford algorithm. However, existing results for linear convergence cannot be applied in the product-space formulation due to a lack of appropriate Lipschitz continuity and strong monotonicity. In this work, we investigate a different approach that does not rely on the Douglas--Rachford algorithm or the product-space directly. We establish linear convergence of the ``resolvent splitting with minimal lifting" algorithm due to Malitsky \& Tam for monotone inclusions with finitely many operators. Our results are then used to derive linear convergence of a primal-dual algorithm for convex minimization problems involving infimal convolutions. The theoretical results are demonstrated on numerical experiments in image denoising.

\end{abstract}
\section{Introduction}
Let $\mathcal{H}$ be a real Hilbert space. In this work, we consider the monotone inclusion problem given by
\begin{equation} \label{eq:1n}
   \text{find } x\in\mathcal{H} \text{ such that } 0\in\sum_{i=1}^{n}A_{i}(x)\subseteq\mathcal{H},
   %\text{find}~x\in H~\text{such that}~0\in\sum_{i=1}^{n}A_{i}x,
\end{equation}
where the (set-valued) operator $A_{i}:\mathcal{H} \setto \mathcal{H}$ is maximally monotone for all $i\in \{1,\dots,n\}$. The setting of problem~\eqref{eq:1n} is quite general and includes many fundamental problems that arise in mathematical optimization such as nonsmooth minimization~\cite{bagirov2014introduction,{rockafellar1970monotone},{rockafellar1997convex}}, variational inequalities~\cite{marcotte1995convergence,{rockafellar1976monotone},tam2023bregman}, and fixed point problems \cite{eckstein1992douglas,lions1979splitting,setzer2009split}. 
Of particular interest for this work is the following convex minimization problem involving infimal convolution. 

\begin{example}\label{example 1.1}
Let $\mathcal{H}_{1} \text{ and } \mathcal{H}_{2}$ be real Hilbert spaces. Suppose $C:\mathcal{H}_{1}\rightarrow\mathcal{H}_{2}$ is bounded and linear, $f_{i}:\mathcal{H}_{1}\rightarrow\mathbb{R}$ is convex and differentiable with Lipschitz continuous gradient for $i=2,\dots,n-1$, $f_{n}:\mathcal{H}_{1}\rightarrow(-\infty,+\infty]$ is proper, closed and strongly convex, $g_{i}:\mathcal{H}_{2}\rightarrow(-\infty,+\infty]$ is proper, closed and strongly convex for $i=2,\dots,n-1$, and $g_{n}:\mathcal{H}_{2}\rightarrow\mathbb{R}$ is convex and differentiable with Lipschitz continuous gradient. Consider the minimization problem
\begin{equation} \label{convex optimization problem intro}
    \min_{u\in\mathcal{H}_{1}}\quad \sum_{i=2}^{n}f_{i}(u)+(g_{2}\Box\cdot\cdot\cdot\Box g_{n})(Cu),
\end{equation}
where $(g_{2}\Box\cdot\cdot\cdot\Box g_{n})$ denotes the infimal convolution of $g_{2},\dots,g_{n}$. The first order optimality condition for \eqref{convex optimization problem intro} can be expressed as the monotone inclusion
\begin{equation} \label{monotone inclusion n=2*}
    \text{find }\begin{pmatrix}
        u\\v
    \end{pmatrix}\in\mathcal{H}_{1}\times\mathcal{H}_{2}\text{ such that }\begin{pmatrix}
        0\\0
    \end{pmatrix}\in\begin{pmatrix}
        0&C^*\\-C&0
    \end{pmatrix}\begin{pmatrix}
        u\\v
    \end{pmatrix}+\sum_{i=2}^{n-1}\begin{pmatrix}
        \nabla f_{i}(u)\\\nabla g_{i}^*(v)\end{pmatrix}+\begin{pmatrix}
        \partial f_{n}(u)\\\partial g^*_{n}(v)
    \end{pmatrix},
\end{equation}
where $f^*_{i}$ and $g^*_{i}$ denote conjugates of $f_{i}$ and $g_{i}$ respectively for $i=2,\dots,n$. The inclusion problem~\eqref{monotone inclusion n=2*} is in the form of~\eqref{eq:1n}  with
\begin{equation*} \label{monotone operators}
    \mathcal{H}=\mathcal{H}_1\times\mathcal{H}_{2},\quad A_{1}=\begin{pmatrix}
        0&C^*\\-C&0
    \end{pmatrix}, \quad A_{i}=\begin{pmatrix}
        \nabla f_{i}\\ \nabla g_{i}^*\end{pmatrix},\quad 
        A_{n}=\begin{pmatrix}
        \partial f_{n}\\ \partial g_{n}^*
    \end{pmatrix},
\end{equation*}
where $i=2,\dots,n-1$.
\end{example}

\medskip

\emph{Resolvent splittings} are a family of algorithms that can be used to solve~\eqref{eq:1n}. These work by invoking each operator in~\eqref{eq:1n} individually, through their resolvents, rather than using the whole sum directly. Recall that the resolvent of a maximally monotone operator $A$ is the operator $J_{A}:\mathcal{H}\rightarrow\mathcal{H}$ defined as $J_{A}=(\Id+A)^{-1}$~\cite[Corollary]{minty1962monotone}. A well known example of a resolvent splitting, which solves the monotone inclusion problem \eqref{eq:1n} when $n=2$, is the \emph{Douglas--Rachford algorithm}~\cite{{lions1979splitting},{svaiter2011weak}}. Let $T_{\rm DR}:\mathcal{H}\rightarrow\mathcal{H}$ and ${z}^{0}\in \mathcal{H}$, this algorithm can be described in terms of the iteration
\begin{equation} \label{eq:4n}
{z}^{k+1}=T_{\rm DR}({z}^k):={z}^k+J_{A_{2}}(2J_{A_{1}}({z}^k)-{z}^k)-J_{A_{1}}({z}^k) \quad \forall k\in\mathbb{N}.
\end{equation}
The sequence $({z}^k)_{k\in \mathbb{N}}$ given by \eqref{eq:4n} converges weakly to a point ${z}\in \mathcal{H}$ with $z=T_{\rm DR}(z)$, and the \emph{shadow sequence} $\bigl(J_{A_{1}}({z}^k)\bigr)_{k\in \mathbb{N}}$ converges weakly to $J_{A_{1}}(z)$, which is a solution of \eqref{eq:1n}, see \cite[Theorem~1]{svaiter2011weak} and \cite[Theorem~2.3]{svaiter2019simplified}. Further, if one operator is Lipschitz continuous and the other is strongly monotone, then the result can be refined --- both sequences can be shown to converge linearly, see~\cite[Theorem~4.3]{moursi2019douglas} and \cite[Corollary~4.10 \& Remark~4.11]{dao1809adaptive}. Linear convergence of the Douglas--Rachford algorithm has also been established in a number of important, but specialized, settings of~\eqref{eq:1n} including where the operators are assumed to be subdifferentials~\cite{giselsson2016linear,giselsson2017tight} or normal cones~\cite{bauschke2016optimal,bauschke2014rate,bauschke2016douglas,hesse2013nonconvex,hesse2014alternating,phan2016linear}. 

The standard way to solve \eqref{eq:1n} for more than $n>2$ operators involves using the Douglas--Rachford algorithm applied to a two operator reformulation in the product space $\mathcal{H}^n$. Precisely, 
\begin{equation}\label{product space DR}
    \text{find }\mathbf{x}=(x,\dots,x)\in \mathcal{H}^n \text{ such that } 0\in (A+N_{\Delta_{n}})(\mathbf{x})\subseteq \mathcal{H}^n,
\end{equation}
where $A=(A_{1},\dots, A_{n})$,  $N_{\Delta_{n}}$ denotes the normal cone to the \emph{diagonal subspace} $\Delta_{n}:=\{\mathbf{x}=(x_{1},\dots, x_{n})\in \mathcal{H}^n: x_{1}=\dots= x_{n}\}$. Any solution $\mathbf{x}=(x,\dots,x)$ of \eqref{product space DR} is necessarily contained in $\Delta_n$ with $x$ a solution to \eqref{eq:1n}, and vice versa. However, many of the existing results for linear convergence of the Douglas--Rachford algorithm do not apply to \eqref{product space DR} as the normal cone $N_{\Delta_{n}}$ is neither Lipschitz continuous nor strongly monotone.

This study aims to establish linear convergence of the ``resolvent splitting algorithm with minimal lifting" due to Malitsky and Tam~\cite{malitsky2023resolvent}. This algorithm does not rely on a product space formulation in solving the inclusion problem~\eqref{eq:1n}. Let $T_{\rm MT}:\mathcal{H}^{n-1}\rightarrow\mathcal{H}^{n-1}$, $\mathbf{z}^0=(z_{1}^0,\dots, z_{n-1}^0)\in \mathcal{H}^{n-1}$, and $\gamma\in(0, 1)$, this algorithm can be described in terms of the iteration
\begin{equation}\label{eq:1}
    \mathbf{z}^{k+1}=T_{\rm MT}(\mathbf{z}^k)=\mathbf{z}^k+\gamma\begin{pmatrix}
x_{2}^{k}-x_{1}^{k}\\x_{3}^{k}-x_{2}^{k}\\\vdots \\x_{n}^{k}-x_{n-1}^{k}
\end{pmatrix},
\end{equation}
where $\mathbf{x}^k=(x_{1}^k,\dots,x_{n}^{k})\in\mathcal{H}^{n}$ depends on $\mathbf{z}=(z_{1}^k, \dots, z_{n-1}^k)\in \mathcal{H}^{n-1}$ and is given by\\
\begin{equation} \label{eq:2}
\left\{\begin{aligned} 
x_{1}^k &=J_{A_{1}}(z_{1}^k)\\
x_{i}^k &=J_{A_{i}}(z_{i}^k+x_{i-1}^k-z_{i-1}^k)&\forall i\in \{2,\dots,n-1\}  \\
x_{n}^k &=J_{A_{n}}(x_{1}^k+x_{n-1}^k-z_{n-1}^k).
\end{aligned}\right.
\end{equation}
The sequence $(\mathbf{z}^k)_{k\in\mathbb{N}}$ given by~\eqref{eq:1} converges weakly to a point $\mathbf{z}^*\in\mathcal{H}^{n-1}$ with $\mathbf{z}^*=T_{\rm MT}(\mathbf{z^*})$, and the shadow sequence $(\mathbf{x}^k)_{k\in\mathbb{N}}$ converges weakly to a point $(x,\dots,x)\in\mathcal{H}^n$ with $x=J_{A_{1}}(z_{1})$, which is a solution of \eqref{eq:1n}, see \cite[Theorem 4.5]{malitsky2023resolvent}. Although this algorithm is known to converge linearly for affine feasibility problems~\cite{bauschke2023splitting}, linear convergence in the setting of \eqref{eq:1n} has not been previously studied. In this work, we address this by establishing linear convergence of this algorithm when applied to the inclusion problems~\eqref{eq:1n}. 

The remainder of this paper is structured as follows. In Section~\ref{s: prel}, we recall the preliminaries needed for our analysis. In Section~\ref{s:resolvent splitting}, we present our main result (Theorem~\ref{theorem for linear convergence}) concerning linear convergence of the ``resolvent splitting with minimal lifting" algorithm \cite{malitsky2023resolvent} for problem~\eqref{eq:1n} with $n\geq2$. When specialized to $n=2$ operators, our result generalizes the findings presented in~\cite{moursi2019douglas}. 
In Section~\ref{s: section 4}, we apply the results of Section~\ref{s:resolvent splitting} to derive linear convergence of a primal-dual algorithm for the convex minimization problem with infimal convolution given in Example~\ref{example 1.1}. In Section~\ref{s: Experiment}, we present experimental results on image denoising which are supported by our findings. Finally, Section~\ref{s: conclusions} concludes by outlining future directions and open question for future research.

\section{Preliminaries}\label{s: prel}
Throughout this paper, $\mathcal{H}$ denotes a real Hilbert space with inner product $\langle\cdot,\cdot\rangle$ and induced norm $\|\cdot\|$. A \emph{set-valued} operator, denoted $A:\mathcal{H}\setto \mathcal{H}$, maps each point $x\in \mathcal{H}$ to a set $A(x)\subseteq \mathcal{H}$. When $A$ is \emph{single-valued} (\emph{i.e.,}~$A(x)$ is a singleton for all $x\in\mathcal{H})$, we write $A:\mathcal{H}\rightarrow\mathcal{H}$. The \emph{graph}, the set of \emph{fixed points} and the set of \emph{zeros} of the operator $A\colon\mathcal{H}\setto\mathcal{H}$ are defined by
$\gra A:=\{(x,u)\in \mathcal{H}\times\mathcal{H}:u\in A(x)\}, \Fix A:=\{x\in \mathcal{H}:x\in A(x)\}$, and $\zer A:=\{x\in \mathcal{H}:0\in A(x)\}$ respectively. The \emph{identity operator} is denoted by $\Id:\mathcal{H}\rightarrow \mathcal{H}$.

An operator $A:\mathcal{H}\setto\mathcal{H}$ is $\mu$-\emph{monotone} if
$$\langle x-y,u-v\rangle\geq\mu\|x-y\|^2\quad \forall (x,u),(y,v)\in \gra A,$$
and it is \emph{maximally $\mu$-monotone}, if there exists no $\mu$-monotone operator $B:\mathcal{H}\setto\mathcal{H}$ such that $\gra B$ properly contains $\gra A$. Depending on the sign of $\mu$, we say $A$ is monotone if $\mu=0$ and $A$ is $\mu$-\emph{strongly monotone} if $\mu>0$. 

A single-valued operator $B:\mathcal{H}\rightarrow\mathcal{H}$ is $\beta$-\emph{Lipschitz}, with $\beta\geq0$, if $$\|B(x)-B(y)\|\leq\beta\|x-y\|\quad \forall (x,y)\in\mathcal{H},$$
and a $\beta$-Lipschitz operator with $\beta\in[0,1)$ is said to be a \emph{$\beta$-contraction}. A $1$-Lipschitz operator is said to be \emph{nonexpansive}. 

The \emph{resolvent} of an operator $A:\mathcal{H}\setto\mathcal{H}$ is defined as $J_{A}:=(\Id+A)^{-1}$. The following proposition summarises its key properties in the presence of monotonicity.
\begin{proposition}\label{nonexpansiveness}
    Let $A:\mathcal{H}\setto\mathcal{H}$ be maximally monotone operator. Then the resolvent $J_{A}$ is single-valued with full domain and satisfies 
    $$ \|J_{A}(x)-J_{A}(y)\|^2+\|(\Id-J_{A})(x)-(\Id-J_{A})(y)\|^2\leq\|x-y\|^2\quad\forall (x,y)\in\mathcal{H}.$$
    In particular, $J_A$ is a nonexpansive.
\end{proposition}
\begin{proof}
    See \cite[Corollary~23.10]{bauschke2011convex}.
\end{proof}
The following theorem will be important for establishing linear convergence. Recall that a sequence $({z}^k)_{k\in\mathbb{N}}$ is said to converge \emph{$R$-linearly} to a point $z\in\mathcal{H}$ if there exists $c\in\mathbb{R}_+$ and $r\in[0,1)$ such that $\|{z}^{k}-{z}\|\leq cr^k$ for all $k\in\mathbb{N}$.
\begin{theorem}[\emph{Banach fixed-point theorem}]\label{Banach Theorem}
    Let $T:\mathcal{H}\rightarrow\mathcal{H}$ be $\beta$-contraction. Given $z^0\in\mathcal{H}$, define a sequence $(z^k)_{k\in\mathbb{N}}$ according to $$z^{k+1}=T(z^k) \quad \forall k\in\mathbb{N}.$$
    Then there exists $z\in\mathcal{H}$ such that the following hold:
    \begin{enumerate}[(i)]
        \item  $z$ is the unique fixed point of $T$.
        \item $\|z^k-z\|\leq\beta^k\|z^0-z\|$ for all 
        $k\in\mathbb{N}$.
        %\item (a priori error estimate) $\|z^k-z\|\leq\frac{\beta^k}{1-\beta}\|z^0-z^1\|$ for all $k\in\mathbb{N}$.
    \end{enumerate}
In particular, the sequence $(z^k)_{k\in\mathbb{N}}$ converges $R$-linearly to $z$. 
\end{theorem}
\begin{proof}
    See \cite[Theorem 1.48]{bauschke2011convex}.
\end{proof}

Given a function $f:\mathcal{H}\rightarrow[-\infty,+\infty]$, we say $f$ is \emph{proper}, if $-\infty\notin f(\mathcal{H})$ and $\dom f:=\{x\in\mathcal{H}:f(x)<+\infty\}\neq\emptyset$. We say $f$ is \emph{lower semi-continuous (lsc)} at $\Bar{x}\in\mathcal{H}$ if $$\liminf_{x\rightarrow\bar{x}}f(x)\geq f(\Bar{x}),$$ and say it is \emph{lower semi-continuous (lsc)}, if it is lsc at every point in $\mathcal{H}$. 
A function $f$ is \emph{convex}, if
    $$f((1-\lambda)x+\lambda y)\leq\lambda f(x)+(1-\lambda)f(y) \quad \forall (x,y)\in\mathcal{H},\quad \lambda\in(0,1),$$
and $f$ is $\alpha$-\emph{strongly convex}, with $\alpha>0$, if $f-\frac{\alpha}{2}\|\cdot\|^2$ is convex. 
The \emph{conjugate (Fenchel conjugate)} of $f$ is the function $f^*:\mathcal{H}\rightarrow[-\infty,+\infty]$ defined by
$$f^*(u)=\sup_{x\in\mathcal{H}}(\langle x,u\rangle-f(x)).$$
The \emph{infimal convolution} of $f_{1},\dots, f_{n}:\mathcal{H}\rightarrow(-\infty,+\infty]$  is the function $(f_{1}\Box\cdots\Box f_{n}):\mathcal{H}\rightarrow[-\infty,+\infty]$ defined by
\begin{equation}\label{infimal convolution}
    (f_{1}\Box\cdots\Box f_{n})(u)=\inf_{(v_{1},\dots,v_{n})\in\mathcal{H}\times\dots\times\mathcal{H}}\{f_{1}(v_{1})+\cdots+f_{n}(v_{n}):u=v_{1}+\dots+v_{n}\}.
\end{equation}
and it is said to be \emph{exact} at a point $u\in\mathcal{H}$, if the infimum in \eqref{infimal convolution} is attained. The following two proposition explore properties of the infimal convolution.

\begin{proposition}\label{remark infimal convolution}
    Suppose $f_{1},\dots,f_{n}:\mathcal{H}\rightarrow(-\infty,+\infty]$ are proper convex functions. Then $$(f_{1}\Box\cdots\Box f_{n})^*=f^*_{n}+\dots+f^*_{n}.$$
\end{proposition}
\begin{proof}
    See \cite[Theorem 16.4]{rockafellar1997convex}.
\end{proof}

\begin{proposition}\label{prop for infimal convolution}
    Suppose $f_{1},\dots,f_{n-1}:\mathcal{H}\rightarrow(-\infty,+\infty]$ are proper lsc $\alpha$-strongly convex, and $f_{n}:\mathcal{H}\rightarrow(-\infty,+\infty)$ is convex. Then $(f_{1}\Box\cdots\Box f_{n})\colon\mathcal{H}\to(-\infty,+\infty)$ is convex and exact at every $v\in\mathcal{H}.$
\end{proposition}
\begin{proof}
Convexity of $f_{1}\Box\cdots\Box f_{n}$ follows by applying \cite[Proposition~8.26]{bauschke2011convex} to the function $F_1:\mathcal{H}\times\mathcal{H}^{n-1}\rightarrow(-\infty,+\infty]:(u,(v_1,\dots,v_{n-1}))\mapsto\sum_{i=1}^{n-1}f_{i}(v_{i})+f_{n}\bigl(u-\sum_{i=1}^{n-1}v_{i}\bigr)$. To show $f_{1}\Box\cdots\Box f_{n}$ is exact, fix $u\in\mathcal{H}$ and consider the convex function 
    $$F_2(v_1,\dots,v_{n-1}):=\sum_{i=1}^{n-1}f_{i}(v_{i})+f_{n}\bigl(u-\sum_{i=1}^{n-1}v_{i}\bigr),$$
where we note that $\dom F_2\supseteq \dom f_1\times\dots\times\dom f_{n-1}$ as $\dom f_n=\mathcal{H}$. Since $f_1,\dots,f_{n-1}$ are proper and lsc, it follows that $F_2$ is also proper and lsc.
Since $f_1,\dots,f_{n-1}$ are $\alpha$-strongly convex on $\mathcal{H}$, it follows that $F_2$ is $\alpha$-strongly convex on $\mathcal{H}^{n-1}$. Applying \cite[Corollary 11.17]{bauschke2011convex} to the proper lsc $\alpha$-convex function $F_2$ implies it has exactly one minimizer. Since $u\in\mathcal{H}$ was chosen arbitrarily, this completes the proof. 
\end{proof}

The \emph{subdifferential} of a function $f:\mathcal{H}\rightarrow(-\infty,+\infty]$ at $x\in\dom f$ is given by $$\partial f(x):=\{u\in\mathcal{H}:\langle y-x,u\rangle+f(x)\leq f(y), \forall y\in\mathcal{H}\},$$ and  at $x\notin \dom f$ it is defined as $\partial f(x):=\emptyset$. 
In order to compute the subdifferential of the sum of two functions, we will make use the following sum-rule which assumes a condition involving the strong relative interior. Recall that a set $D\subseteq\mathcal{H}$ is \emph{cone} if it satisfies $D=\mathbb{R}_{++}D$. The smallest cone in $\mathcal{H}$  containing $D$ is denoted $\cone D$, and the smallest closed linear subspace of $\mathcal{H}$ containing $D$ is denoted $\overline{\text{span} D}$. The \emph{strong relative interior} of $D$ is given by
$$\sri D:=\{x\in D: \cone(D-x)=\overline{\text{span}(D-x)}\}.$$
Note that when $\mathcal{H}$ is finite-dimensional, the notion of strong relative interior coincides with the usual notation of \emph{relative interior}~\cite[Fact 6.14(i)]{bauschke2011convex}.
\begin{theorem}\label{sum rule of subdifferential for two functions}
   Let $\mathcal{H}_{1}$ and $\mathcal{H}_{2}$ be real Hilbert spaces. Suppose $f:\mathcal{H}_{1}\rightarrow(-\infty,+\infty]$ and $g:\mathcal{H}_{2}\rightarrow(-\infty,+\infty]$ are proper lsc convex functions, and $C:\mathcal{H}_{1}\rightarrow\mathcal{H}_{2}$ is bounded and linear. If $0\in\sri(\dom g-C\dom f)$ then
   $$\partial(f+g\circ C)=\partial f+C^*\circ\partial g\circ C.$$
\end{theorem}
\begin{proof}
    See \cite[Theorem 16.37(i)]{bauschke2011convex}.
\end{proof}
Now introduce the following proposition which will be useful for simplifying our result.
\begin{proposition}\label{lemma for gap}
    Suppose $f\colon\mathcal{H}\to(-\infty,+\infty]$ is proper lsc convex, and $(u^k)$ converges $R$-linearly to $u$. If there exists a bounded sequence of subgradients $\phi^k\in\partial f(u^k)$ and $\partial f(u)\neq \emptyset$, then $f(u^k)$ converges $R$-linearly to $f(u)$.
\end{proposition}
\begin{proof}
By assumption, there exists $M>0$ such that $\|\phi^k\|\leq M$ for all $k\in\mathbb{N}$. On one hand, since $\phi^k\in\partial f(u^k)$, we have
 $f(u^k)-f(u)\leq \langle \phi^k,u^k-u\rangle \leq \|\phi^k\|\|u^k-u\|\leq M\|u^k-u\|. $
On the other hand, for any $\phi\in\partial f(u)\neq\emptyset$, we have
 $  f(u)-f(u^k)\leq \langle \phi,u-u^k\rangle \leq \|\phi\|\|u-u^k\|. $
Since $(u^k)$ converges $R$-linearly to $u$, the result follows by combining these inequalities.
\end{proof}
Given a proper lsc convex function $f:\mathcal{H}\rightarrow(-\infty,+\infty]$, its \emph{proximal operator} \cite[Definition 12.23]{bauschke2011convex}, denoted by $\prox_{f}\colon\mathcal{H}\rightarrow\mathcal{H}$, is given by 
$$\prox_f:=\argmin_{u\in\mathcal{H}}\left\{f(u)+\frac{1}{2}\|\cdot-u\|^2\right\}.$$
The proximal operator of $f$ be can viewed as the resolvent of $\partial f$. In other words, $J_{\partial f}=\prox_{f}$ (see \cite[Example 23.3]{bauschke2011convex}). Finally, we recall the \emph{Moreau decomposition} which relates the proximal operator of a function to the proximal operator of its conjugate. 
\begin{theorem}[\emph{Moreau decomposition}]\label{Moreau decomposition}
    Let $f:\mathcal{H}\rightarrow(-\infty,+\infty]$ be a proper lsc convex function. Then
$$x=\prox_f(x)+\prox_{f^*}(x) \quad \forall x\in\mathcal{H}.$$
\end{theorem}
\begin{proof}
    See \cite[Remark 14.4]{bauschke2011convex}.
\end{proof}
\section{Linear Convergence of Resolvent Splitting with Minimal Lifting}\label{s:resolvent splitting}
In this section, we establish linear convergence of the algorithm given by \eqref{eq:1} and \eqref{eq:2} for solving the inclusion \eqref{eq:1n}. This algorithm is a fixed-point algorithm based on the operator $T_{\rm MT}:\mathcal{H}^{n-1}\rightarrow\mathcal{H}^{n-1}$ defined as 
\begin{equation}\label{eq: fixed point operator}
    T_{\rm MT}(\mathbf{z})=\mathbf{z}+\gamma\begin{pmatrix}
x_{2}-x_{1}\\x_{3}-x_{2}\\\vdots\\x_{n}-x_{n-1}
\end{pmatrix},
\end{equation}
where $\mathbf{x}=(x_{1},\dots,x_{n})\in\mathcal{H}^{n}$ depends on $\mathbf{z}=(z_{1},\dots, z_{n-1})\in \mathcal{H}^{n-1}$ and is given by\\
\begin{equation} \label{eq: def of x}
\left\{\begin{aligned} 
x_{1} &=J_{A_{1}}(z_{1})\\
x_{i} &=J_{A_{i}}(z_{i}+x_{i-1}-z_{i-1})&\forall i\in \{2,\dots,(n-1)\}  \\
x_{n} &=J_{A_{n}}(x_{1}+x_{n-1}-z_{n-1}).
\end{aligned}\right.
\end{equation}
Our analysis identifies conditions under which the operator $T_{\rm MT}$ is a $\beta$-contraction with $\beta\in(0,1)$, as detailed in Lemma~\ref{lemma for contraction factor}, and our main regarding linear convergence is given in Theorem~\ref{theorem for linear convergence}. 

We will use the following lemmas to simplify the presentation of our main result. We begin by recalling the following Lemma~\ref{new lemma} concerning fixed point of $T_{\rm MT}$.
\begin{lemma}\label{new lemma}
    Let $n\geq2$ and $\gamma\in(0,1)$. Suppose $A_{1},\dots,A_{n}:\mathcal{H}\setto\mathcal{H}$ are maximally monotone. Let $\mathbf{z}^*=(z^*_{1},\dots,z^*_{n-1})\in\Fix T_{MT}$ and set $x^*=J_{A_{1}}({z_{1}}^*)$. Then 
    $x^*\in\zer(\sum_{i=1}^n A_{i})$,  and 
    \begin{equation} \label{eq: def of x^*}
x^* =J_{A_{i}}(z^*_{i}+x^*-z^*_{i-1})=J_{A_{n}}(2x^*-z^*_{n-1})\quad \forall i\in \{2,\dots,(n-1)\}.
\end{equation}
\end{lemma}
\begin{proof}
    See \cite[Lemma 4.2]{malitsky2023resolvent}.
\end{proof}

The following lemma refines \cite[Lemma 4.3]{malitsky2023resolvent} and its proof to the setting where some of the operators are potentially strongly monotone. 
\begin{lemma} \label{lemma 3.1}
Let $n\geq 2$ and $\gamma\in(0, 1)$. Suppose $A_{1},\dots,A_{n}: \mathcal{H}\setto \mathcal{H}$ are maximally $\mu_{i}$-monotone with $\mu_{i}\geq0$ for $i\in\{1,\dots,n\}$. Then, for all $\mathbf{z}=(z_{1},\dots, z_{n-1})\in \mathcal{H}^{n-1}$ and $\mathbf{\Bar{z}}=(\bar{z}_{1},\dots, \bar{z}_{n-1})\in \mathcal{H}^{n-1}$, we have
\begin{multline} \label{eq:3}
    \| T_{\rm MT}(\mathbf{z})-T_{\rm MT}(\Bar{\mathbf{z}})\|^2 +\gamma(1-\gamma)\sum_{i=1}^{n-1}\|({x}_{i}-{x}_{i+1})-(\Bar{x}_{i}-\Bar{{x}}_{i+1})\|^2+\gamma\|(x_{n}-x_{1})-(\Bar{x}_{n}-\Bar{x}_{1})\|^2\\
    \leq \|\mathbf{z}-\bar{\mathbf{z}}\|^2-2\gamma\sum_{i=1}^{n}\mu_{i}\|x_{i}-\bar{x}_{i}\|^2,
\end{multline}
where $T_{\rm MT}:\mathcal{H}^{n-1}\rightarrow \mathcal{H}^{n-1}$ is defined by \eqref{eq: fixed point operator}, $\mathbf{x}=(x_{1},\dots,x_{n})\in \mathcal{H}^{n}$ is given by \eqref{eq: def of x} and $\Bar{\mathbf{x}}=(\Bar{x}_{1},\dots,\bar{x}_{n})\in \mathcal{H}^{n}$ is given analogously.
\end{lemma}
\begin{proof}
For convenience, denote $\mathbf{z}^+:= T_{\rm MT}(\mathbf{z})$ and $\Bar{\mathbf{z}}^+:=T_{\rm MT}(\Bar{\mathbf{z}})$. Since $z_{1}-x_{1}\in A_{1}(x_{1})$ and $\bar{z}_{1}-\bar{x}_{1}\in A_{1}(\bar{x}_{1})$, maximally $\mu_{1}$-monotonicity of $A_{1}$ implies
\begin{equation} \label{eq:4}
\begin{aligned}
 \mu_{1}\|x_{1}-\bar{x}_{1}\|^2&\leq\left<x_{1}-\bar{x}_{1},(z_{1}-x_{1})-(\bar{z}_{1}-\bar{x}_{1})\right>\\
 &=\left<x_{2}-\bar{x}_{1},(z_{1}-x_{1})-(\bar{z}_{1}-\bar{x}_{1})\right>+\left<x_{1}-x_{2},(z_{1}-x_{1})-(\bar{z}_{1}-\bar{x}_{1})\right>.
\end{aligned}
\end{equation}
For $i\in\{2,\dots,n-1\}, z_{i}-z_{i-1}+x_{i-1}-x_{i}\in A_{i}(x_{i})$ and $\bar{z}_{i}-\bar{z}_{i-1}+\bar{x}_{i-1}-\bar{x}_{i}\in A_{i}(\bar{x}_{i})$. Thus maximally $\mu_{i}$-monotonicity of $A_{i}$ yields
\begin{equation*}
\begin{aligned}
    \mu_{i}\|x_{i}-\bar{x}_{i}\|^2&\leq\langle x_{i}-\bar{x}_{i}, (z_{i}-z_{i-1}+x_{i-1}-x_{i})-(\bar{z}_{i}-\bar{z}_{i-1}+\bar{x}_{i-1}-\bar{x}_{i})\rangle\\&=\langle x_{i}-\bar{x}_{i}, (z_{i}-x_{i})-(\bar{z}_{i}-\bar{x}_{i})\rangle-\langle x_{i}-\bar{x}_{i}, (z_{i-1}-x_{i-1})-(\bar{z}_{i-1}-\bar{x}_{i-1})\rangle\\
    &=\langle x_{i+1}-\bar{x}_{i}, (z_{i}-x_{i})-(\bar{z}_{i}-\bar{x}_{i})\rangle+\langle x_{i}-{x}_{i+1}, (z_{i}-x_{i})-(\bar{z}_{i}-\bar{x}_{i})\rangle\\
    &\qquad -\left<x_{i}-\bar{x}_{i-1}, (z_{i-1}-x_{i-1})-(\bar{z}_{i-1}-\bar{x}_{i-1})\right>-\left<\bar{x}_{i-1}-\bar{x}_{i}, (z_{i-1}-x_{i-1})-(\bar{z}_{i-1}-\bar{x}_{i-1})\right>.
\end{aligned}
\end{equation*}
Summing this inequality for $i\in\{2,\dots,n-1\}$ and simplifying gives
\begin{multline} \label{eq:5}
    \mu_{i}\|x_{i}-\bar{x}_{i}\|^2\leq\left<x_{n}-\bar{x}_{n}, (z_{n-1}-x_{n-1})-(\bar{z}_{n-1}-\bar{x}_{n-1})\right>-\left<x_{2}-\bar{x}_{1}, (z_{1}-x_{1})-(\bar{z}_{1}-\bar{x}_{1})\right>\\
    +\sum_{i=2}^{n-1}\left<x_{i}-{x}_{i+1}, (z_{i}-x_{i})-(\bar{z}_{i}-\bar{x}_{i})\right>-\sum_{i=1}^{n-2}\left<\bar{x}_{i}-\bar{x}_{i+1}, (z_{i}-x_{i})-(\bar{z}_{i}-\bar{x}_{i})\right>.
\end{multline}
Since $x_{1}+x_{n-1}-x_{n}-z_{n-1}\in A_{n}(x_{n})$ and $\bar{x}_{1}+\bar{x}_{n-1}-\bar{x}_{n}-\bar{z}_{n-1}\in A_{n}(\bar{x}_{n})$, maximally $\mu_{n}$-monotonicity of $A_{n}$ gives
\begin{equation} \label{eq:6}
\begin{aligned} 
\mu_{n}\|x_{n}-\Bar{x}_{n}\|^2&\leq\langle x_{n}-\bar{x}_{n}, (x_{1}+x_{n-1}-x_{n}-z_{n-1})-(\bar{x}_{1}+\bar{x}_{n-1}-\bar{x}_{n}-\bar{z}_{n-1})\rangle\\
&=\langle x_{n}-\bar{x}_{n}, (x_{n-1}-z_{n-1})-(\bar{x}_{n-1}-\bar{z}_{n-1})\rangle+\langle x_{n}-\bar{x}_{n}, (x_{1}-\bar{x}_{1})-({x}_{n}-\bar{x}_{n})\rangle\\
&=-\langle x_{n}-\bar{x}_{n-1},(z_{n-1}-x_{n-1})-(\bar{z}_{n-1}-\bar{x}_{n-1})\rangle+\langle\bar{x}_{n}-\bar{x}_{n-1},(z_{n-1}-x_{n-1})-(\bar{z}_{n-1}-\bar{x}_{n-1})\rangle\\
  &\qquad +\frac{1}{2}(\|x_{1}-\bar{x}_{1}\|^2-\|x_{n}-\bar{x}_{n}\|^2-\|(x_{1}-x_{n})-(\bar{x}_{1}-\bar{x}_{n})\|^2).
\end{aligned}  
\end{equation}
Adding \eqref{eq:4}, \eqref{eq:5}, and \eqref{eq:6} and rearranging gives
\begin{multline} \label{eq:7}
    \sum_{i=1}^n\mu_{i}\|x_{i}-\bar{x}_{i}\|^2\leq\sum_{i=1}^{n-1}\langle(x_{i}-\bar{x}_{i})-(x_{i+1}-\bar{x}_{i+1}), \bar{x}_{i}-x_{i}\rangle+\sum_{i=1}^{n-1}\langle(x_{i}-\bar{x}_{i})-(x_{i+1}-\bar{x}_{i+1}), {z}_{i}-\bar{z}_{i}\rangle\\+\frac{1}{2}(\|x_{1}-\bar{x}_{1}\|^2-\|x_{n}-\bar{x}_{n}\|^2-\|(x_{1}-x_{n})-(\bar{x}_{1}-\bar{x}_{n})\|^2).
\end{multline}
The first term in \eqref{eq:7} can be expressed as
\begin{equation} \label{eq:8}
 \begin{aligned}
 &\sum_{i=1}^{n-1}\langle(x_{i}-\bar{x}_{i})-(x_{i+1}-\bar{x}_{i+1}), \bar{x}_{i}-x_{i}\rangle\\
   &=\frac{1}{2}\sum_{i=1}^{n-1}(\|x_{i+1}-\bar{x}_{i+1}\|^2-\|x_{i}-\bar{x}_{i}\|^2-\|(x_{i}-x_{i+1})-(\bar{x}_{i}-\bar{x}_{i+1})\|^2)\\
   &=\frac{1}{2}(\|x_{n}-\bar{x}_{n}\|^2-\|x_{1}-\bar{x}_{1}\|^2-\sum_{i=1}^{n-1}\|(x_{i}-x_{i+1})-(\bar{x}_{i}-\bar{x}_{i+1})\|^2).
\end{aligned}   
\end{equation}
Also the second term in \eqref{eq:7} can be written as
\begin{equation} \label{eq:9}
    \begin{aligned}
    &\sum_{i=1}^{n-1}\left<(x_{i}-\bar{x}_{i})-(x_{i+1}-\bar{x}_{i+1}), {z}_{i}-\bar{z}_{i}\right>\\
    &=\frac{1}{\gamma}\sum_{i=1}^{n-1}\left<(z_{i}-z_{i}^+)-(\bar{z}_{i}-\bar{z}_{i}^+),z_{i}-\bar{z}_{i}\right>\\
    &=\frac{1}{\gamma}\left<(\mathbf{z}-\mathbf{z}^+)-(\bar{\mathbf{z}}-\bar{\mathbf{z}}^+), \mathbf{z}-\bar{\mathbf{z}}\right>\\
     &=\frac{1}{2\gamma}\left(\|(\mathbf{z}-\mathbf{z}^+)-(\bar{\mathbf{z}}-\bar{\mathbf{z}}^+)\|^2+\|\mathbf{z}-\bar{\mathbf{z}}\|^2-\|\mathbf{z}^+-\bar{\mathbf{z}}^+\|^2\right)\\
     &=\frac{1}{2\gamma}\left(\sum_{i=1}^{n-1}\|(z_{i}-z^+_{i})-(\bar{z}_{i}-\bar{z}^+_{i})\|^2+\|\mathbf{z}-\bar{\mathbf{z}}\|^2-\|\mathbf{z}^+-\bar{\mathbf{z}}^+\|^2\right)\\
     &=\frac{\gamma}{2}\sum_{i=1}^{n-1}\|(x_{i}-x_{i+1})-(\bar{x}_{i}-\bar{x}_{i+1})\|^2+\frac{1}{2\gamma}\left(\|\mathbf{z}-\bar{\mathbf{z}}\|^2-\|\mathbf{z}^+-\bar{\mathbf{z}}^+\|^2\right).
    \end{aligned}
\end{equation}
Thus substituting \eqref{eq:8} and \eqref{eq:9} into \eqref{eq:7}, and simplifying gives \eqref{eq:3}. This completes the proof.
\end{proof}
In what follows, we will make frequent use of the inequality
\begin{equation}\label{inequality}
    ab\leq \frac{1}{2\epsilon}a^2+\frac{\epsilon}{2}b^2\text{ for }a,b\geq0 \text{ and }\epsilon>0.
\end{equation}
\begin{lemma}\label{lipschitz operators}
    Let $n\geq 2$. Suppose that $A_{1},\dots,A_{n-1}: \mathcal{H}\rightarrow \mathcal{H}$ are maximally monotone and $L$-Lipschitz, and $A_{n}:\mathcal{H}\setto\mathcal{H}$ is maximally monotone. Then there exists $\eta\in(0,1)$ such that for all $\mathbf{z}=(z_{1},\dots, z_{n-1})\in \mathcal{H}^{n-1}$ and $\mathbf{\Bar{z}}=(\bar{z}_{1},\dots, \bar{z}_{n-1})\in \mathcal{H}^{n-1}$, we have
    \begin{equation}\label{lipschitz for n*}
    \sum_{i=1}^{n-1}\|x_{i}-\Bar{x}_{i}\|^2\geq \eta\|\mathbf{z}-\bar{\mathbf{z}}\|^2, 
\end{equation}
where $\mathbf{x}=(x_{1},\dots,x_{n})\in \mathcal{H}^{n}$ is given by \eqref{eq: def of x}, and $\Bar{\mathbf{x}}=(\Bar{x}_{1},\dots,\bar{x}_{n})\in \mathcal{H}^{n}$ is given analogously.
\end{lemma}
\begin{proof}
Since $z_{1}-x_{1}\in A_{1}(x_{1})$ and $\bar{z}_{1}-\bar{x}_{1}\in A_{1}(\bar{x}_{1})$, $L$-Lipschitz continuity of $A_{1}$ implies
\begin{align} \label{eq34}
    L^2\|x_{1}-\Bar{x}_{1}\|^2\geq\|A_{1}(x_{1})-A_{1}(\bar{x}_{1})\|^2=\|{(z_{1}-x_{1})-(\Bar{z}_{1}-\Bar{x}_{1})}\|^2.
\end{align}    
For $i\in\{2,\dots,n-1\}, z_{i}-z_{i-1}+x_{i-1}-x_{i}\in A_{i}(x_{i})$ and $\bar{z}_{i}-\bar{z}_{i-1}+\bar{x}_{i-1}-\bar{x}_{i}\in A_{i}(\bar{x}_{i})$. Thus, for any $\epsilon_{i}>0$, $L$-Lipschitz continuity of $A_{i}$ followed by applying \eqref{inequality} yields
\begin{equation}\begin{aligned}\label{eq:A_i Lips}
    L^2\| x_{i}-\bar{x}_{i}\|^2&\geq
\| A_{i}(x_{i})-A_{i}(\bar{x}_{i})\|^2\\
&=\|(z_{i}-z_{i-1}+x_{i-1}-x_{i})-(\bar{z}_{i}-\bar{z}_{i-1}+\bar{x}_{i-1}-\bar{x}_{i})\|^2\\
&=\|\{(z_{i}-x_{i})-(\bar{z}_{i}-\bar{x}_{i})\}-\{(z_{i-1}-x_{i-1})-(\bar{z}_{i-1}-\bar{x}_{i-1})\}\|^2\\
&=\|(z_{i}-x_{i})-(\bar{z}_{i}-\bar{x}_{i})\|^2+\|(z_{i-1}-x_{i-1})-(\bar{z}_{i-1}-\bar{x}_{i-1})\|^2\\&\qquad-2\langle(z_{i}-x_{i})-(\bar{z}_{i}-\bar{x}_{i}),(z_{i-1}-x_{i-1})-(\bar{z}_{i-1}-\bar{x}_{i-1})\rangle\\
&\geq\|(z_{i}-x_{i})-(\bar{z}_{i}-\bar{x}_{i})\|^2+\|(z_{i-1}-x_{i-1})-(\bar{z}_{i-1}-\bar{x}_{i-1})\|^2\\
&\qquad-\frac{1}{\epsilon_{i}}\|(z_{i}-x_{i})-(\bar{z}_{i}-\bar{x}_{i})\|^2-\epsilon_{i}\|(z_{i-1}-x_{i-1})-(\bar{z}_{i-1}-\bar{x}_{i-1})\|^2\\
&=(1-\frac{1}{\epsilon_{i}})\|(z_{i}-x_{i})-(\bar{z}_{i}-\bar{x}_{i})\|^2+(1-\epsilon_{i})\|(z_{i-1}-x_{i-1})-(\bar{z}_{i-1}-\bar{x}_{i-1})\|^2.
\end{aligned}\end{equation}
Summing the inequality~\eqref{eq:A_i Lips} for $i\in\{2,\dots,n-1\}$ and then adding \eqref{eq34} gives
\begin{equation}\label{*}
\begin{aligned}
    \sum_{i=1}^{n-1}L^2\| x_{i}-\bar{x}_{i}\|^2&\geq\|{(z_{1}-x_{1})-(\Bar{z}_{1}-\Bar{x}_{1})}\|^2+\sum_{i=2}^{n-1}(1-\frac{1}{\epsilon_{i}})\|(z_{i}-x_{i})-(\bar{z}_{i}-\bar{x}_{i})\|^2\\&\qquad+\sum_{i=2}^{n-1}(1-\epsilon_{i})\|(z_{i-1}-x_{i-1})-(\bar{z}_{i-1}-\bar{x}_{i-1})\|^2\\
    &\geq(2-\epsilon_{2})\|{(z_{1}-x_{1})-(\Bar{z}_{1}-\Bar{x}_{1})}\|^2+\sum_{i=2}^{n-2}\left(2-\frac{1}{\epsilon_{i}}-\epsilon_{i+1}\right)\|(z_{i}-x_{i})-(\bar{z}_{i}-\bar{x}_{i})\|^2\\
    &\qquad+\left(1-\frac{1}{\epsilon_{n-1}}\right)\|(z_{n-1}-x_{n-1})-(\bar{z}_{n-1}-\bar{x}_{n-1})\|^2.
\end{aligned}
\end{equation}
Now fix $\epsilon_{2}\in(1,2)$. We claim that we can choose constants $\epsilon_3,\dots,\epsilon_{n-1}\in(1,2)$ such that 
\begin{equation}\label{min of epsilon'}
    \epsilon':=\min_{i\in\{2,\dots,n-2\}}\left\{(2-\epsilon_{2}),\left(2-\frac{1}{\epsilon_{i}}-\epsilon_{i+1}\right),\left(1-\frac{1}{\epsilon_{n-1}}\right)\right\}>0.
\end{equation}
Indeed, first note that $2-\epsilon_2>0$ by assumption. Next suppose $\epsilon_i\in(1,2)$ for some $i\in\{2,\dots,n-2\}$. Since $1<(2-\frac{1}{\epsilon_i})<2$, we deduce that
$$\epsilon_{i+1}:=\sqrt{2-\frac{1}{\epsilon_{i}}}\in(1,2) \implies  \epsilon_{i+1} < \epsilon_{i+1}^2 = 2-\frac{1}{\epsilon_{i}} \implies 2-\frac{1}{\epsilon_{i}} - \epsilon_{i+1}>0. $$
Finally, by construction $\epsilon_{n-1}\in(1,2)$ and so $1-\frac{1}{\epsilon_{n-1}}>0$.

Now, combining \eqref{min of epsilon'} and \eqref{*} followed by applying \eqref{inequality}, we deduce that
\begin{equation}\label{simplify for epsilon*}
    \begin{aligned}
        L^2\sum_{i=1}^{n-1}\|x_{i}-\Bar{x}_{i}\|^2
        &\geq 
        \epsilon'\sum_{i=1}^{n-1}\|(z_{i}-x_{i})-(\bar{z}_{i}-\bar{x}_{i})\|^2\\
        &= \epsilon'\sum_{i=1}^{n-1}\left(\|z_{i}-\bar{z}_i\|^2+\|x_{i}-\bar{x}_{i}\|^2-2\langle z_i-\bar{z}_i,x_i-\bar{x}_i\rangle \right)\\
        &\geq \epsilon'\sum_{i=1}^{n-1}\left(\|z_{i}-\bar{z}_i\|^2+\|x_{i}-\bar{x}_{i}\|^2-\frac{\sqrt{\epsilon'}}{\sqrt{\epsilon'}+L}\|z_i-\bar{z}_i\|^2-\frac{\sqrt{\epsilon'}+L}{\sqrt{\epsilon'}}\|x_i-\bar{x}_i\|^2 \right)\\        
        &= \frac{\epsilon'L}{\sqrt{\epsilon'}+L}\|\mathbf{z}-\mathbf{\Bar{z}}\|^2-\sqrt{\epsilon'}L\sum_{i=1}^{n-1}\|x_{i}-\Bar{x}_{i}\|^2.
    \end{aligned}
\end{equation}
Rearranging this expression gives
\begin{equation}\label{lipschitz for n operator}
    \sum_{i=1}^{n-1}\|x_{i}-\Bar{x}_{i}\|^2\geq\frac{1}{\left(1+\frac{1}{\sqrt{\epsilon'}}L\right)^2}\|\mathbf{z}-\bar{\mathbf{z}}\|^2,
\end{equation}
which implies \eqref{lipschitz for n*}. This completes the proof.
\end{proof}
\begin{lemma}\label{lemma for contraction factor}
Let $n\geq 2$ and $\gamma\in(0,1)$. Suppose that one of the following holds:
\begin{enumerate}[(a)]
    \item $A_{1},\dots,A_{n-1}: \mathcal{H}\rightarrow \mathcal{H}$ are maximally monotone and $L$-Lipschitz, and $A_{n}\colon \mathcal{H}\setto \mathcal{H}$ is maximally $\mu$-strongly monotone. 
    \item $A_{1},\dots,A_{n-1}: \mathcal{H}\rightarrow \mathcal{H}$ are maximally $\mu$-strongly monotone and $L$-Lipschitz, and $A_{n}\colon \mathcal{H}\setto \mathcal{H}$ is maximally monotone.
\end{enumerate}
Then $T_{\rm MT}$ is a contraction.
\begin{proof}
    For convenience, denote $\mathbf{z}^+:= T_{\rm MT}(\mathbf{z})$ and $\bar{\mathbf{z}}^+:= T_{\rm MT}(\bar{\mathbf{z}})$. Let $\textbf{x}=(x_{1},\dots,x_{n})\in \mathcal{H}^n$ be given by \eqref{eq: def of x} and $\Bar{\textbf{x}}=(\Bar{x}_{1},\dots,\bar{x}_{n})\in \mathcal{H}^n$ be given analogously. 

(a):~Since $A_{1},\dots,A_{n-1}$ are maximally monotone and $A_{n}$ is maximally $\mu$-strongly monotone, Lemma~\ref{lemma 3.1} implies 
\begin{equation}\label{correct version for n}
    \| \mathbf{z}^+ - \Bar{\mathbf{z}}^+\|^2+\gamma(1-\gamma)\sum_{i=1}^{n-1}\|({x}_{i}-{x}_{i+1})-(\Bar{x}_{i}-\Bar{{x}}_{i+1})\|^2\leq\| \mathbf{z}-\bar{\mathbf{z}}\|^2-2\gamma\mu\|x_{n}-\bar{x}_{n}\|^2.
\end{equation}
        For $i\in\{1,\dots,n-1\}$ and any $\alpha_{i}>0$, applying \eqref{inequality} gives
\begin{equation}\label{new 33}
\begin{aligned}
        \|(x_{i}-x_{i+1})-(\Bar{x}_{i}-\Bar{x}_{i+1})\|^2&\geq \|x_{i+1}-\Bar{x}_{i+1}\|^2+\|x_{i}-\Bar{x}_{i}\|^2-2\langle x_{i}-\bar{x}_{i},x_{i+1}-\bar{x}_{i+1}\rangle\\
        &\geq (1-\alpha_{i})\|x_{i+1}-\Bar{x}_{i+1}\|^2+(1-\frac{1}{\alpha_{i}})\|x_{i}-\Bar{x}_{i}\|^2.
\end{aligned}
\end{equation}
By combining \eqref{correct version for n} and \eqref{new 33}, we obtain
\begin{multline}\label{new eq 33}
        \| \mathbf{z}^+ - \Bar{\mathbf{z}}^+\|^2+\gamma(1-\gamma)\left[\left(1-\frac{1}{\alpha_{1}}\right)\|x_{1}-\bar{x}_{1}\|^2+\sum_{i=2}^{n-1}\left(2-\frac{1}{\alpha_{i}}-\alpha_{i-1}\right)\|x_{i}-\Bar{x}_{i}\|^2\right]\\+[2\gamma\mu+\gamma(1-\gamma)(1-\alpha_{n-1})]\|x_{n}-\bar{x}_{n}\|^2\leq\| \mathbf{z}-\bar{\mathbf{z}}\|^2.
\end{multline}
We claim that we can choose constants $\alpha_{1},\dots,\alpha_{n-1}$ such that
\begin{equation}\label{p'}
    \alpha':=\min_{i\in\{2,\dots,n-1\}}\left\{\left(1-\frac{1}{\alpha_{1}}\right),\left(2-\frac{1}{\alpha_{i}}-\alpha_{i-1}\right)\right\}>0.
\end{equation}
Set $\alpha_{n-1}:=1+\frac{2\mu}{(1-\gamma)}>1$ and note that $2-\frac{1}{\alpha_{n-1}}>1$. Suppose $\alpha_i>1$ for some $i\in\{n-1,\dots,2\}$. Since $2-\frac{1}{\alpha_i}>1$, we deduce that
$$\alpha_{i-1}:=\sqrt{2-\frac{1}{\alpha_{i}}}>1\implies  \alpha_{i-1} < \alpha_{i-1}^2 = 2-\frac{1}{\alpha_{i}} \implies 2-\frac{1}{\alpha_{i}} - \alpha_{i-1}>0.$$
Finally, by construction $\alpha_{1}>1$ and so $1-\frac{1}{\alpha_{1}}>0$.

Now, using \eqref{p'} in \eqref{new eq 33} implies
\begin{equation} \label{eq:33}
  \|\mathbf{z}^+ - \Bar{\mathbf{z}}^+\|^2\leq\| \mathbf{z}-\bar{\mathbf{z}}\|^2-\gamma(1-\gamma)\alpha'\sum_{i=1}^{n-1}\|x_{i}-\Bar{x}_{i}\|^2.
\end{equation}
Since $A_{i}$ is maximally monotone and $L$-Lipschitz for $i\in\{1,\dots,n-1\}$, Lemma~\ref{lipschitz operators} implies there exists $\eta\in(0,1)$ such that
\begin{equation}\label{lipschitz for n}
    \sum_{i=1}^{n-1}\|x_{i}-\Bar{x}_{i}\|^2\geq\eta\|\mathbf{z}-\bar{\mathbf{z}}\|^2.
\end{equation}
Substituting \eqref{lipschitz for n} into \eqref{eq:33} and rearranging the equation we get,
\begin{equation} \label{eq:37}
  \|\mathbf{z}^+ - \Bar{\mathbf{z}}^+\|^2\leq\left[(1-\gamma(1-\gamma)\alpha'\eta\right]\|\mathbf{z}-\mathbf{\Bar{z}}\|^2.
\end{equation}
Therefore, $T_{\rm MT}$ is a $\beta$-contraction with $\beta=(1-\gamma(1-\gamma)\alpha'\eta)\in(0, 1)$. This completes the proof.

(b):~Since $A_{1},\dots,A_{n-1}$ are maximally $\mu$-strongly monotone and $A_{n}$ is maximally monotone, Lemma~\ref{lemma 3.1} implies 
\begin{equation}\label{correct version for n*}
    \| \mathbf{z}^+ - \Bar{\mathbf{z}}^+\|^2\leq\| \mathbf{z}-\bar{\mathbf{z}}\|^2-2\gamma\mu\sum_{i=1}^{n-1}\|x_{i}-\bar{x}_{i}\|^2.
\end{equation}
Since $A_{1},\dots,A_{n-1}$ are maximally monotone and $L$-Lipschitz, Lemma~\ref{lipschitz operators} implies there exists $\eta\in(0,1)$ such that
\begin{equation}\label{lipschitz}
    \sum_{i=1}^{n-1}\|x_{i}-\Bar{x}_{i}\|^2\geq\eta\|\mathbf{z}-\bar{\mathbf{z}}\|^2.
\end{equation}
Substituting \eqref{lipschitz} into \eqref{correct version for n*} gives
\begin{equation} \label{eq:37*}
  \|\mathbf{z}^+ - \Bar{\mathbf{z}}^+\|^2\leq\left(1-2\gamma\mu\eta\right)\|\mathbf{z}-\mathbf{\Bar{z}}\|^2.
\end{equation} 
Therefore, $T_{\rm MT}$ is a $\beta$-contraction with $\beta=(1-2\gamma\mu\eta)\in(0,1)$. This completes the proof.
\end{proof}
\end{lemma}
\begin{remark}
In the absence of appropriate strong monotonicity or Lipschitz continuity (such as in Lemma~\ref{lemma for contraction factor}), the operator $T_{\rm MT}$ need not be a contraction. In what follows, we provide two such examples of the monotone inclusion problem \eqref{eq:1n} with $n=3$. The first example shows that, without strong monotonicity, $T_{MT}$ need not be a contraction even when all the operators are Lipschitz continuous. The second shows that, without Lipschitz continuity, $T_{MT}$ need not be a contraction even when all the operators are strongly monotone. In both cases, we show that $\Fix T_{\rm MT}$ contains more than one point which implies $T_{\rm MT}$ is not a contraction.
\begin{enumerate}[(a)]
    \item Consider the operators defined on $\mathbb{R}$ given by
\begin{equation*}
    A_{1}=0,\quad A_{2}=0,\quad A_{3}=0.
\end{equation*}
    Any $x^*\in\mathbb{R}$ is a solution of the inclusion, and the operators $A_{1}, A_{2}, A_{3}$ are monotone (but not strongly monotone) and $L$-Lipschitz for all $L>0$. The resolvents  are given by
    $$J_{A_{1}}=\Id,\quad J_{A_{2}}=\Id,\quad J_{A_{3}}=\Id.$$    
Let $\mathbf{z}=\binom{z_{1}}{z_{2}}\in\mathbb{R}\binom{1}{1}$. Then \eqref{eq: fixed point operator} and \eqref{eq: def of x} become
\begin{equation*}
\left\{\begin{aligned} 
x_{1} &=J_{A_{1}}(z_{1}) = z_1\\
x_{2} &=J_{A_{2}}(z_{2}+x_{1}-z_{1}) = J_{A_2}(z_2) = z_{2}\\ 
x_{3} &= J_{A_{3}}(x_1+x_2-z_2) = J_{A_3}(z_{1}) = z_{1}
\end{aligned}\right.
\implies \quad
  T_{\rm MT}(\mathbf{z})  = \mathbf{z}+\gamma\begin{pmatrix}
     z_{2}-z_{1} \\ z_{1}-z_{2}\\
 \end{pmatrix} =\mathbf{z},
\end{equation*}
and thus we conclude that $\mathbb{R}\binom{1}{1}\subseteq\Fix T_{\rm MT}$. Since $T_{\rm MT}$ has more than one fixed point,  we conclude that it is not a contraction. 
\item Let $\mu>0$ and consider the operators defined on $\mathbb{R}$ given by
$$ A_1 = \mu \Id + N_{\mathbb{R}_-},\quad A_2 = \mu \Id + N_{\mathbb{R}_+},\quad A_3 = \mu \Id + N_{\{0\}}. $$
Note that $x^*=0$ is the unique solution of the inclusion,  and the operators $A_1,A_2,A_3$ are $\mu$-strongly monotone (but not Lipschitz continuous). The resolvent \cite[Example 23.4]{bauschke2011convex} of these operators are given by   
$$ J_{A_1} = P_{N_{\mathbb{R}_-}}\circ \frac{1}{1+\mu}\Id,\quad J_{A_2} = P_{N_{\mathbb{R}_+}}\circ \frac{1}{1+\mu}\Id,\quad J_{A_3} = P_{N_{\{0\}}}\circ \frac{1}{1+\mu}\Id,$$
where $P_{N_{\mathbb{R}_-}}, P_{N_{\mathbb{R}_+}}, P_{N_{\{0\}}}$ denote the projection onto $N_{\mathbb{R}_-}, N_{\mathbb{R}_+}$ and $N_{\{0\}}$ respectively.

Let $\mathbf{z}=\binom{z_1}{z_2}\in\mathbb{R}_-\times\{0\}$. Then \eqref{eq: fixed point operator} and \eqref{eq: def of x} become
\begin{equation*}
\left\{\begin{aligned} 
x_{1} &=J_{A_{1}}(z_{1}) = P_{\mathbb{R}_+}\left(\frac{1}{1+\mu}z_1\right)=0 \\
x_{2} &=J_{A_{2}}(z_{2}+x_{1}-z_{1}) = P_{\mathbb{R}_-}\left(-\frac{1}{1+\mu}z_1\right) = 0\\ 
x_{3} &= J_{A_{3}}(x_1+x_2-z_2) = P_{\{0\}}\left(\frac{1}{1+\mu}\cdot 0\right)=0
\end{aligned}\right.
\implies T_{\rm MT}(\mathbf{z})  = \mathbf{z} + \gamma\begin{pmatrix}
     0\\ 0\\ 
 \end{pmatrix} = \mathbf{z},
\end{equation*}
and thus we conclude that $\mathbb{R}_-\times\{0\}\subseteq\Fix T_{\rm MT}$. Since $T_{\rm MT}$ has more than one fixed point,  we conclude that it is not a contraction.
\end{enumerate}
\end{remark}
We are now ready to state the main result of this section regarding linear convergence of the algorithm presented in \eqref{eq:1} and \eqref{eq:2}.
\begin{theorem}\label{theorem for linear convergence}
    Let $n\geq2$ and $\gamma\in(0,1)$. Suppose that one of the following holds:
    \begin{enumerate}[(a)]
        \item $A_{1},\dots,A_{n-1}:\mathcal{H}\rightarrow\mathcal{H}$ are maximally monotone and $L$-Lipschitz, and $A_{n}:\mathcal{H}\setto\mathcal{H}$ is maximally $\mu$-strongly monotone.
        \item $A_{1},\dots,A_{n-1}:\mathcal{H}\rightarrow\mathcal{H}$ are maximally $\mu$-strongly monotone and $L$-Lipschitz, and $A_{n}:\mathcal{H}\setto\mathcal{H}$ is maximally monotone.
    \end{enumerate} 
    Given $\mathbf{z}^0\in \mathcal{H}^{n-1}$, let $(\mathbf{z}^k)_{k\in\mathbb{N}}$ and $(\mathbf{x}^k)_{k\in\mathbb{N}}$ be the sequences given by~\eqref{eq:1} and \eqref{eq:2}. Then the following assertions hold:
    \begin{enumerate}[(i)]
        \item $(\mathbf{z}^k)_{k\in\mathbb{N}}$ converges $R$-linearly to the unique fixed point $\mathbf{z}^*\in\Fix T_{\rm MT}$. 
        \item $(\mathbf{x}^k)_{k\in\mathbb{N}}$ converges $R$-linearly to a point $(x^*,\dots, x^*)\in \mathcal{H}^n$ where $x^*$ is the unique element of $\zer(\sum_{i=1}^{n}A_{i})$. 
    \end{enumerate}
\end{theorem}
\begin{proof}
(i):~Since the operators $A_{1},\dots,A_{n}$ satisfy either (a) or (b), Lemma~\ref{lemma for contraction factor} implies that the operator $T_{\rm MT}$ is a $\beta$-contraction for some $\beta\in(0,1)$. Thus, according to the Banach fixed-point theorem (Theorem~\ref{Banach Theorem}), $T_{\rm MT}$ has a unique fixed point, say $\{\mathbf{z}^*\}=\Fix T_{\rm MT}$ and 
\begin{equation}\label{linear convergence}
        \|\mathbf{z}^k-\mathbf{z}^*\|\leq c\beta^k\quad\forall k\in\mathbb{N},
    \end{equation}
where $c:=\|\mathbf{z}^{0}-\mathbf{z}^*\|$. In particular, this shows that the sequence $(\mathbf{z}^k)$ converges $R$-linearly to $\mathbf{z}^*\in\Fix T_{\rm MT}$.

(ii):~Since $\mathbf{z}^*\in\Fix T_{\rm MT}$, Lemma~\ref{new lemma} implies $x^*=J_{A_{1}}({z}^*_{1})\in \zer(\sum_{i=1}^n A_{i})$ and that \eqref{eq: def of x^*} holds. Furthermore, since $\sum_{i=1}^nA_{i}$ is strongly monotone, $\{x^*\}=\zer(\sum_{i=1}^n A_{i})$. Recall that $J_{A_{i}}$ is nonexpansive for $i\in\{1,\dots,n\}$ (Proposition~\ref{nonexpansiveness}). Thus, together with \eqref{linear convergence}, we deduce that
\begin{equation}\label{eq:r-lin x1}
\|x^k_{1}-x^*\|=\|J_{A_{1}}(z^k_{1})-J_{A_{1}}(z^*_{1})\|\leq\|z_{1}^k-z^*_{1}\| \leq \|\mathbf{z}^k-\mathbf{z}^*\| \leq c\beta^k.
\end{equation}
For $i\in\{2,\dots,n-1\}$, with the help of \eqref{linear convergence}, \eqref{eq:r-lin x1} and \eqref{eq:r-lin xi} (inductively), we get
\begin{equation}\label{eq:r-lin xi}\begin{aligned}
\|x^k_{i}-x^*\|
&=\|J_{A_{i}}(z^k_{i}+x^k_{i-1}-z^k_{i-1})-J_{A_{i}}(z^*_{i}+x^*-z^*_{i-1})\| \\
 &\leq \|z_{i}^k-z^*_{i}\|+\|z^k_{i-1}-z^*_{i-1}\|+\|x^k_{i-1}-x^*\|\\
 &\leq \|\mathbf{z}^k-\mathbf{z}^*\| + \|\mathbf{z}^k-\mathbf{z}^*\| + \bigl(2(i-1)-1\bigr) c\beta^k  \leq (2i-1)c\beta^k.
\end{aligned}
\end{equation}
Finally, using \eqref{linear convergence}, \eqref{eq:r-lin x1} and \eqref{eq:r-lin xi}, we deduce
\begin{align*}
\|x^k_{n}-x^*\| 
&=\|J_{A_{n}}(x^k_{1}+x^k_{n-1}-z^k_{n-1})-J_{A_{n}}(x^*+x^*-z^*_{n-1})\|\\
&\leq\|z^k_{n-1}-z^*_{n-1}\|+\|x^k_{1}-x^*\|+\|x^k_{n-1}-x^*\|\\
 &\leq \|\mathbf{z}^k-\mathbf{z}^*\| + c\beta^k + \bigl(2(n-1)-1\bigr) c\beta^k 
\leq (2n-1)c\beta^k.
\end{align*}
This shows that $(\mathbf{x}^k)$ converges $R$-linearly to $(x^*,\dots,x^*)\in\mathcal{H}^n$. 
\end{proof}

\section{Linear Convergence of a Primal-Dual Algorithm} \label{s: section 4}
In this section, we apply the results of Section~\ref{s:resolvent splitting} to the setting of Example~\ref{example 1.1} to derive linear convergence of a primal-dual algorithm based on the resolvent splitting algorithm presented in \eqref{eq:1} and \eqref{eq:2}. To this end, let $\mathcal{H}_{1}$ and $\mathcal{H}_{2}$ denote real Hilbert spaces and consider the minimization problem with $n\geq 2$ given by
\begin{equation} \label{convex optimization}
    \min_{u\in\mathcal{H}_{1}}\quad \sum_{i=2}^{n}f_{i}(u)+(g_{2}\Box\cdot\cdot\cdot\Box g_{n})(Cu)
\end{equation}
under the hypotheses in stated below in Assumption~\ref{assumption}. Using Proposition~\ref{remark infimal convolution}, the Fenchel dual of \eqref{convex optimization} can be expressed as 
\begin{equation} \label{convex optimization dual}
    \min_{v\in\mathcal{H}_{2}}\quad (f^*_{2}\Box\cdots\Box f^*_{n})(C^*v)+\sum_{i=2}^{n}g^*_{i}(-v).
\end{equation}
\begin{assumption}\label{assumption}
    \begin{enumerate}[(i)]
    \item $C:\mathcal{H}_{1}\rightarrow\mathcal{H}_{2}$ is bounded and linear. 
    \item $f_{i}:\mathcal{H}_{1}\rightarrow\mathbb{R}$ is convex and differentiable with $\alpha$-Lipschitz continuous gradient for $i=2,\dots,n-1$.
    \item $f_{n}:\mathcal{H}_{1}\rightarrow(-\infty,+\infty]$ is proper, closed and $\sigma$-strongly convex.
    \item $g_{i}:\mathcal{H}_{2}\rightarrow(-\infty,+\infty]$ is proper, closed and $\tau$-strongly convex for $i=2,\dots,n-1$.
    \item $g_{n}:\mathcal{H}_{2}\rightarrow\mathbb{R}$ is convex and differentiable with $\beta$-Lipschitz continuous gradient.
    \end{enumerate}
\end{assumption}

Denote $f(u):=\sum_{i=2}^{n-1}f_{i}(u)+f_{n}(u)$ and $g(v):=(g_{2}\Box\cdots\Box g_{n})(v)$. By Assumptions~\ref{assumption}\hyperref[assumption]{(ii)}-\hyperref[assumption]{(iii)}, we have $\dom f=\cap_{i=2}^n\dom f_i=\dom f_n\neq\emptyset$ and, by Assumptions~\ref{assumption}\hyperref[assumption]{(iv)}-\hyperref[assumption]{(v)} combined with Proposition~\ref{prop for infimal convolution},  we have that $g$ is proper lsc convex with $\dom g=\mathcal{H}_2$. We therefore have that 
$$\sri(\dom g-C(\dom f)) =\sri(\mathcal{H}_2-C(\dom f_n))=\sri\mathcal{H}_2= \mathcal{H}_2\ni 0.$$
Thus the first order optimality condition for the primal problem~\eqref{convex optimization} followed by the subdifferential sum rule (Theorem~\ref{sum rule of subdifferential for two functions}) gives
\begin{equation}\label{eq:fop}
    0\in\partial(f+g\circ C)(u)=\partial f(u)+C^*\partial g(Cu).
\end{equation}
The inclusion \eqref{eq:fop} holds precisely when there exists $v\in\partial g(Cu)$ such that $0\in\partial f(u)+C^*v$. Since $(\partial g)^{-1}=\partial g^*$ by \cite[p.~216]{rockafellar1970maximal}, $v\in \partial g(Cu)\iff Cu\in (\partial g)^{-1}(v)=\partial g^*(v)$ is equivalent to $0\in\partial g^*(v)-Cu$. Combining these two inclusions into a single system gives the monotone inclusion
\begin{equation}\label{inclusion for f and g}
    \begin{pmatrix}
        0\\0
    \end{pmatrix}\in\begin{pmatrix}
        0&C^*\\-C&0
    \end{pmatrix}\begin{pmatrix}
        u\\v
    \end{pmatrix}+\begin{pmatrix}
        \partial f(u)\\\partial g^*(v)
    \end{pmatrix}.
\end{equation}
Since $f_{2},\dots,f_{n-1}$ are differentiable with full domain, we have $\partial f=\sum_{i=2}^{n-1}\nabla f_i+\partial f_n$ by Theorem~\ref{sum rule of subdifferential for two functions}. Similarly, $g^*_{2},\dots,g^*_{n-1}$ are differentiable by \cite[Theorem 2.1]{bauschke2009baillon}, thus Theorems~\ref{remark infimal convolution}~\&~\ref{sum rule of subdifferential for two functions} yield $\partial g^*=(g_1\Box\cdots\Box g_n)^*=\sum_{i=2}^{n-1}\nabla g_i^*+\partial g_n^*$. Substituting these two identities into \eqref{inclusion for f and g} gives
\begin{equation}\label{n monotone}
    \begin{pmatrix}
        0\\0
    \end{pmatrix}\in\begin{pmatrix}
        0&C^*\\-C&0
    \end{pmatrix}\begin{pmatrix}
        u\\v
    \end{pmatrix}+\sum_{i=2}^{n-1}\begin{pmatrix}
        \nabla f_{i}(u)\\\nabla g_{i}^*(v)\end{pmatrix}+\begin{pmatrix}
        \partial f_{n}(u)\\\partial g_{n}^*(v)
    \end{pmatrix}.
\end{equation}
In other words, \eqref{n monotone} is a monotone inclusion with $n$ operators in the form of \eqref{eq:1n} with
\begin{equation} \label{monotone operators*}
    x=(u,v),\quad \mathcal{H}=\mathcal{H}_{1}\times\mathcal{H}_{2},\quad  A_{1}=\begin{pmatrix}
        0&C^*\\-C&0
    \end{pmatrix}, \quad A_{i}=\begin{pmatrix}
        \nabla f_{i}\\\nabla g_{i}^*\end{pmatrix},  \quad A_{n}=\begin{pmatrix}
        \partial f_{n}\\\partial g_{n}^*
    \end{pmatrix}.
\end{equation}

The following lemma summarises properties of the operators defined in \eqref{monotone operators*}.
\begin{lemma}\label{lemma primal dual splitting}
Let $A_1,\dots,A_n$ be the operators defined in \eqref{monotone operators*}, and suppose Assumption~\ref{assumption} holds. Then the following assertions hold.
    \begin{enumerate}[(i)]
        \item $A_{1}$ is maximally monotone and $\|C\|$-Lipschitz.
        \item $A_{i}$ is maximally monotone and $L$-Lipschitz with $L=\max\{\alpha,\frac{1}{\tau}\} \text{ for }i=2,\dots,n-1$.
        \item $A_{n}$ is maximally $\mu$-strongly monotone with $\mu=\min\{\sigma, \frac{1}{\beta}\}$.
    \end{enumerate}
\end{lemma}
\begin{proof}
(i):~Since $A_{1}$ is skew-symmetric, \cite[Example~20.30]{bauschke2011convex} implies that $A_{1}$ is maximally monotone. For all $(u,v),(u',v')\in\mathcal{H}_{1}\times\mathcal{H}_{2}$, we have 
\begin{align*}
\|A_1(u,v)-A_1(u',v')\|^2 
&= \|C^*(v-v')\|^2+\|C(u-u')\|^2 \\
&\leq \|C^*\|^2\|v-v'\|^2+\|C\|^2\|u-u'\|=\|C\|^2\|(u,v)-(u',v')\|^2,
\end{align*}
where the last equality uses the identity $\|C^*\|=\|C\|$. Hence, $A_{1}$ is $\|C\|$-Lipschitz.
(ii):~By Assumption~\ref{assumption}\hyperref[assumption]{(iv)} and \cite[Theorem~2.1]{bauschke2009baillon},  $g^*_{i}$ is differentiable with $\frac{1}{\tau}$-Lipschitz gradients for $i=2,\dots,n-1$. Therefore, $A_{i}$ is maximally monotone from \cite[Corollary~20.25]{bauschke2011convex} and $L$-Lipschitz continuous with $L=\max\{\alpha,\frac{1}{\tau}\}$.
(iii):~By Assumption~\ref{assumption}\hyperref[assumption]{(iii)} and \cite[Theorem~20.40~\&~Example~22.3(iv)]{bauschke2011convex}, $\partial f_{n}$ is maximally $\sigma$-strongly monotone.  By Assumption~\ref{assumption}\hyperref[assumption]{(v)} and \cite[Theorem~2.1]{bauschke2009baillon}, $\partial g^*_{n}$ is maximally $\frac{1}{\beta}$-strongly monotone. Therefore, $A_{n}$ is maximally $\mu$-strongly monotone with $\mu=\min\{\sigma, \frac{1}{\beta}\}$. 
\end{proof}

Let $\gamma\in(0,1)$. When applied to \eqref{monotone operators*}, the fixed point operator from Section~\ref{s:resolvent splitting} (given in \eqref{eq: fixed point operator} and \eqref{eq: def of x}) becomes
\begin{equation} \label{new algorithm def}
\begin{aligned}
    T_{\rm PD}&\begin{pmatrix}\mathbf{p},\mathbf{q}\end{pmatrix}=\begin{pmatrix}\mathbf{p},\mathbf{q}\end{pmatrix}+\gamma
\begin{pmatrix}
\begin{pmatrix}
   u_{2},v_{2} 
\end{pmatrix}-\begin{pmatrix}
   u_{1},v_{1} 
\end{pmatrix}\\\begin{pmatrix}
   u_{3},v_{3} 
\end{pmatrix}-\begin{pmatrix}
   u_{2},v_{2} 
\end{pmatrix}\\\vdots\\\begin{pmatrix}
   u_{n},v_{n} 
\end{pmatrix}-\begin{pmatrix}
   u_{n-1},v_{n-1}
\end{pmatrix}
\end{pmatrix},
\end{aligned}
\end{equation}
where $\mathbf{u}=(u_{1},\dots, u_{n})\in{\mathcal{H}}_{1}^n$ and $\mathbf{v}=(v_{1},\dots, v_{n})\in\mathcal{H}_{2}^n$ are given by
\begin{equation} \label{new x def}
\left\{\begin{aligned} 
\begin{pmatrix}
    u_{1}\\v_{1}
\end{pmatrix}&=J_{A_{1}}\begin{pmatrix}
        p_{1}\\q_{1}
    \end{pmatrix}\\
    \begin{pmatrix}
    u_{i}\\v_{i}
\end{pmatrix}&=J_{A_{i}}\begin{pmatrix}
    p_{i}+u_{i-1}-p_{i-1}\\q_{i}+v_{i-1}-q_{i-1}
\end{pmatrix}&\forall i\in \{2,\dots,n-1\}\\
\begin{pmatrix}
    u_{n}\\v_{n}
\end{pmatrix}&=J_{A_{n}}\begin{pmatrix}
    u_{1}+u_{n-1}-p_{n-1}\\v_{1}+v_{n-1}-q_{n-1}
\end{pmatrix}.
\end{aligned}\right.
\end{equation}

We now look at how to compute the resolvents in \eqref{new x def}. Using the standard formula for the inverse of a $2\times 2$ block matrix  (see, for example, \cite[Section 2.1]{o2014primal}), the resolvent of $A_{1}$ can be expressed as
\begin{align*}
    J_{{A}_{1}}=\begin{pmatrix}
        \Id&C^*\\-C&\Id
    \end{pmatrix}^{-1} = \begin{pmatrix}
        0&0\\0&\Id
    \end{pmatrix}+\begin{pmatrix}
        \Id\\C
    \end{pmatrix}(\Id+C^*C)^{-1}\begin{pmatrix}
        \Id\\-C
    \end{pmatrix}^*.
\end{align*}
For $i\in\{2,\dots,n\}$, by using the Moreau decomposition for $g_i$ (Theorem~\ref{Moreau decomposition}), we deduce the resolvent of $A_{i}$
$$ J_{A_i} = \binom{\prox_{f_i}}{\prox_{g^*_i}} = \binom{\prox_{f_i}}{\Id-\prox_{g_i}}. $$
The algorithm obtained by putting this altogether is described in Algorithm~\ref{alg:PD}, and its convergence is analyzed in Corollary~\ref{corollary primal dual linear convergence}.

\SetKwComment{Comment}{/* }{ */}
\RestyleAlgo{ruled}
\begin{algorithm}[!htb]
\caption{A primal-dual algorithm for solving \eqref{convex optimization} and \eqref{convex optimization dual}.\label{alg:PD}}
\KwIn{Choose $\mathbf{p}^0=(p^0_{1},\dots,p^0_{n-1})\in\mathcal{H}^{n-1}_{1}$, $\mathbf{q}^0=(q^0_{1},\dots,q^0_{n-1})\in\mathcal{H}^{n-1}_{2}$ and $\gamma\in(0,1)$.}
\For{$k=1,2\dots$}{Compute $\mathbf{u}^k=(u^k_{1},\dots, u^k_{n})\in{\mathcal{H}}_{1}^n$ and $\mathbf{v}^k=(v^k_{1},\dots, v^k_{n})\in\mathcal{H}_{2}^n$ according to
\begin{equation} \label{new x}
\left\{\begin{aligned} 
u^k_{1}&=(\Id+C^*C)^{-1}(p_1^k-C^*q_1^k)\\
v^k_{1}&=q_1^k+Cu_1^k\\
    u^k_{i}&=\prox_{f_{i}}(p^k_{i}+u^k_{i-1}-p^k_{i-1})&\forall i\in \{2,\dots,n-1\}\\
    v^k_{i}&=(\Id-\prox_{g_{i}})(q^k_{i}+v^k_{i-1}-q^k_{i-1})&\forall i\in \{2,\dots,n-1\}\\
    u^k_{n}&=
    \prox_{{f}_{n}}(u^k_{1}+u^k_{n-1}-p^k_{n-1})\\
    v^k_{n}&=(\Id-\prox_{g_{n}})(v^k_{1}+v^k_{n-1}-q^k_{n-1}).
\end{aligned}\right.
\end{equation}
Update $\mathbf{p}^{k}=(p^k_{1},\dots, p^k_{n-1})\in \mathcal{H}_{1}^{n-1}$ and $\mathbf{q}^{k}=(q^k_{1},\dots, q^k_{n-1})\in \mathcal{H}_{2}^{n-1}$ according to
\begin{equation} \label{new algorithm}
    \mathbf{p}^{k+1}=\mathbf{p}^k+\gamma
\begin{pmatrix}
   u^{k}_{2} -
   u^{k}_{1}\\
   u^{k}_{3}-
   u^{k}_{2}\\\vdots\\
   u^{k}_{n}-
   u^{k}_{n-1}
\end{pmatrix},
\qquad \mathbf{q}^{k+1}=\mathbf{q}^k+\gamma\begin{pmatrix}
   v^{k}_{2} -
   v^{k}_{1}\\
   v^{k}_{3}-
   v^{k}_{2}\\\vdots\\
   v^{k}_{n}-
   v^{k}_{n-1}
\end{pmatrix}.
\end{equation}}
\end{algorithm}

Let $u^*$ be the (unique) solution of the primal problem~\eqref{convex optimization}, and $v^*$ be the unique solution of the dual problem~\eqref{convex optimization dual}. Then the \emph{primal-dual gap} given by
\begin{equation}\label{primal-dual gap}
    G(u,v)=\left( \sum_{i=2}^{n}f_{i}(u)+\langle Cu,v^*\rangle - \sum_{i=2}^{n}g_i^*(v^*)\right)-\left( \sum_{i=2}^{n}f_{i}(u^*)+\langle Cu^*,v\rangle - \sum_{i=2}^{n}g_i^*(v)\right).
\end{equation}
Now we are ready for the main result of this section.
\begin{corollary}\label{corollary primal dual linear convergence}
    Suppose Assumption~\eqref{assumption} holds. Given $(\mathbf{p}^0,\mathbf{q}^0)\in\mathcal{H}_{1}^{n-1}\times\mathcal{H}_{2}^{n-1}$, let $(\mathbf{p}^k,\mathbf{q}^k)\in\mathcal{H}_{1}^{n-1}\times\mathcal{H}_{2}^{n-1}$ and $(\mathbf{u}^k,\mathbf{v}^k)\in\mathcal{H}_{1}^n\times\mathcal{H}^n_{2}$ be the sequences given by \eqref{new x} and \eqref{new algorithm} in Algorithm~\ref{alg:PD}. Then the following assertions hold:
    \begin{enumerate}[(i)]
        \item $(\mathbf{p}^k,\mathbf{q}^k)_{k\in\mathbb{N}}$ converges $R$-linearly to the unique fixed point $(\mathbf{p^*},\mathbf{q^*})\in\Fix T_{\rm PD}$. 
        \item $(\mathbf{u}^k)_{k\in\mathbb{N}}$ converges $R$-linearly to   $(u^*,\dots, u^*)\in\mathcal{H}^n_{1}$  where $u^*$ is the unique solution of the primal problem~\eqref{convex optimization}.
        \item $(\mathbf{v}^k)_{k\in\mathbb{N}}$ converges $R$-linearly to $(v^*,\dots, v^*)\in\mathcal{H}^n_{2}$ where $v^*$ is the unique solution of the dual problem~\eqref{convex optimization dual}.
        \item $G(u^k_n,v^k_n)$ converges $R$-linearly to zero.
    \end{enumerate}
\end{corollary}
\begin{proof}
(i)-(iii):~By Lemma~\ref{lemma primal dual splitting}, the  operators $A_1,\dots,A_n$ satisfy the assumptions of Theorem~\ref{theorem for linear convergence}\hyperref[theorem for linear convergence]{(a)}. Thus, by applying  Theorem~\ref{theorem for linear convergence}, we have that (i) holds, that $(\mathbf{u}^k,\mathbf{v}^k)$ converges $R$-linearly to a point $((u^*,\dots,u^*),(v^*,\dots,v^*))$ such that $(u^*,v^*)\in\zer(\sum_{i=1}^nA_i)$, and that \eqref{eq: def of x^*} holds. The latter together with  \cite[Theorem~19.1]{bauschke2011convex} implies that $u^*$ solves the primal problem~\eqref{convex optimization} and $v^*$ solves the dual problem~\eqref{convex optimization dual}. Since both primal and dual problems are strongly convex, their minimizers are unique. (iv):~To establish the result, we analyze each term in $G$ separately. First note that
\begin{equation}\label{inequality for inner}
\begin{aligned}
    \left|\langle Cu_n^k,v^*\rangle- \langle Cu^*,v_n^k\rangle\right|
    &= 
\left| \langle Cu^k_n,v^*\rangle - \langle Cu^*,v^*\rangle + \langle Cu^*,v^*\rangle -\langle Cu^*,v^k_n\rangle\right| \\
&= \left|\langle u^k_n-u^*,C^*v^*\rangle + \langle Cu^*,v^*-v^k_n\rangle \right| \\
&\leq \left|\langle u^k_n-u^*,C^*v^*\rangle \right|+\left|\langle Cu^*,v^*-v^k_n\rangle\right|\\
&\leq \frac{1}{2}\|u^k_n-u^*\|\|C^*v^*\|+\frac{1}{2}\|Cu^*\|\|v^k_n-v^*\|,
\end{aligned}
\end{equation}
from which follows that $(\langle Cu_n^k,v^*\rangle- \langle Cu^*,v_n^k\rangle)_{k\in\mathbb{N}}$ converges $R$-linearly to zero.

Next, since $f_{i}$ is convex and differentiable with $\alpha$-Lipschitz gradient for $i\in\{2,\dots,{n-1}\}$, we have
$$ \|\nabla f_i(u_n^k)\| \leq \|\nabla f_i(u_n^k)-\nabla f_i(u^*)\|+\|\nabla f_i(u^*)\| \leq \alpha \|u_n^k-u^*\|+ \|\nabla f_i(u^*)\|. $$
Since $(u^k_n)$ converge $R$-linearly to $u^*$ by (ii), it follows $(\nabla f_i(u_n^k))$ is bounded. By applying  Lemma~\ref{lemma for gap}, it follows that  $f_{i}(u^k_{n})$ converge $R$-linearly to $f_{i}(u^*)$. Since $g^*_{i}$ is convex and differentiable with $\alpha$-Lipschitz gradient for $i\in\{2,\dots,{n-1}\}$, an analogous argument shows that $g^*_{i}(v^k_{n})$ converge $R$-linearly $g^*_{i}(v^*)$.

Due to \eqref{new x}, we have $u^k_{1}+u^k_{n-1}-u^k_{n}-p^k_{n-1}\in\partial f_{n}(u^k_{n})$ which is a bounded sequence as the sum of convergent  sequences, and, due to \eqref{eq: def of x^*}, $u^*-p^*_{n-1}\in\partial f_{n}(u^*)$. Thus, by applying  Lemma~\ref{lemma for gap}, it follows that  $f_{n}(u^k_{n})$ converge $R$-linearly to $f_{n}(u^*)$. An analogous argument shows that $g^*_{n}(v^k_{n})$ converge $R$-linearly $g^*_{n}(v^*)$. As all terms in $G$ converge $R$-linearly, and so the result follows.
\end{proof}

\section{Numerical Experiment: Image denoising}\label{s: Experiment}
In this section, we give a numerical illustration of the Algorithm~\ref{alg:PD} applied to image denoising. All computations are run in Python 3.12.0 on a MacBook Pro machine equipped with 16GB memory and an Apple M1 Pro Chip.

Let $u$ denote an $M\times M$ gray scale image, where pixel values range from 0 to 1 (with 0 representing black and 1 representing white), represented by a vector in $\mathbb{R}^m$ with $m=M\times M$ which is to be recovered from a noisy observation $b\in\mathbb{R}^m$.  In our numerical experiments, the observed image $b$ is generated by adding Gaussian noise to the true image.

Let $\lambda_1, \lambda_2, \lambda_3, \lambda_4>0$. To recover $u$ from $b$, we consider the  minimization problem~\cite[Section 6.2]{aragon2021strengthened}
\begin{equation}\label{delburring problem}
    \min_{u\in\mathbb{R}^{m}}\frac{1}{2}\|u-b\|^2+\frac{\lambda_{1}}{2}\|u\|^2+\left((\lambda_{2}\|\cdot\|_{\iso}+\frac{\lambda_3}{2}\|\cdot\|^2) \Box (\frac{\lambda_4}{2}\|\cdot\|^2)\right)(Du),
\end{equation}
where $\|\binom{v^1}{v^2}\|_{\iso}:=\sum_{i=1}^m\sqrt{(v^1_i)^2+(v^2_i)^2}$ and $D\in\mathbb{R}^{2m\times m}$ denotes the \emph{discrete gradient} given by 
\begin{equation*}
    D=\begin{pmatrix}
    \Id\otimes D_{1}\\D_{1}\otimes\Id
\end{pmatrix},\qquad D_{1}=\begin{pmatrix}
        -1&1&0&\cdots&0&0\\
        0&-1&1&\cdots&0&0\\
        \vdots&\vdots&\vdots&\ddots&\vdots&\vdots\\
        0&0&0&\cdots&-1&1\\
        0&0&0&\cdots&0&0\\
    \end{pmatrix}\in\mathbb{R}^{M\times M}.
\end{equation*} 

Here $\otimes$ denotes the Kronecker product.
Under this notation, $u\mapsto \|Du\|_{\iso}$ is the \emph{discrete isotropic total variation (TV) norm} introduced in \cite{rudin1992nonlinear}. In~\eqref{delburring problem}, the first term $\frac{1}{2}\|u-b\|^2$ ensures data fidelity, and the term $\|Du\|_{\iso}+\frac{1}{2}\|Du\|^2$ ensures the reconstruction has small total variation which promotes ``sharp edges'' in the image. The terms $\frac{1}{2}\|u\|^2$ and $\frac{1}{2}\|D u\|^2$ are regularizers needed in the setting of our algorithm. The parameters $\lambda_1,\lambda_2,\lambda_3, \lambda_4>0$ are used to control the relative importance of each of these components.
Problem~\eqref{delburring problem} can be seen as a special case of problem \eqref{convex optimization} with $n=3$, $\mathcal{H}_{1}=\mathbb{R}^m$, $\mathcal{H}_{2}=\mathbb{R}^{2m}$ and 
$$f_2(u)=\frac{1}{2}\|u-b\|^2, \quad f_3(u)=\frac{\lambda_{1}}{2}\|u\|^2, \quad g_{2}(v)=\lambda_{2}\|v\|_{\iso}+ \frac{\lambda_{3}}{2}\|v\|^2, \quad g_{3}(v)=\frac{\lambda_4}{2}\|v\|^2,\quad C=D.$$
Here, $f_{2}$ is differentiable with $1$-Lipschitz continuous gradient, $f_{3}$ is $\lambda_{1}$-strongly convex, $g_{2}$ is $\lambda_{3}$-strongly convex, and $g_{3}$ is differentiable with $\lambda_4$-Lipschitz continuous gradient. We can therefore apply Algorithm~\ref{alg:PD}, for which Corollary~\ref{corollary primal dual linear convergence} guarantees linear convergence to the solution of \eqref{delburring problem}
 To do so, it is necessary to compute the proximal operators appearing \eqref{new x}. In what follows, we collect expressions for these operators. 
\begin{enumerate}[(i)]
    \item %According to  \cite[p.~526]{o2017total}, 
    The proximal operator of the functions $f_2$, $f_3$, $g_3$, respectively, are given by
$$\prox_{f_{2}}(u)=\frac{u+b}{2}, \quad \prox_{f_{3}}(u)=\frac{u}{1+\lambda_{1}}, \quad\prox_{g_3}(v)=\frac{v}{1+\lambda_4}.$$

\item Denote $y=\binom{v^1}{v^2}\in\mathbb{R}^{2m}$. By combining~\cite[Proposition 23.29(i)]{bauschke2011convex} and \cite[p.~1727]{o2014primal},  we obtain
$$\prox_{g_2}(v)=\prox_{\frac{\lambda_2}{\lambda_3+1}\|\cdot\|_{\iso}}\left(\frac{1}{\lambda_3+1}v\right) = \frac{1}{\lambda_3+1}\binom{\alpha\odot v^1}{\alpha\odot v^2},$$
where $\odot$ denotes the Hadamard product and $\alpha\in\mathbb{R}^{m}$ is given by
 $$ \alpha_i = \begin{cases}
 1-\dfrac{\lambda_2}{\sqrt{({v^1_{i}})^2+({v^2_{i}})^2}} & \text{if } \sqrt{({v^1_{i}})^2+({v^2_{i}})^2}>{\lambda_2}, \\
 0 & \text{otherwise.}
               \end{cases} $$
\end{enumerate}
\begin{table}[b]
    \centering
    \renewcommand{\arraystretch}{1.2}
    \caption{The effect of different values of $\gamma$ for Algorithm~\ref{alg:PD} after 100 iterations.}
    \begin{tabular}{|c|c|c|c|}
    \hline
    $\gamma$ & $\frac{\|(\mathbf{p}^k,\mathbf{q}^k)-(\mathbf{p}^{k-1},\mathbf{q}^{k-1})\|}{m}$ & SNR \\
    \hline
    0.01 & $1.10\times10^{-5}$ & 1.72\\
    0.05 & $1.81\times10^{-5}$ & 6.52 \\
    0.1  & $8.49\times 10^{-6}$ & 10.56\\
    0.5  & $2.14\times 10^{-7}$ & 12.10\\
    0.9  & $4.51\times 10^{-8}$ & 12.10\\
    0.99 & $3.19\times 10^{-8}$ & 12.10\\
    \hline
    \end{tabular}
    \label{tab:different_values_of_gamma}
\end{table}
\begin{figure}[t]
    \centering
    \begin{subfigure}[b]{0.3\textwidth}
    \includegraphics[width=0.95\textwidth]{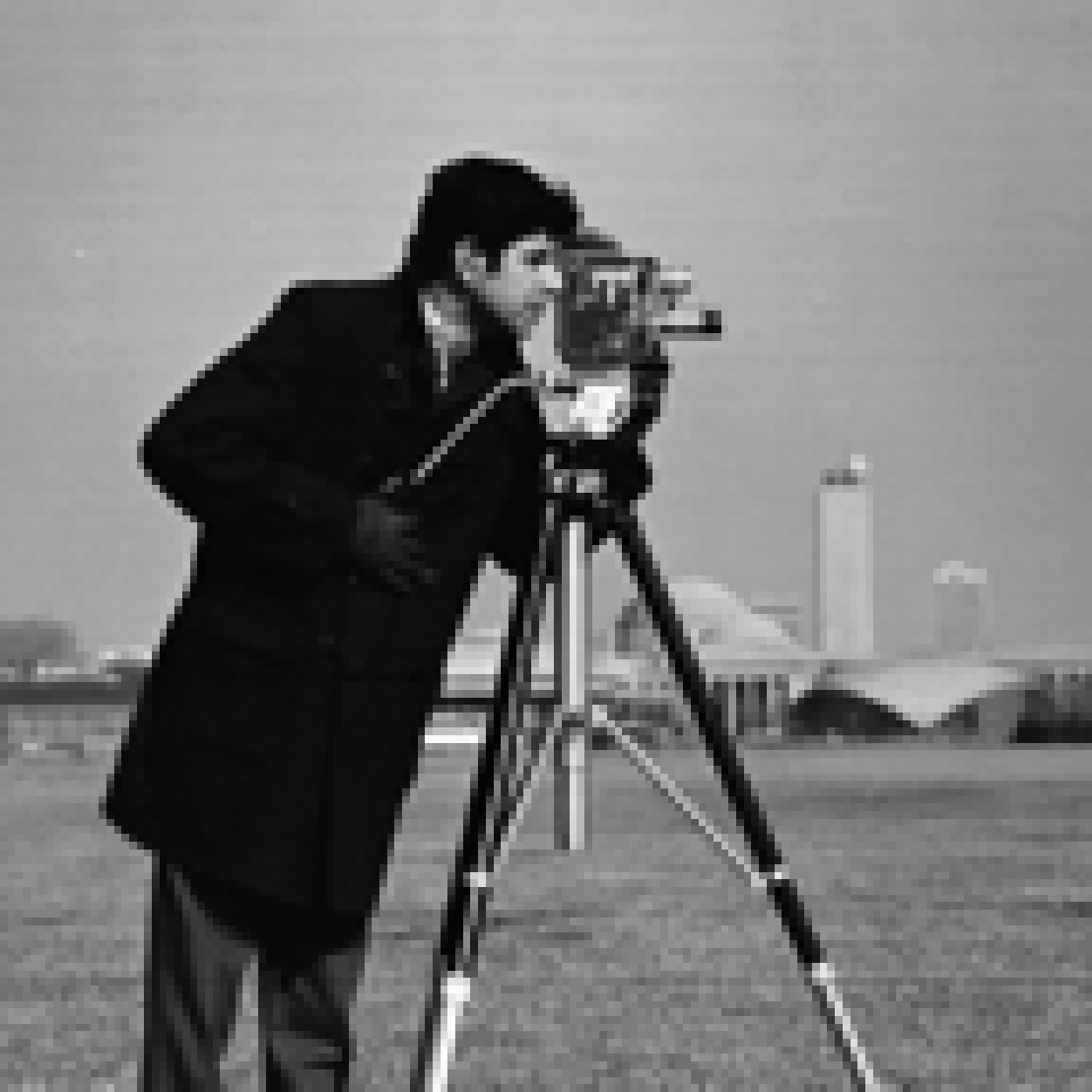}
    \caption{Original image}
    \end{subfigure}
    \begin{subfigure}[b]{0.3\textwidth}
    \includegraphics[width=0.95\textwidth]{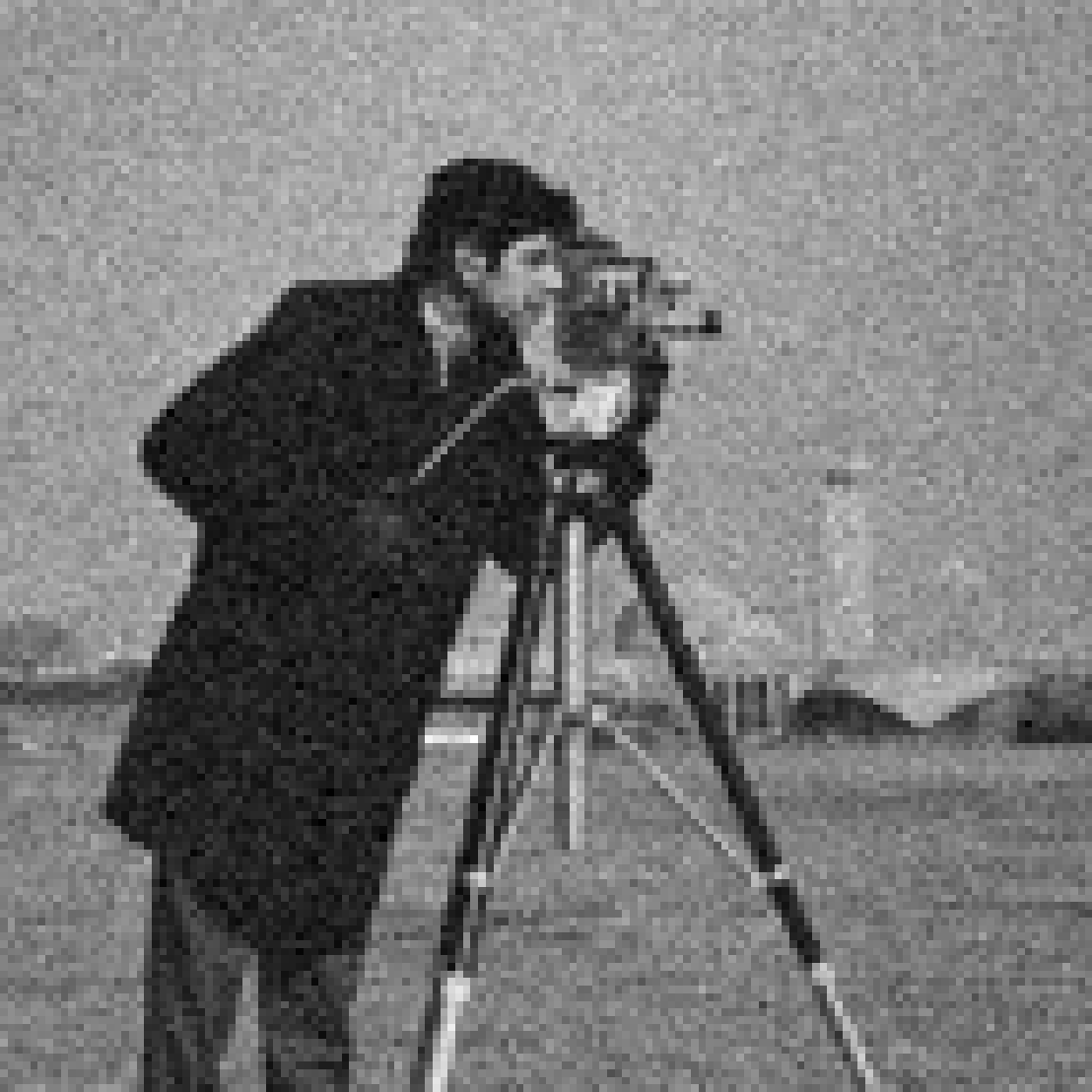}
    \caption{Noisy image}
    \end{subfigure}
    \begin{subfigure}[b]{0.3\textwidth}
    \includegraphics[width=0.95\textwidth]{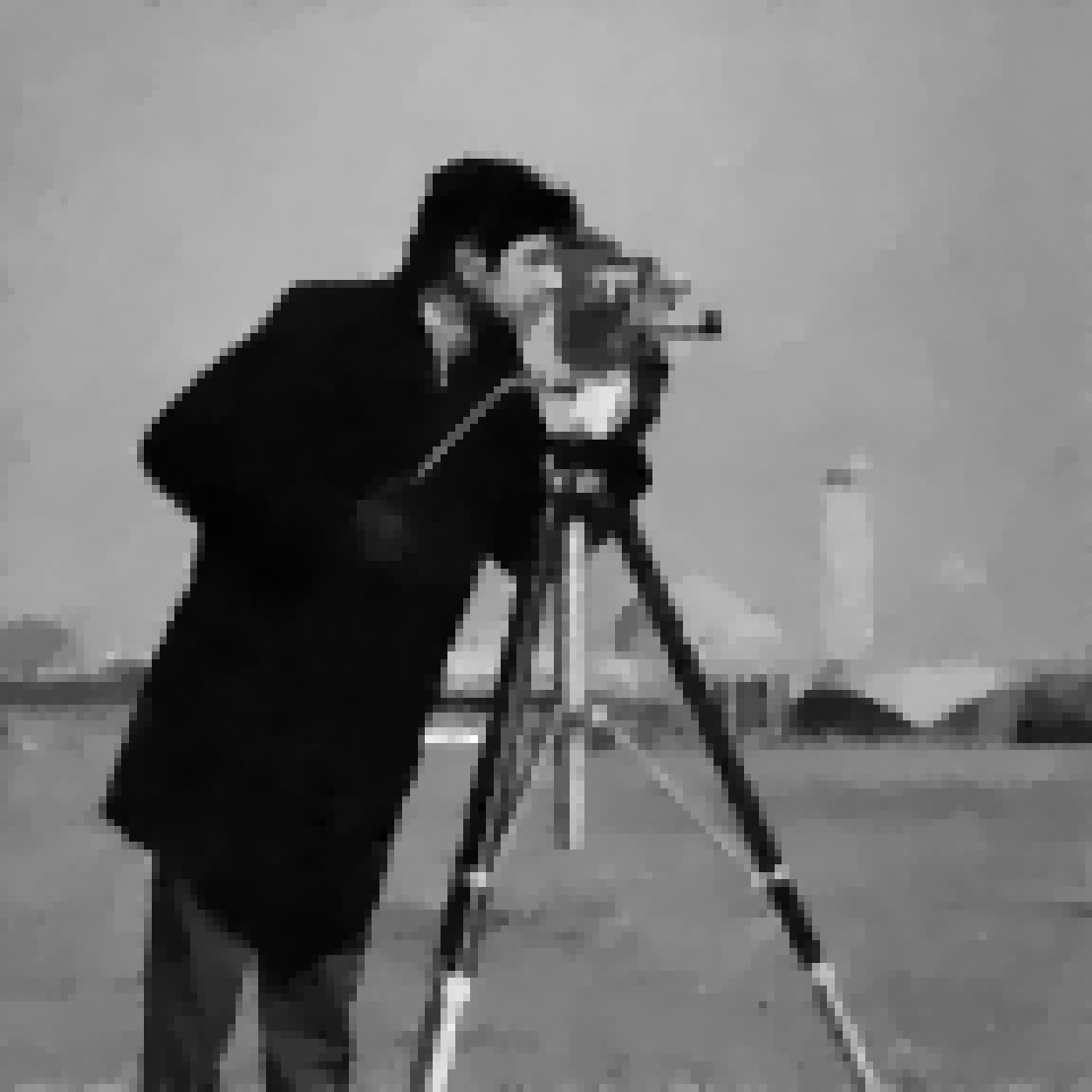}
    \caption{Restored image}
    \end{subfigure}
    \caption{The denoised image for Algorithm~\ref{alg:PD} with $\gamma=0.99, \lambda_{1}=0.01, \lambda_2=0.05, \lambda_3=0.0001$, and $\lambda_{4}=10$.}
    \label{fig:observed from noisy}
\end{figure}
\begin{figure}[h!]
\centering
    \begin{subfigure}[b]{0.19\textwidth}
    \includegraphics[width=0.95\textwidth]{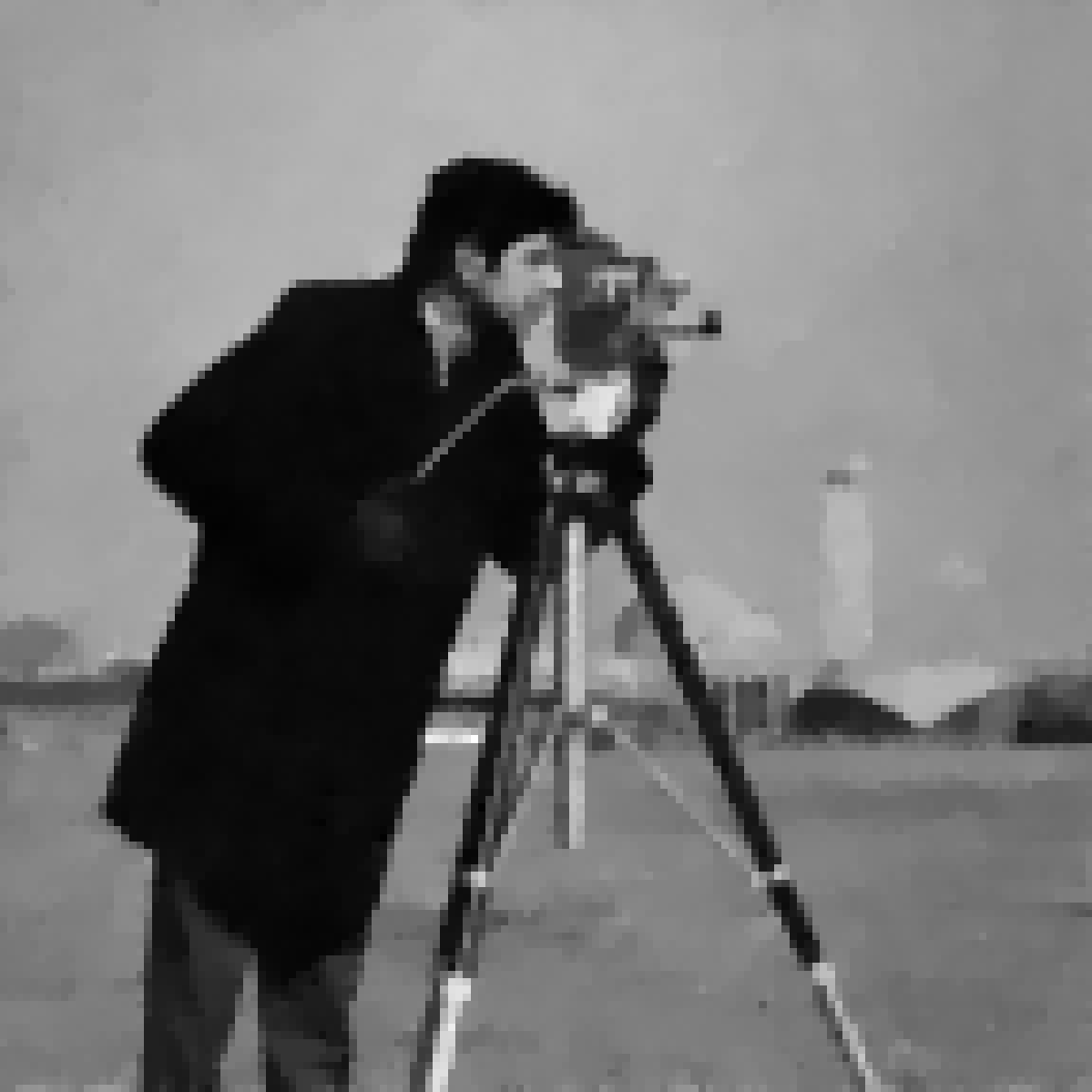}
    \caption{$\lambda_{1}=0.0001$}
    \end{subfigure}
    \begin{subfigure}[b]{0.19\textwidth}
    \includegraphics[width=0.95\textwidth]{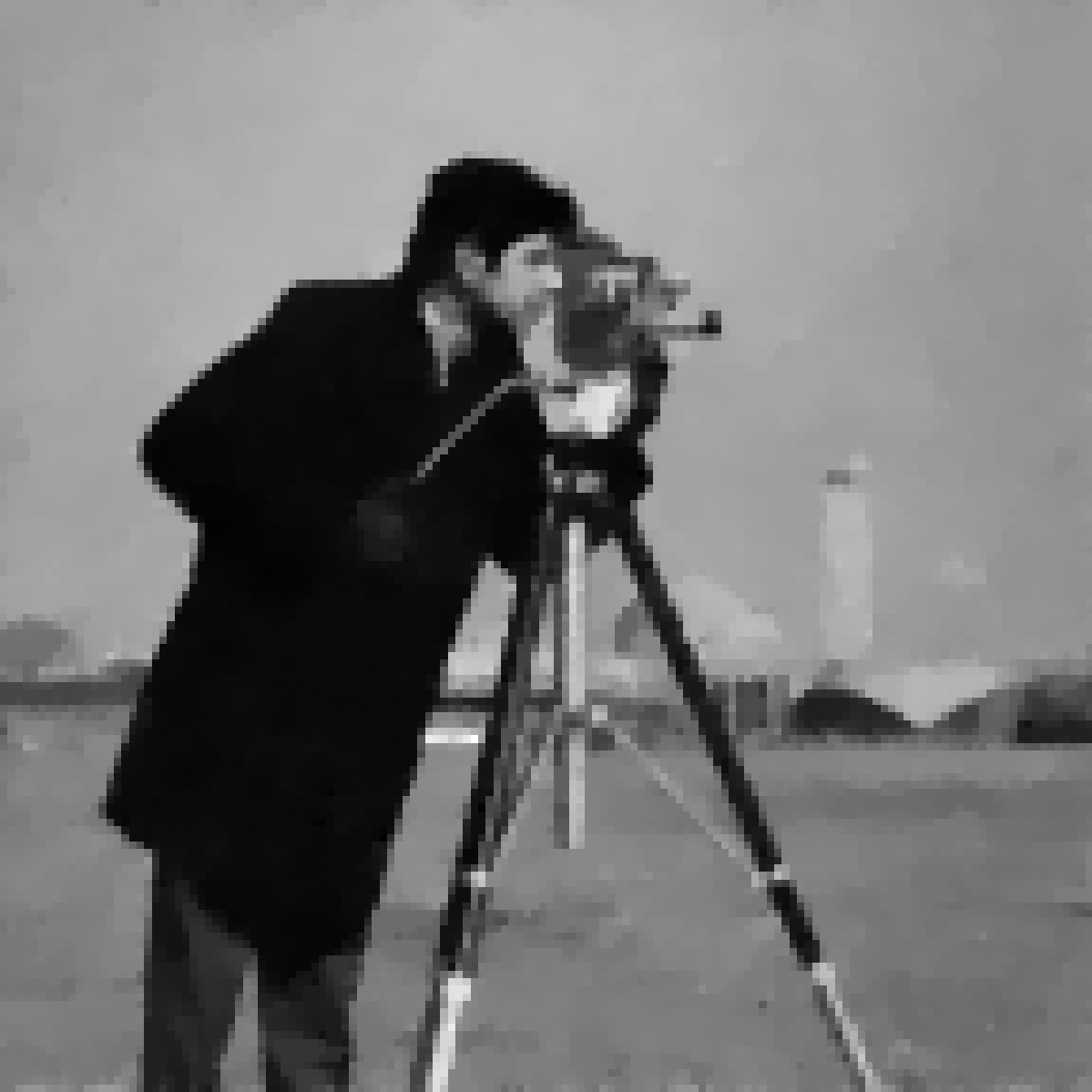}
    \caption{$\lambda_{1}=0.001$}
    \end{subfigure}
    \begin{subfigure}[b]{0.19\textwidth}
    \includegraphics[width=0.95\textwidth]{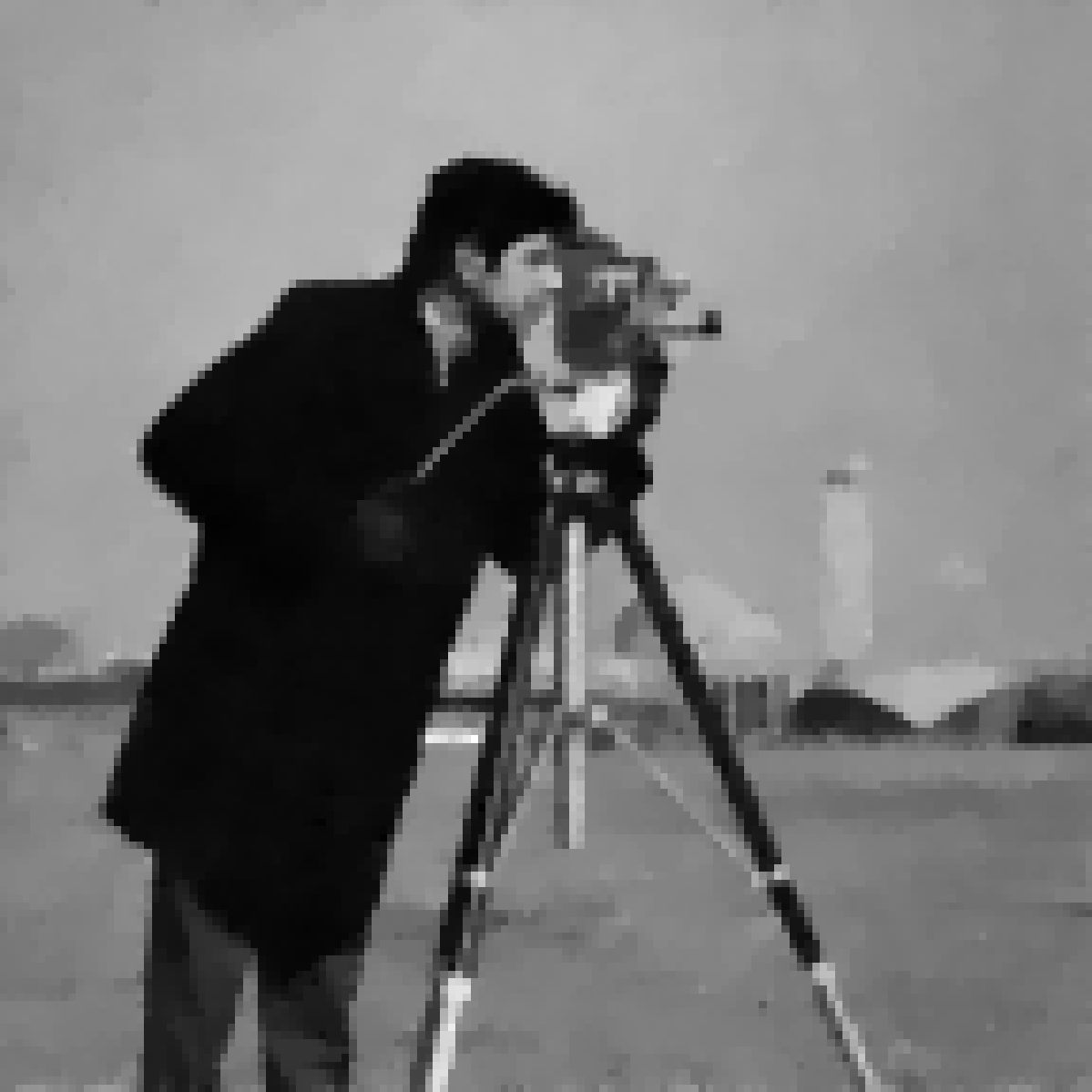}
    \caption{$\lambda_{1}=0.01$}
    \end{subfigure}
    \begin{subfigure}[b]{0.19\textwidth}
\includegraphics[width=0.95\textwidth]{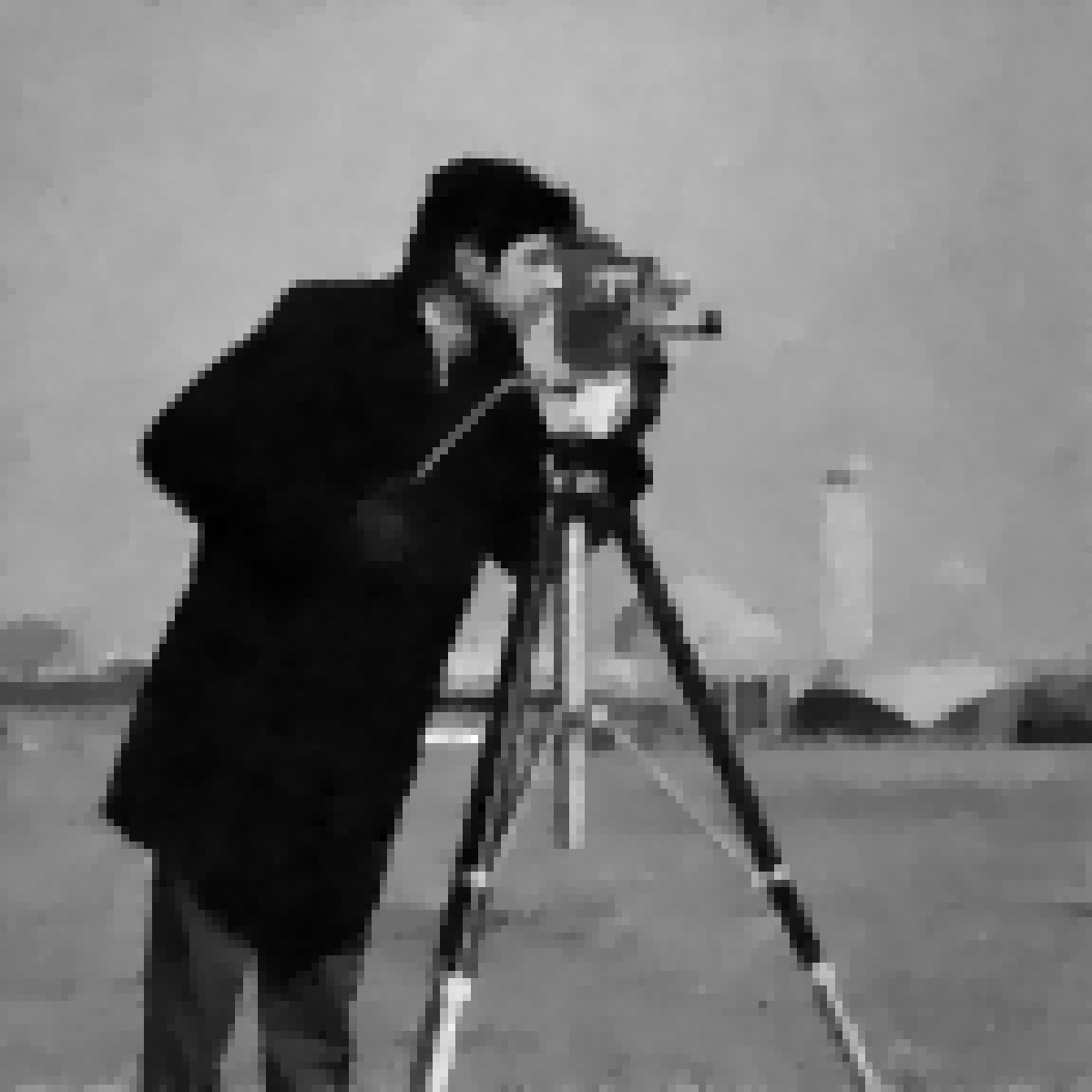}
\caption{$\lambda_{1}=1$}
    \end{subfigure}
    \begin{subfigure}[b]{0.19\textwidth}
\includegraphics[width=0.95\textwidth]{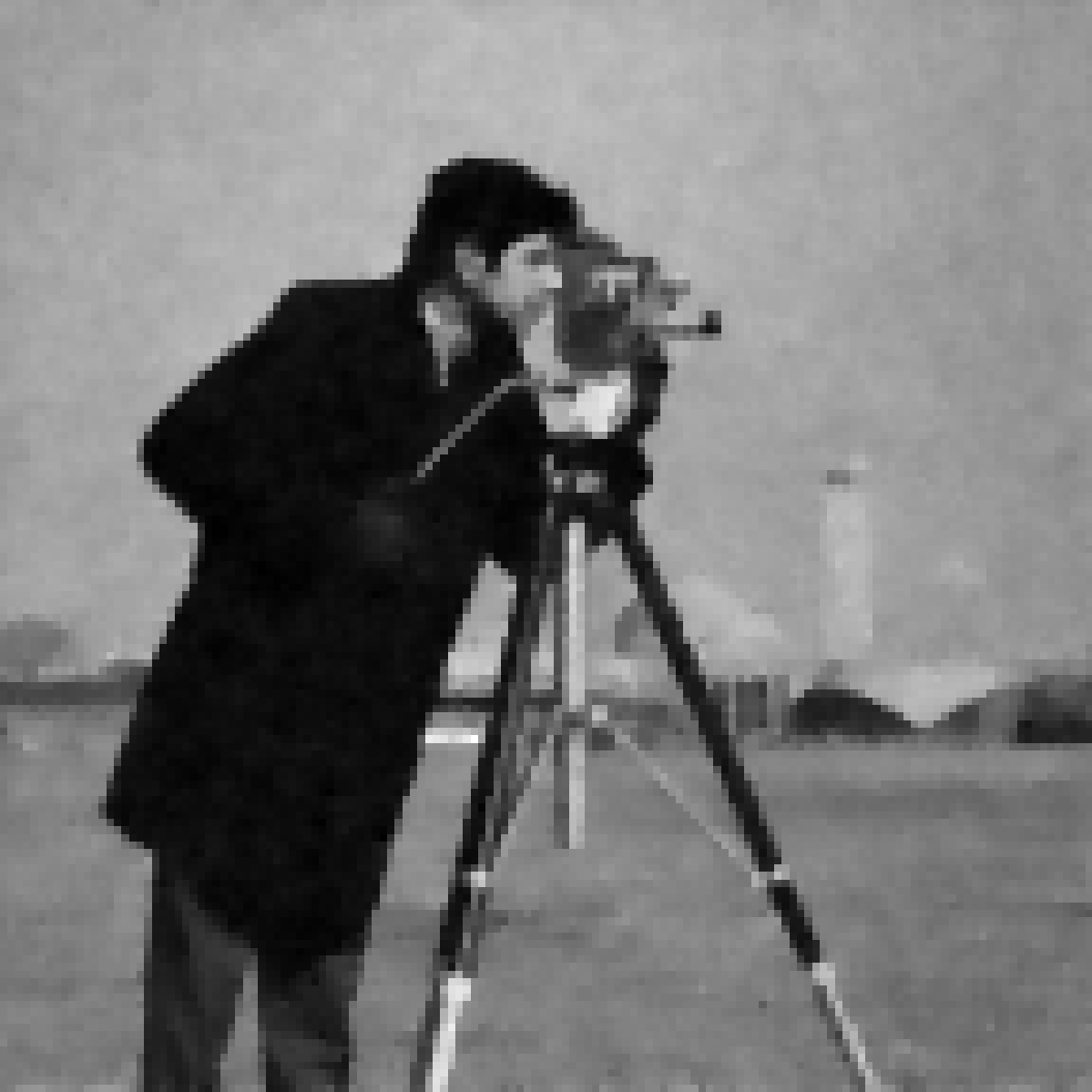}
\caption{$\lambda_1=5$}
\end{subfigure}
    \begin{subfigure}[b]{0.19\textwidth}
    \includegraphics[width=0.95\textwidth]{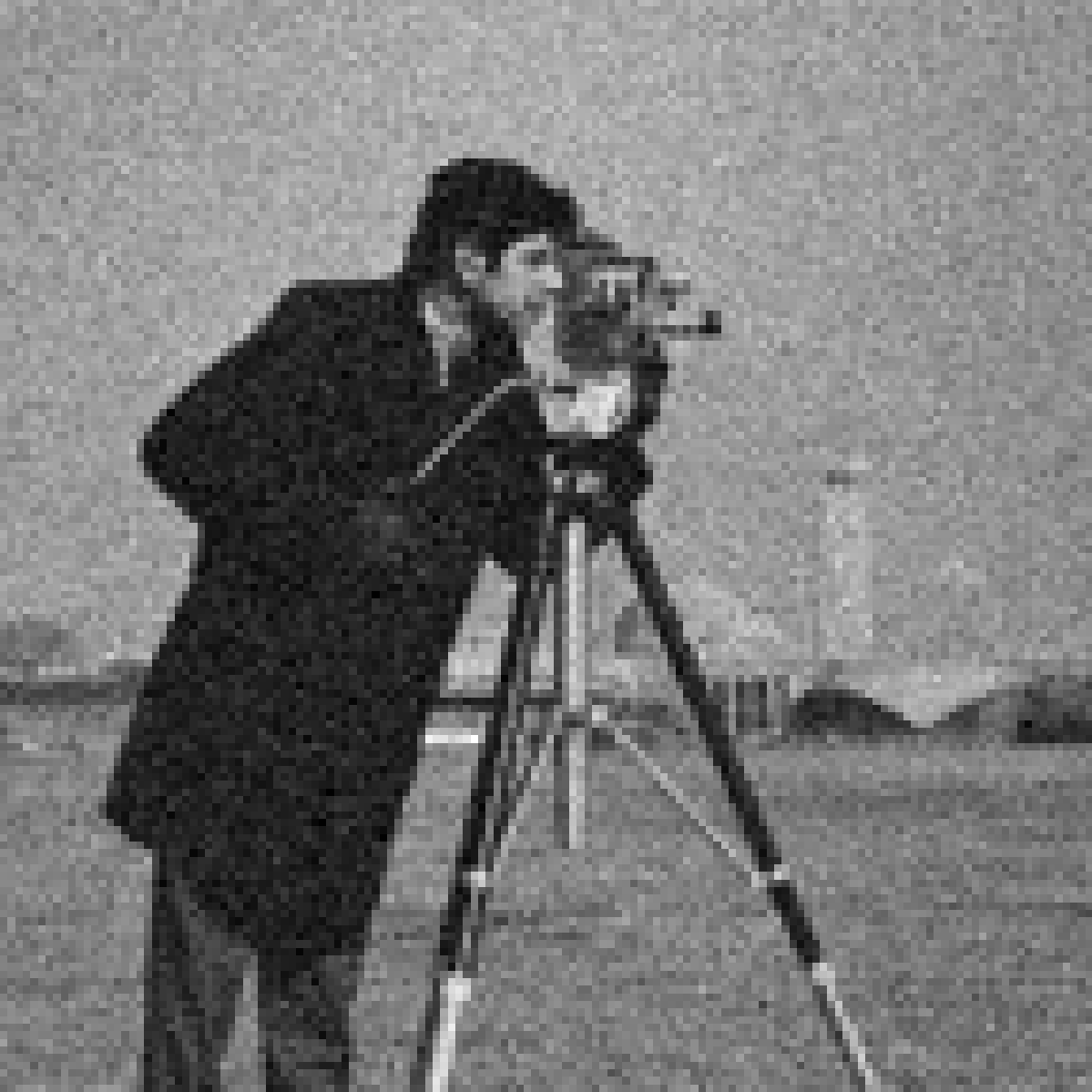}
    \caption{$\lambda_{2}=0.001$}
    \end{subfigure}
    \begin{subfigure}[b]{0.19\textwidth}
    \includegraphics[width=0.95\textwidth]{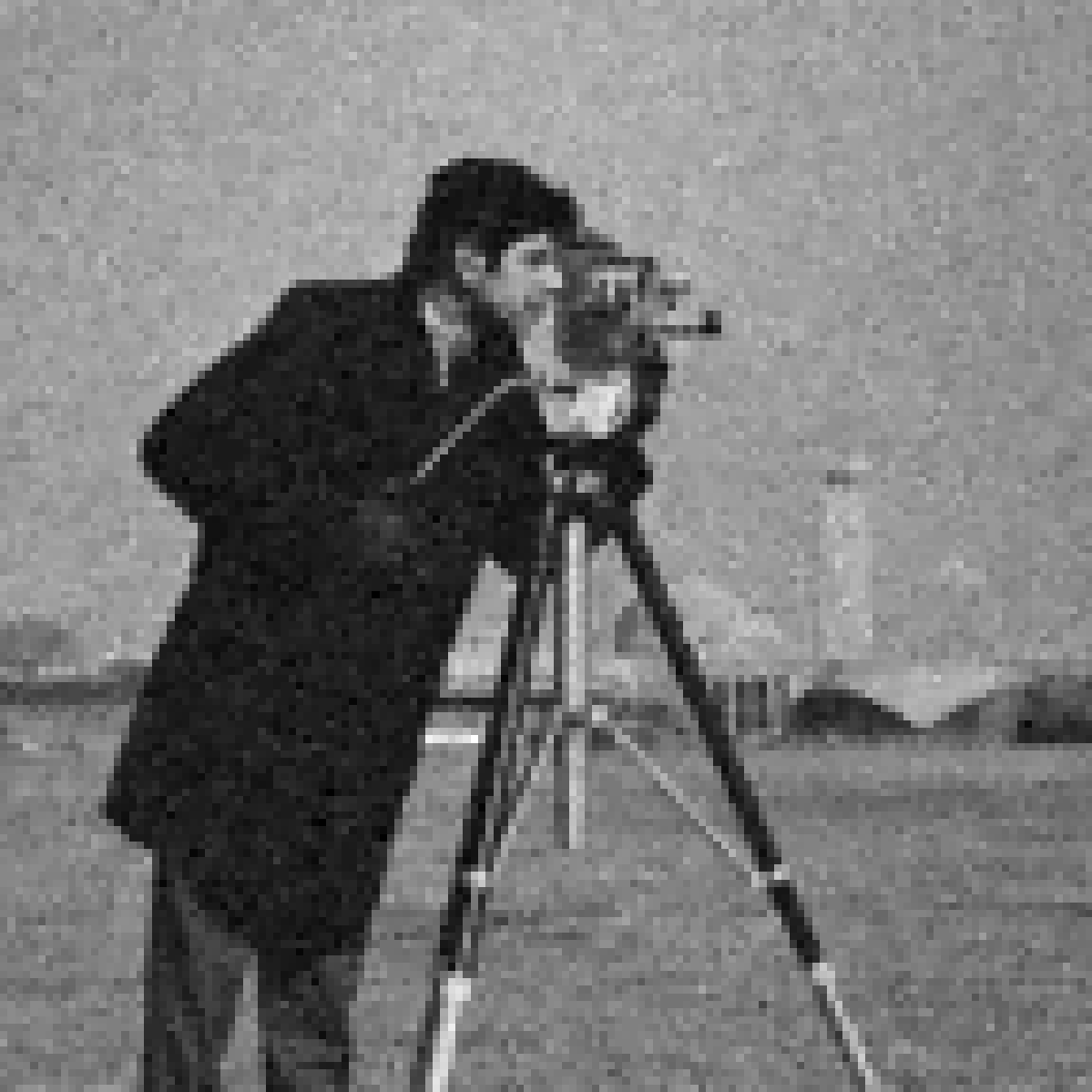}
     \caption{$ \lambda_{2}=0.01$}
    \end{subfigure}
    \begin{subfigure}[b]{0.19\textwidth}
    \includegraphics[width=0.95\textwidth]{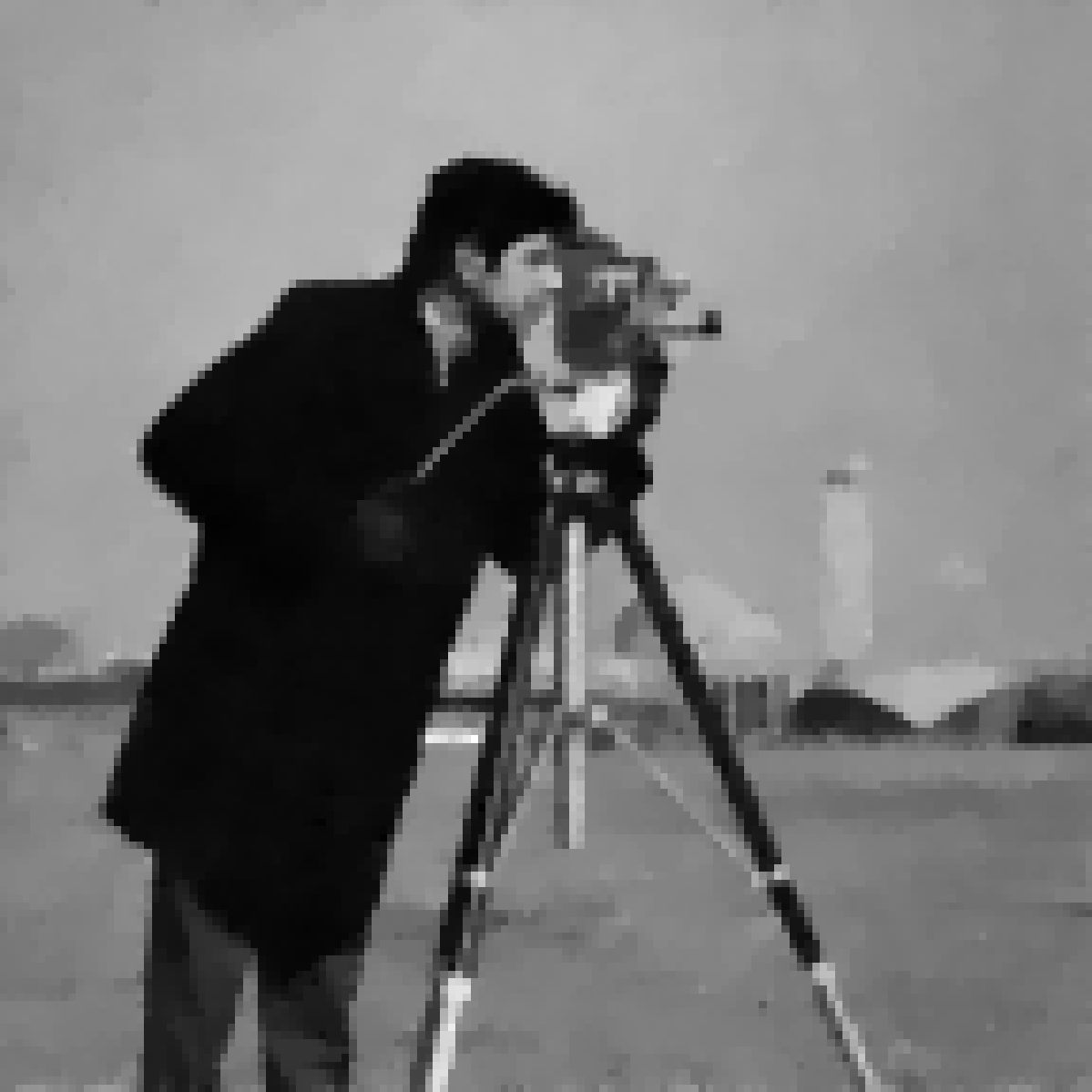}
     \caption{$ \lambda_{2}=0.05$}
    \end{subfigure}
    \begin{subfigure}[b]{0.19\textwidth}
    \includegraphics[width=0.95\textwidth]{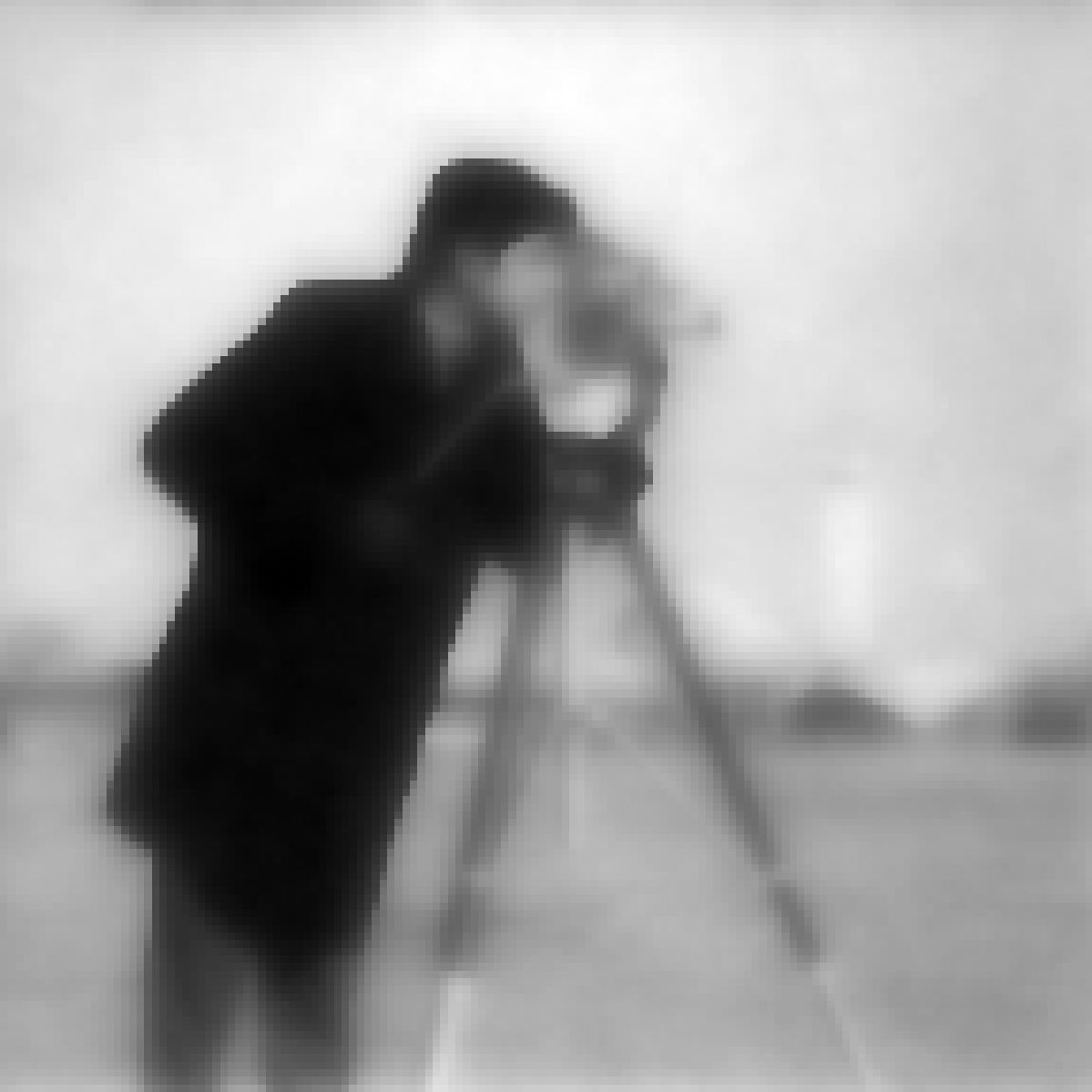}
    \caption{$ \lambda_{2}=0.5$}
    \end{subfigure}
    \begin{subfigure}[b]{0.19\textwidth}
    \includegraphics[width=0.95\textwidth]{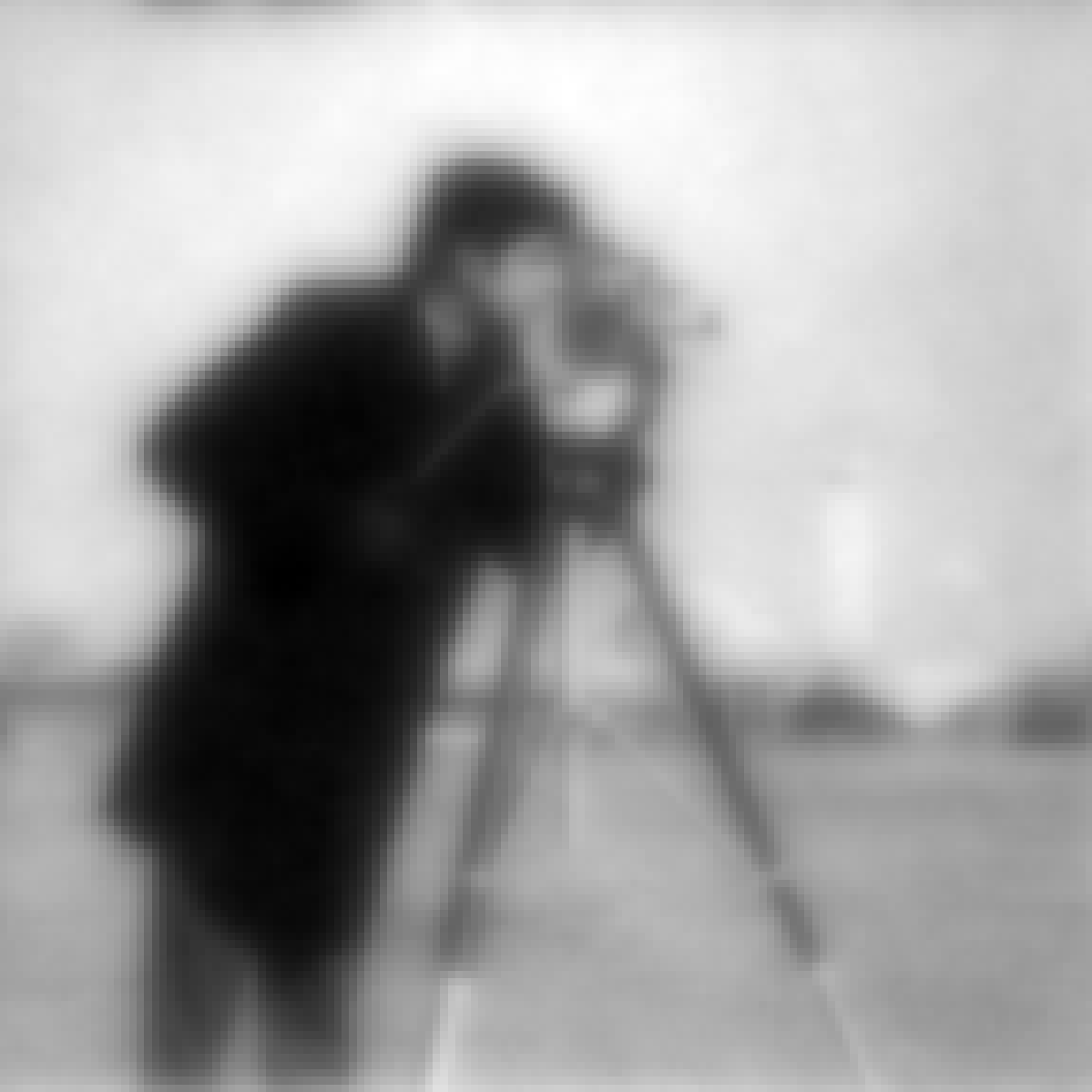}
    \caption{$\lambda_{2}=1$}
    \end{subfigure}
    \begin{subfigure}[b]{0.19\textwidth}
    \centering
    \includegraphics[width=0.95\textwidth]{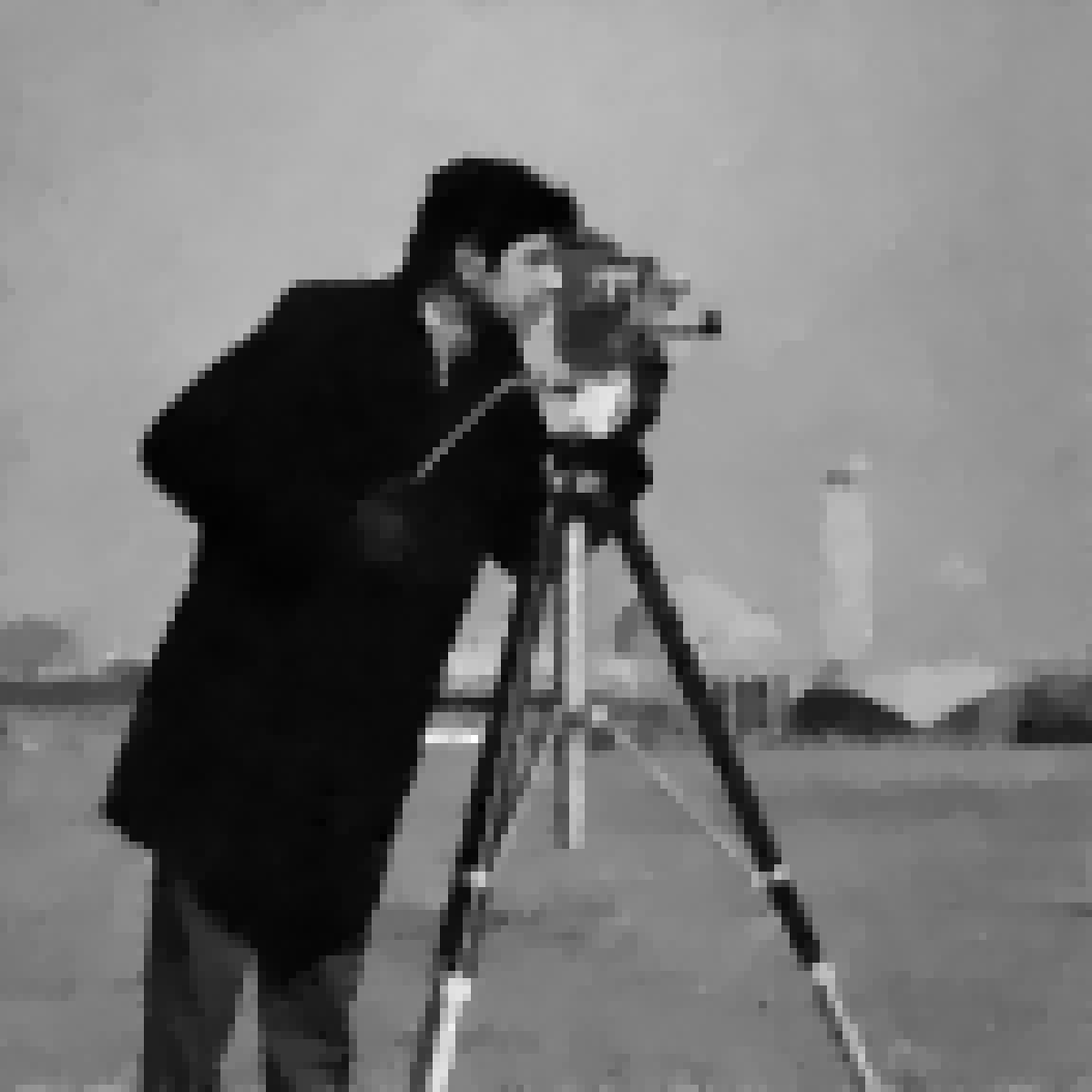}
    \caption{$\lambda_3=0.000001$}
    \end{subfigure}
    \begin{subfigure}[b]{0.19\textwidth}
    \includegraphics[width=.95\textwidth]{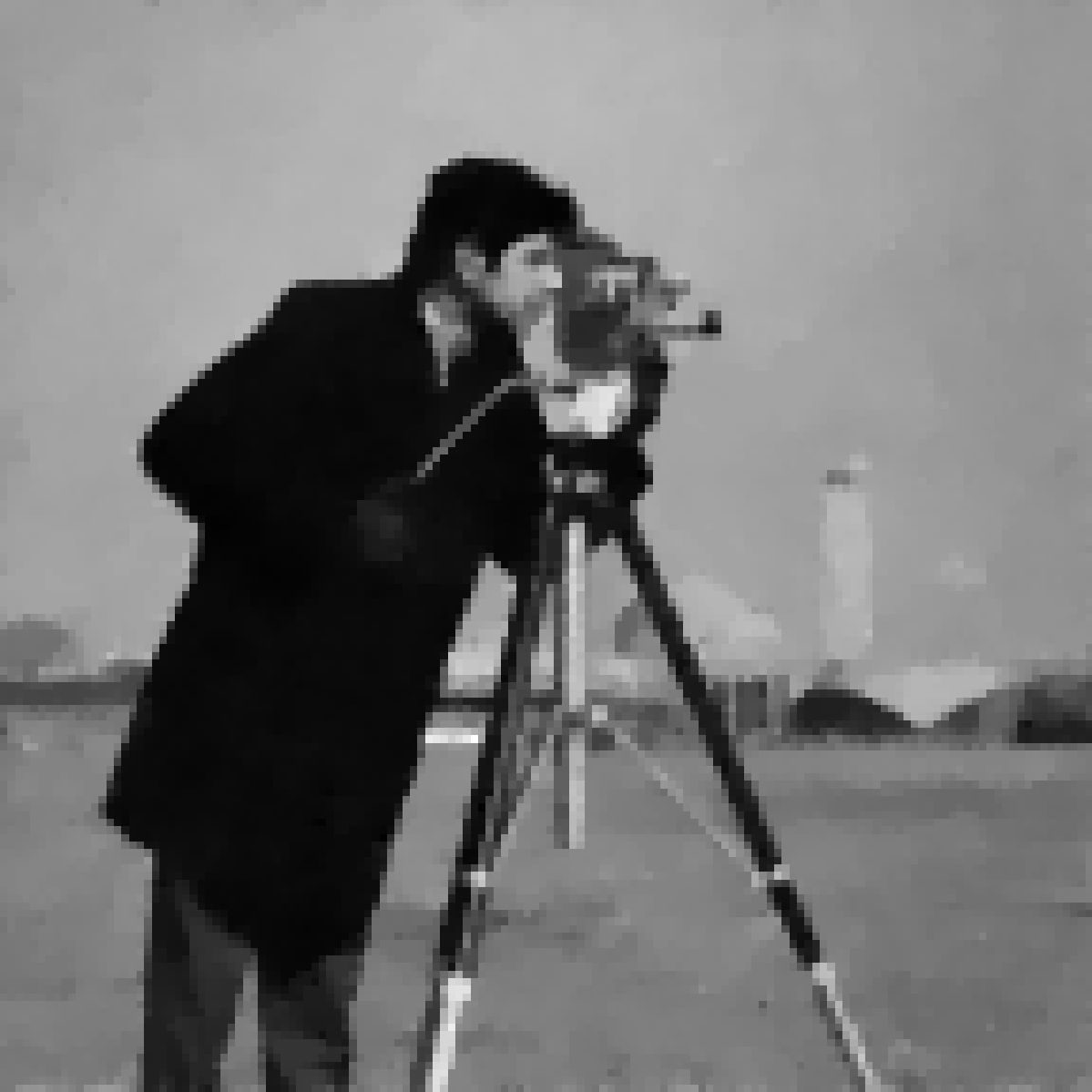}
    \caption{$\lambda_3=0.00001$}
    \end{subfigure}
    \begin{subfigure}[b]{0.19\textwidth}
    \includegraphics[width=0.95\textwidth]{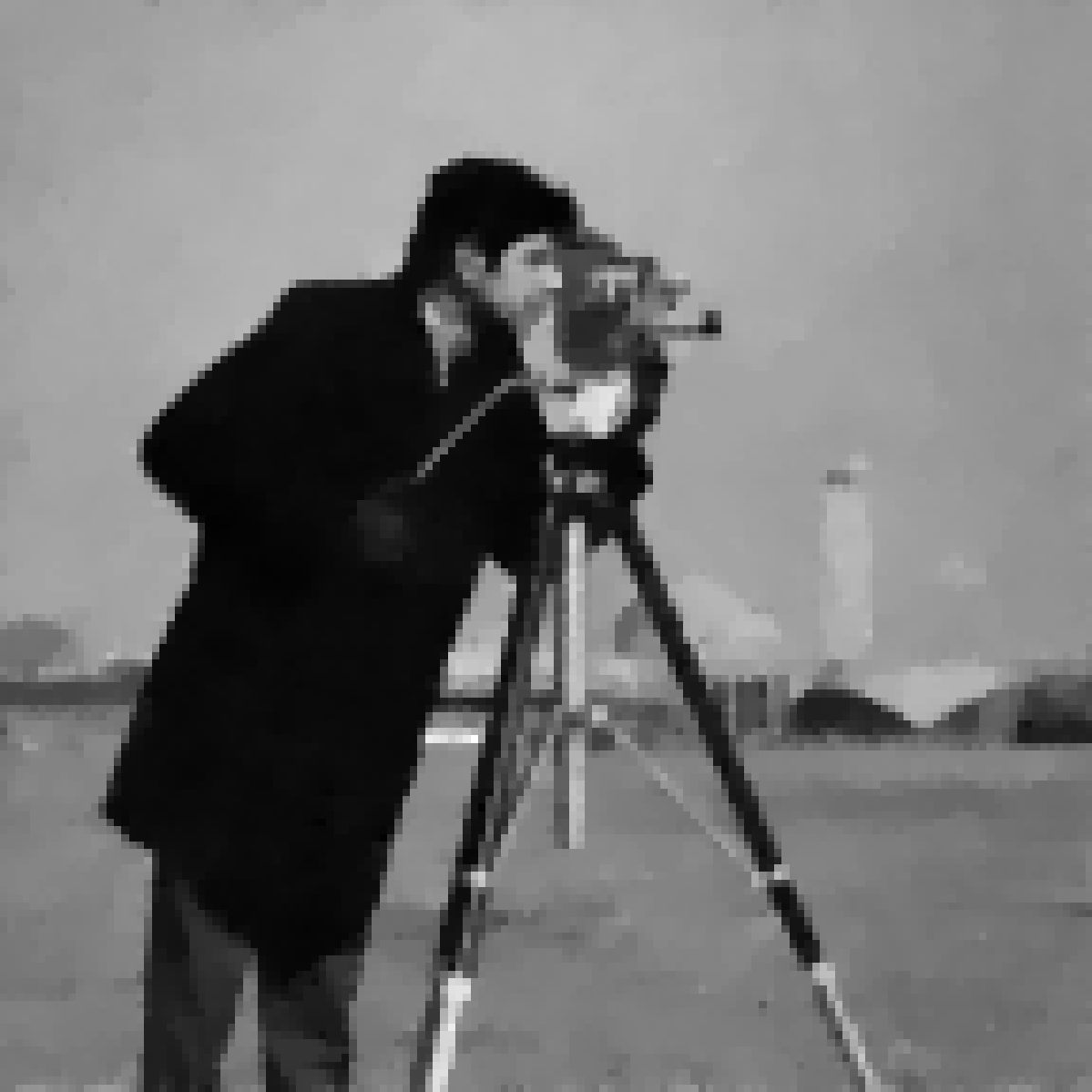}
    \caption{$\lambda_3=0.0001$}
    \end{subfigure}
    \begin{subfigure}[b]{0.19\textwidth}
    \includegraphics[width=0.95\textwidth]{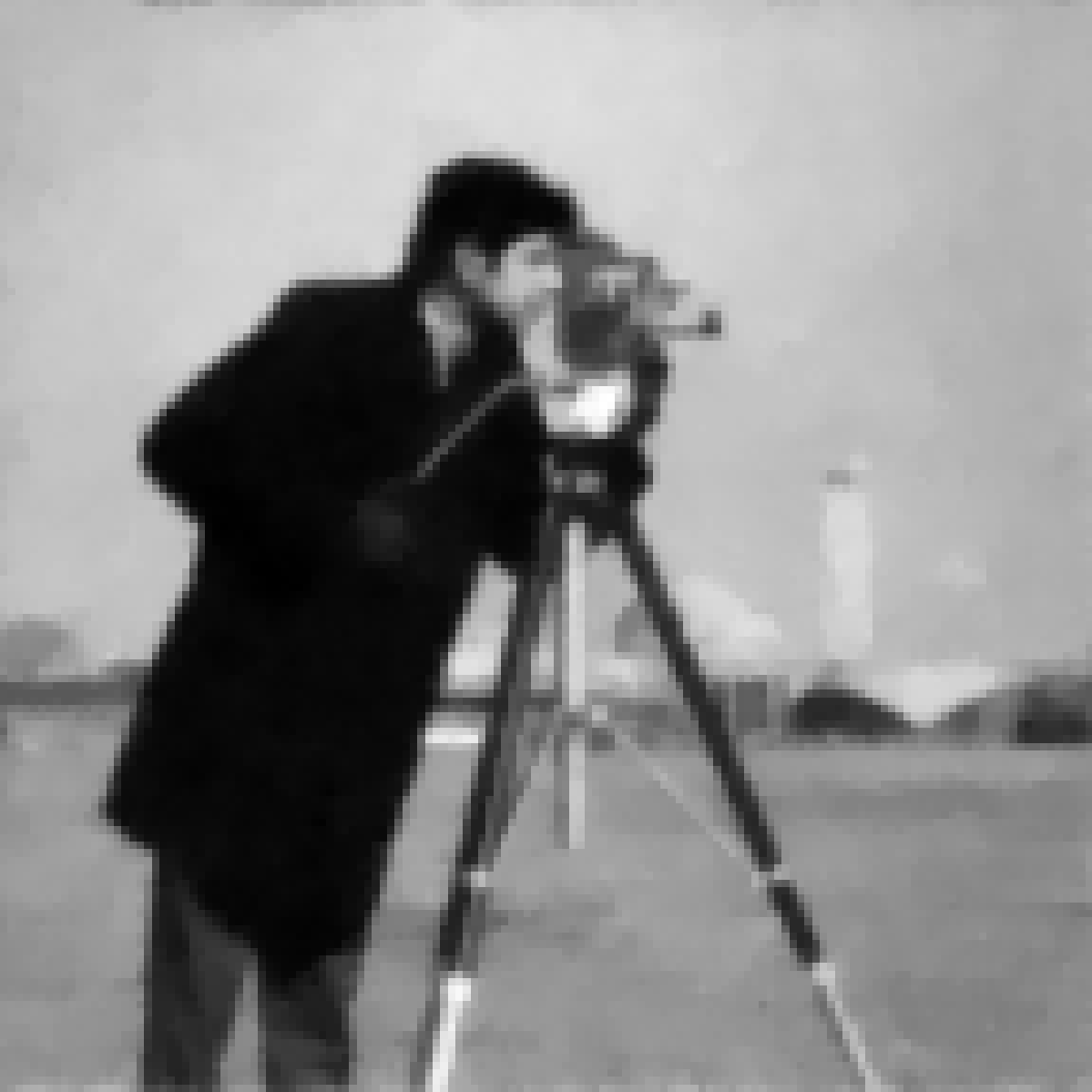}
    \caption{$\lambda_3=0.5$}
    \end{subfigure}
    \begin{subfigure}[b]{0.19\textwidth}
    \includegraphics[width=0.95\textwidth]{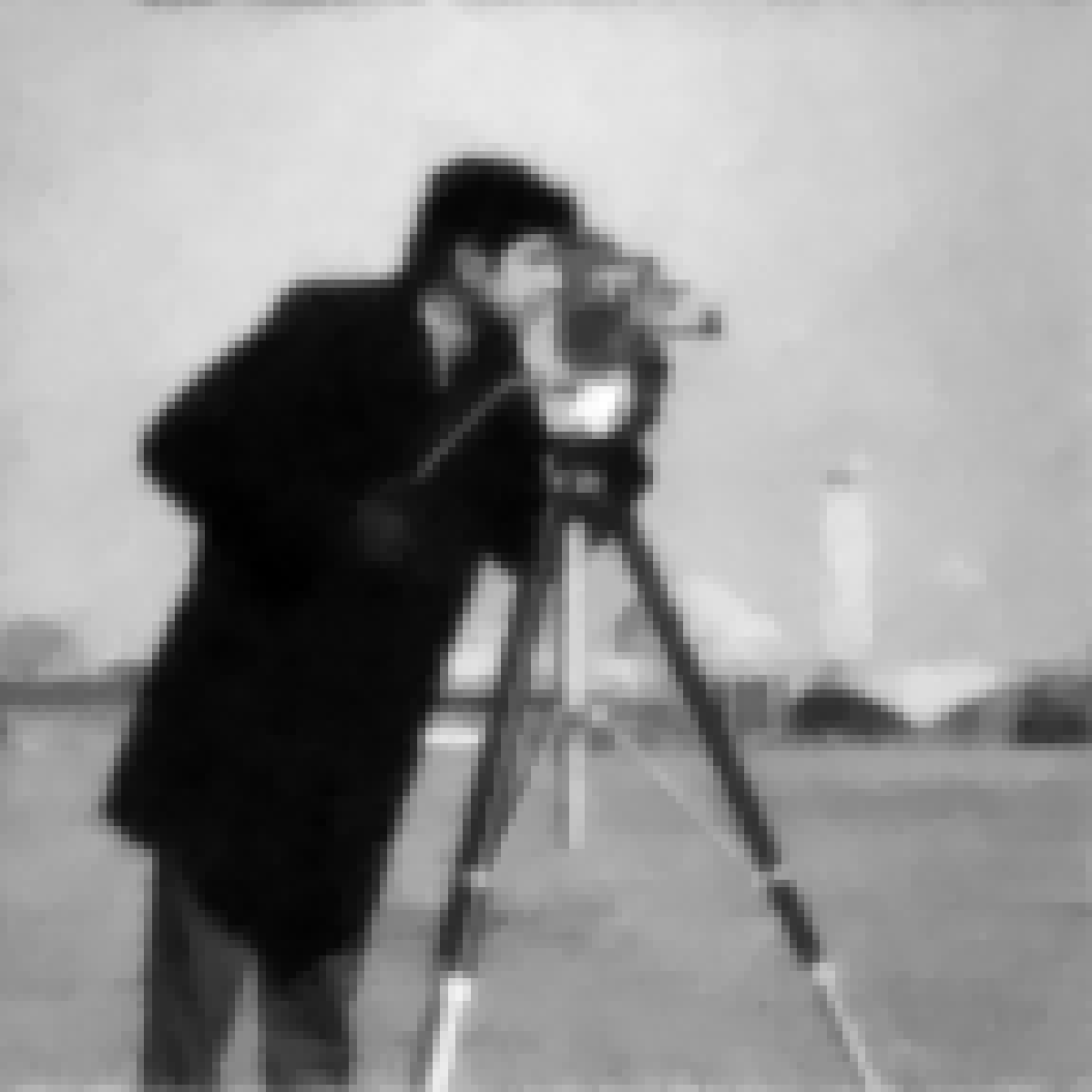}
    \caption{$\lambda_3=1$}
    \end{subfigure}
    \begin{subfigure}[b]{0.19\textwidth}
    \includegraphics[width=0.95\textwidth]{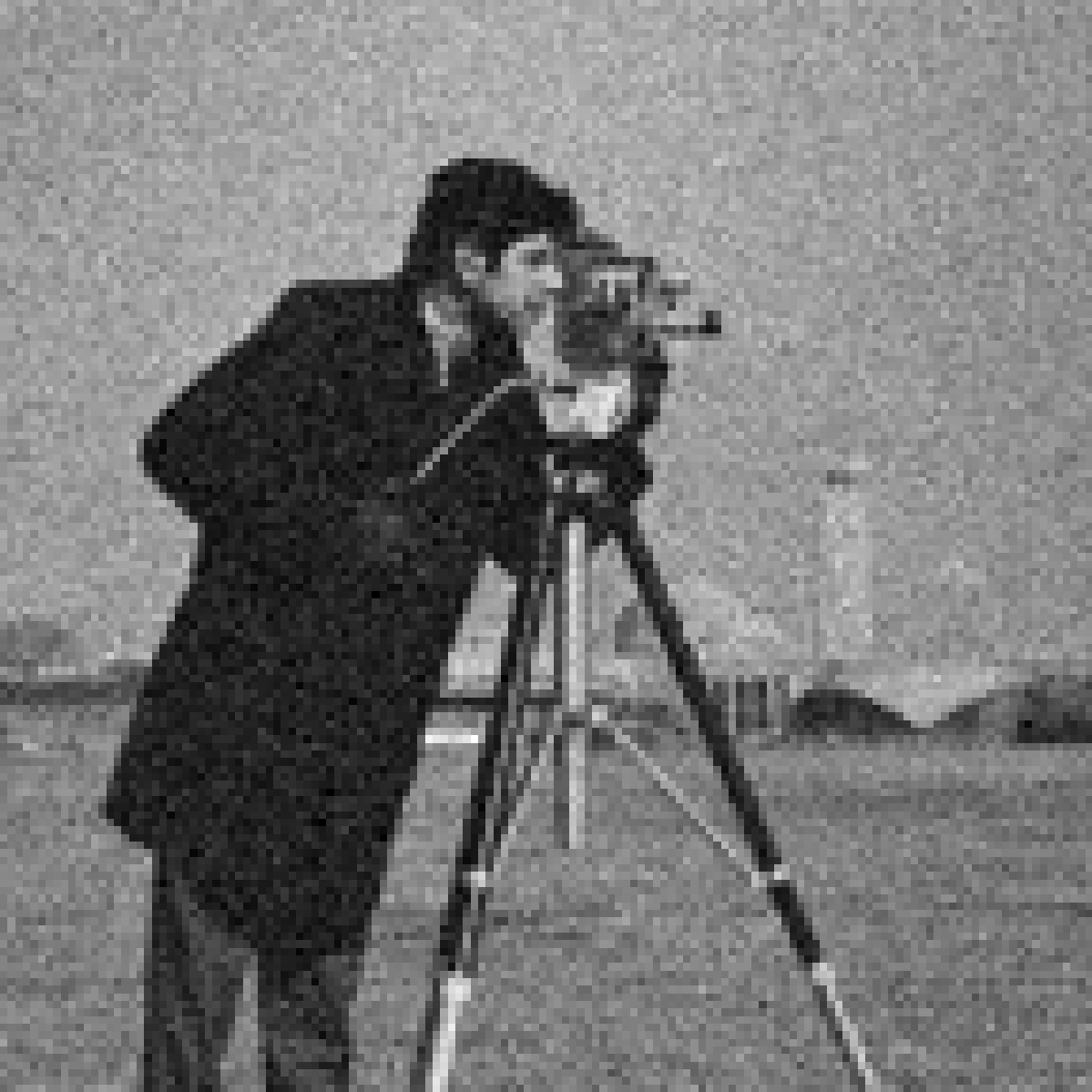}
    \caption{$\lambda_4=0.001$}
    \end{subfigure}
    \begin{subfigure}[b]{0.19\textwidth}
    \includegraphics[width=0.95\textwidth]{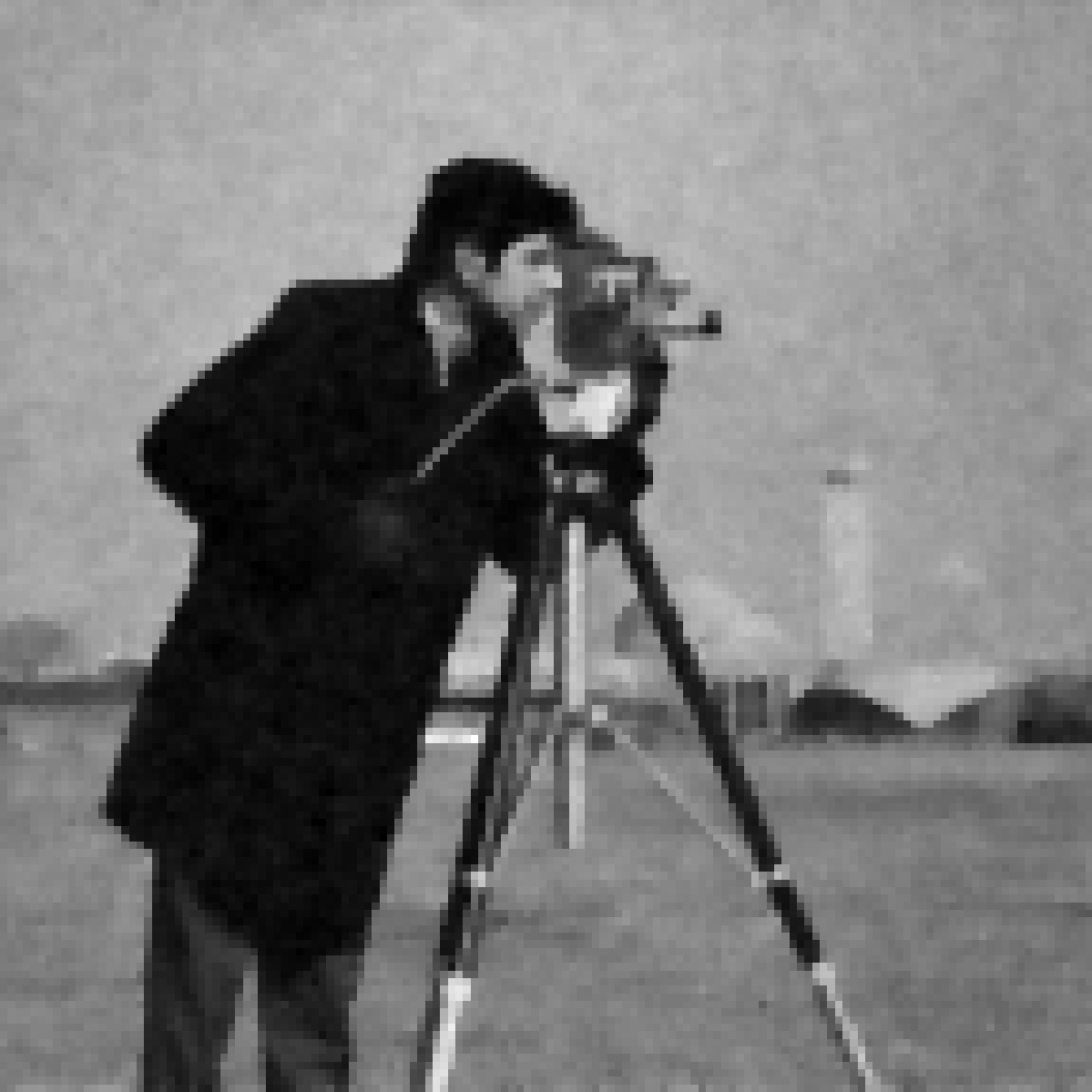}
    \caption{$\lambda_4=1$}
    \end{subfigure}
    \begin{subfigure}[b]{0.19\textwidth}
    \includegraphics[width=0.95\textwidth]{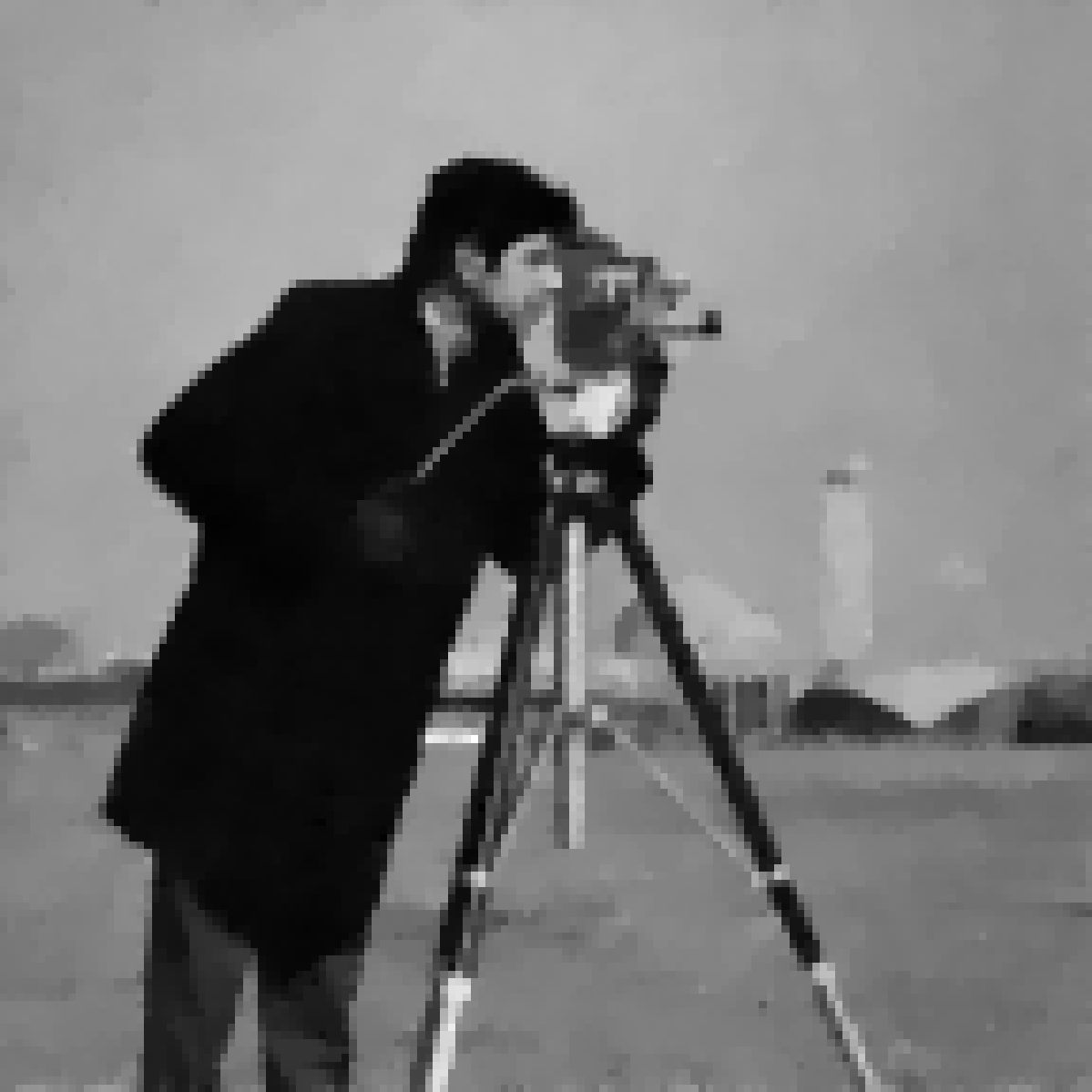}
    \caption{$\lambda_4=10$}
    \end{subfigure}
    \begin{subfigure}[b]{0.19\textwidth}
    \includegraphics[width=0.95\textwidth]{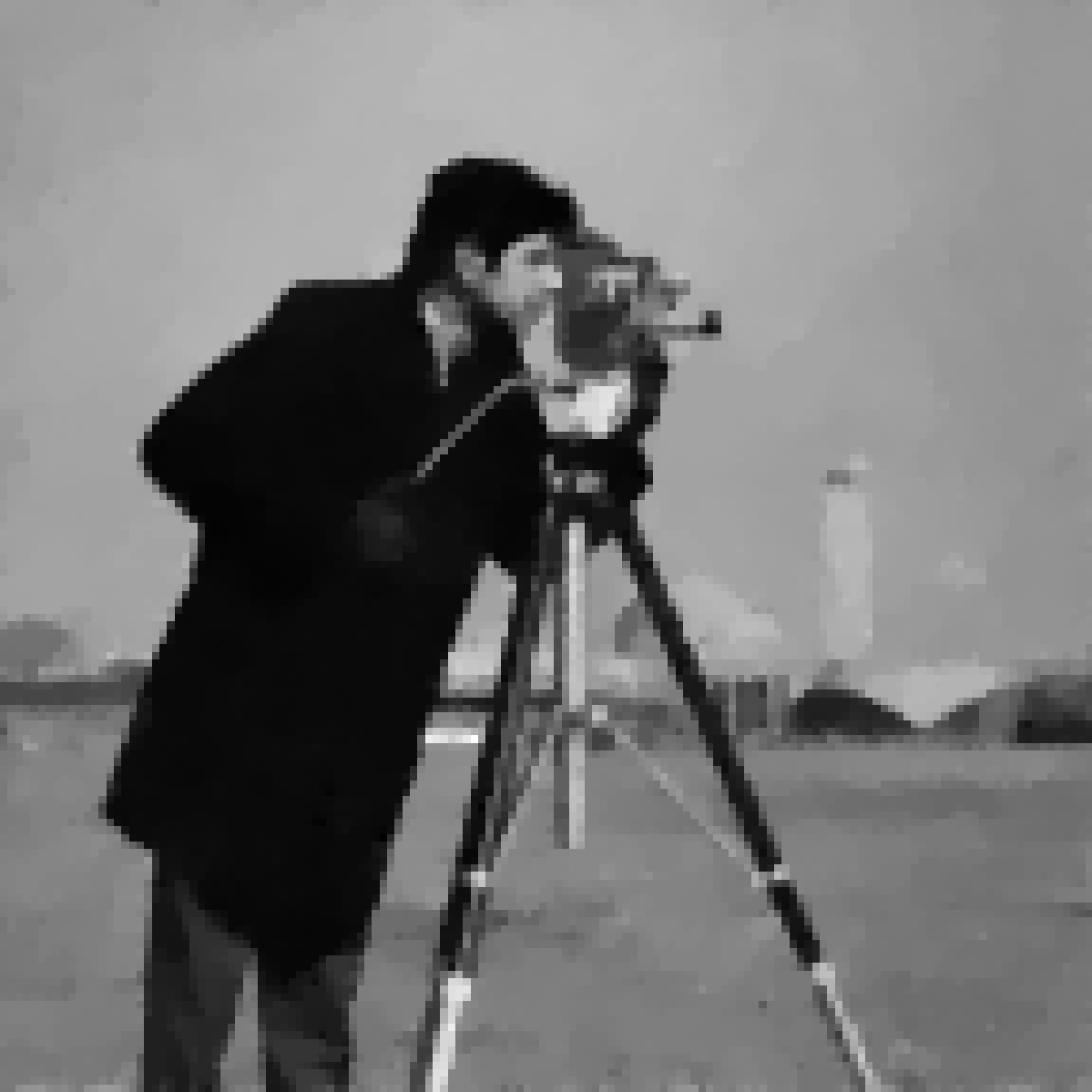}
    \caption{$\lambda_4=20$}
    \end{subfigure}
    \begin{subfigure}[b]{0.19\textwidth}
    \includegraphics[width=0.95\textwidth]{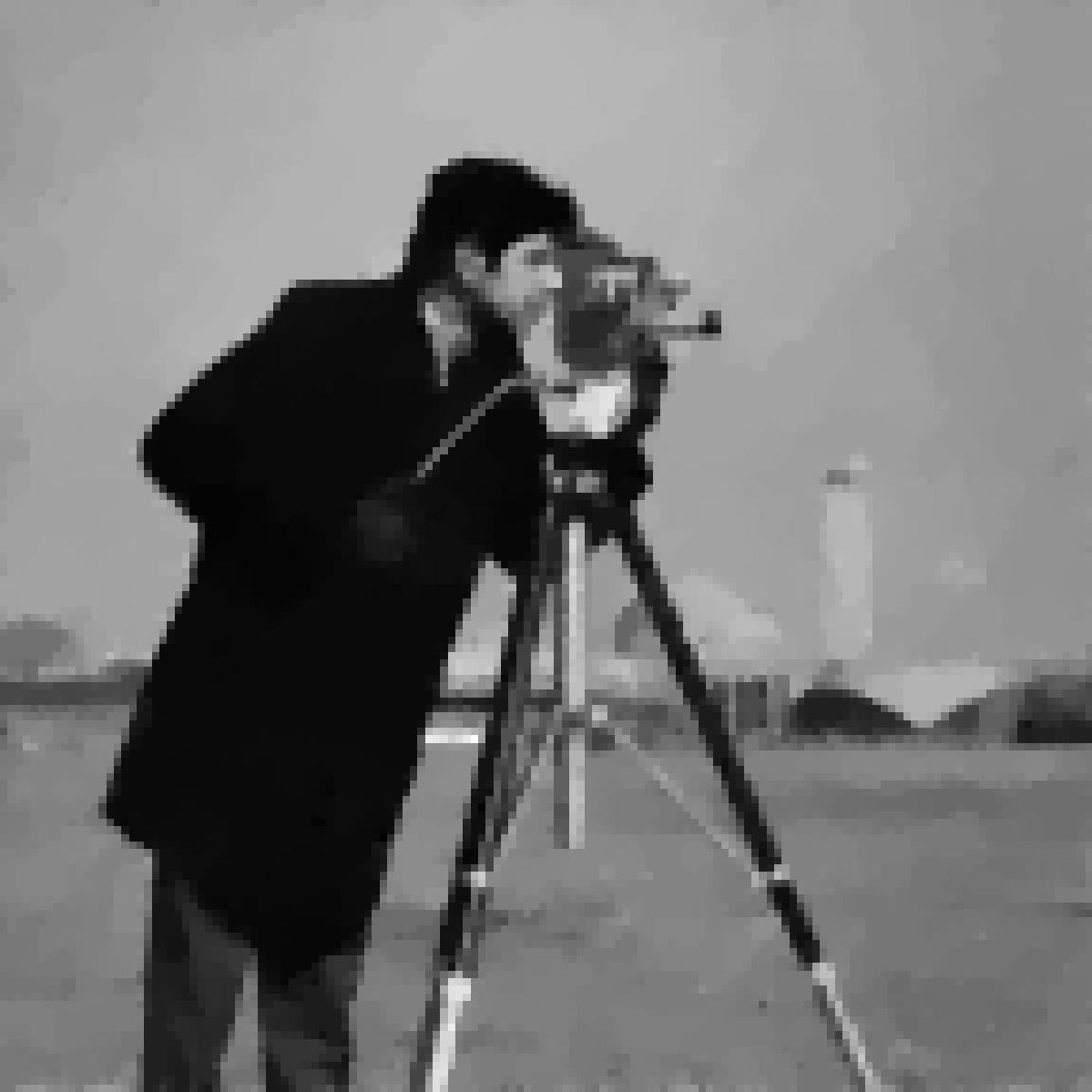}
    \caption{$\lambda_4=100$}
    \end{subfigure}
    \caption{The effect of varying $\lambda_1, \lambda_2, \lambda_3$ and $\lambda_4$ on the solution of \eqref{delburring problem}.}
    \label{fig:effect of lambda4}
\end{figure}
To examine the effect of varying regularization parameters, we applied Algorithm~\ref{alg:PD} to the problem \eqref{delburring problem} with the ``Cameraman" image~\cite{schreiber1978image}. The original image was degraded by a Gaussian noise with zero mean and standard deviation 0.05. Figure~\ref{fig:observed from noisy} shows the original, noisy, and the recovered image. Using a trial-and-error process, summarized in Figure~\ref{fig:effect of lambda4}, $(\lambda_{1},\lambda_2,\lambda_3,\lambda_4)=(0.01,0.05,0.0001,10)$ were found to give to be a good reconstruction, which is used in all experiments. For the initialization, we set $\mathbf{p}^0=(\mathbf{0},\mathbf{0})\in\mathbb{R}^{m}\times \mathbb{R}^m$ and $\mathbf{q}^0=(\mathbf{0},\mathbf{0})\in\mathbb{R}^{2m}\times\mathbb{R}^{2m}$. Table~\ref{tab:different_values_of_gamma} shows the normalized final change iterates $\frac{\|(\mathbf{p}^k,\mathbf{q}^k)-(\mathbf{p}^{k-1},\mathbf{q}^{k-1})\|}{m}$ and the signal to noise ratio after 100 iterations. Here the normalization by $m$ ensures the algorithm performance is independent across different image sizes, making the convergence rate independent of image resolution, and signal to noise ratio~\cite{price1993signals} is defined as $$\text{SNR}=10\cdot\log_{10}\left(\frac{\|\text{Original}-\text{Restored}\|}{\|\text{Original}\|}\right).$$ The Table~\ref{tab:different_values_of_gamma} suggests that $\gamma=0.99$ as a good choice for Algorithm~\ref{alg:PD}.

\begin{figure}[h!]
\centering
    \begin{subfigure}[b]{0.45\textwidth}
    \includegraphics[height=5cm]{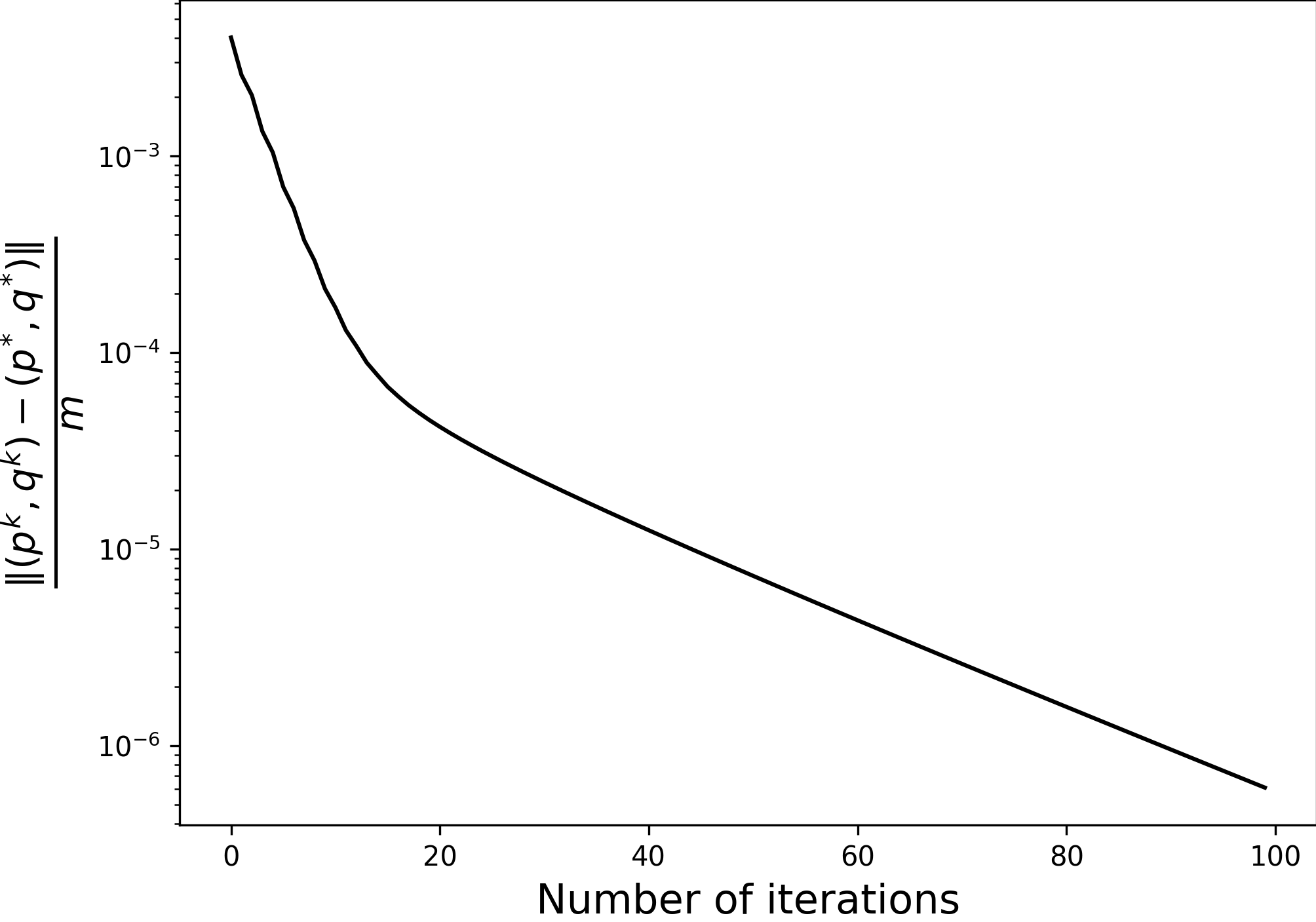}
    \caption{}
    \end{subfigure}
    \begin{subfigure}[b]{0.45\textwidth}
    \includegraphics[height=5cm]{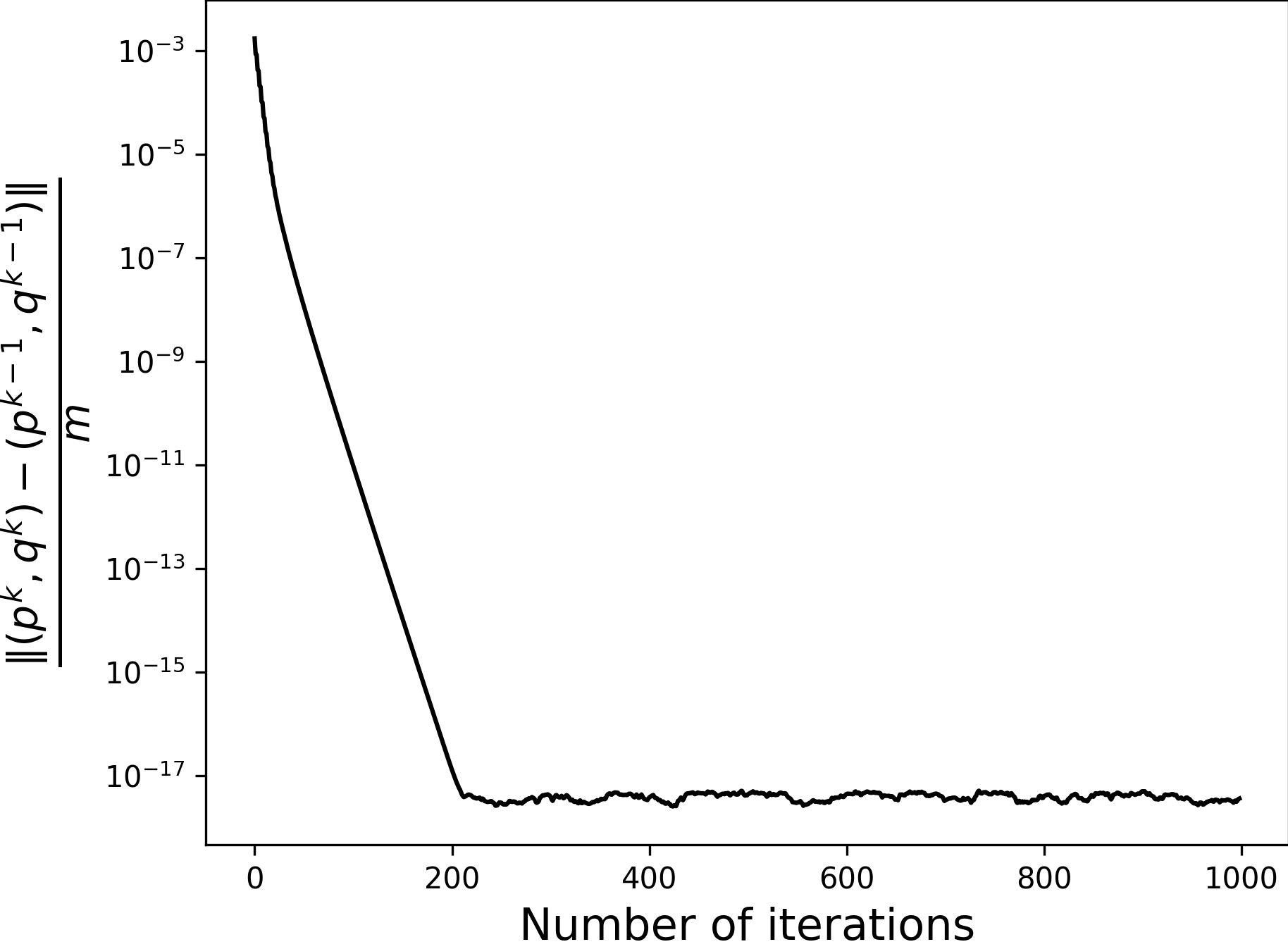}
    \caption{}
    \end{subfigure}
    \caption{(a) Normalized distance to the solution vs iterations, and (b) normalized final change iterates with $\lambda_{1}=0.01,\lambda_{2}=0.05,\lambda_3=0.0001$ and $\lambda_4=10$ for Algorithm~\ref{alg:PD}.}
    \label{fig:error for different lambda}
\end{figure}

Figure~\ref{fig:error for different lambda}\hyperref[fig:error for different lambda]{(a)} illustrates that Algorithm~\ref{alg:PD} converges linearly which supports the results presented in Corollary~\ref{corollary primal dual linear convergence}. In this figure, we analyzed the normalized distance to the solution $\frac{\|(\mathbf{p}^k,\mathbf{q}^k)-(\mathbf{p}^{*},\mathbf{q}^*)\|}{m}$ for the first 100 iterations, where $(\mathbf{p}^k,\mathbf{q}^k)$ is the sequence generated by Algorithm~\ref{alg:PD} at each iteration $k$  and  $(\mathbf{p}^*,\mathbf{q}^*)$ is an approximate optimal values obtained by running our Algorithm~\ref{alg:PD} for 200 iterations as the change of iterates are almost same after 200 iterations which are shown in Figure~\ref{fig:error for different lambda}\hyperref[fig:error for different lambda]{(b)}. 

We compared Algorithm~\ref{alg:PD} and Douglas--Rachford algorithm in product space discussed in \eqref{product space DR} applied to \eqref{monotone inclusion n=2*}. For this comparison, we took four test images~\cite{cvg2023} with different sizes shown in Figure \ref{fig:restored lena shapman}. Table~\ref{tab:comparison table 2} reports the average number of iterations, normalized distance to the solution and elapsed time to achieve the stopping criteria $$\frac{\|(\mathbf{p}^k,\mathbf{q}^k)-(\mathbf{p}^{k-1},\mathbf{q}^{k-1})\|}{m}\leq 10^{-4} $$ for 10 different instances of Gaussian noise.
\begin{figure}[h!]
\centering
    \begin{subfigure}[b]{0.23\textwidth}
    \centering
    \includegraphics[width=0.9\textwidth]{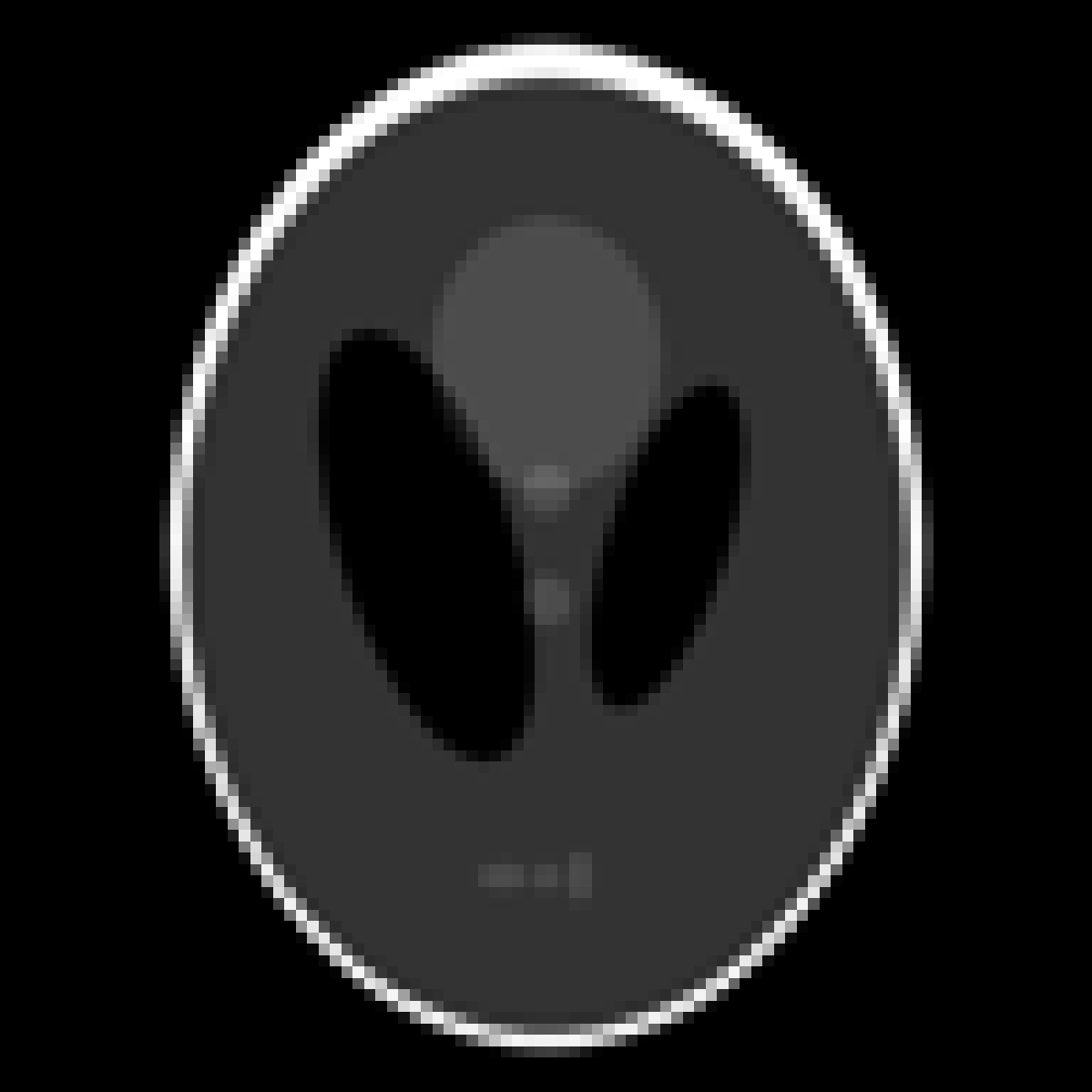}
    \end{subfigure}
    \begin{subfigure}[b]{0.23\textwidth}
    \centering
    \includegraphics[width=0.9\textwidth]{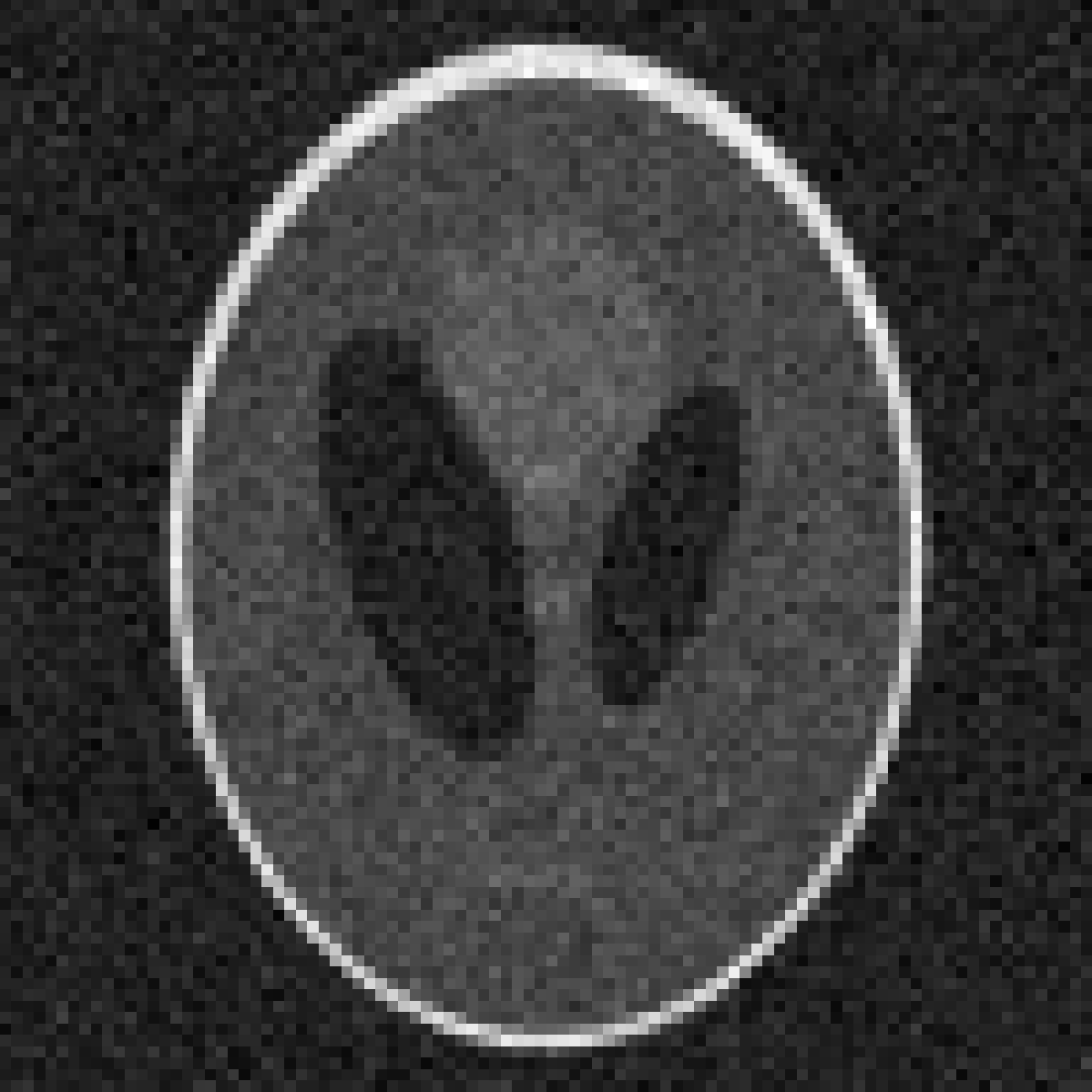}
    \end{subfigure}
    \begin{subfigure}[b]{0.23\textwidth}
    \centering
    \includegraphics[width=0.9\textwidth]{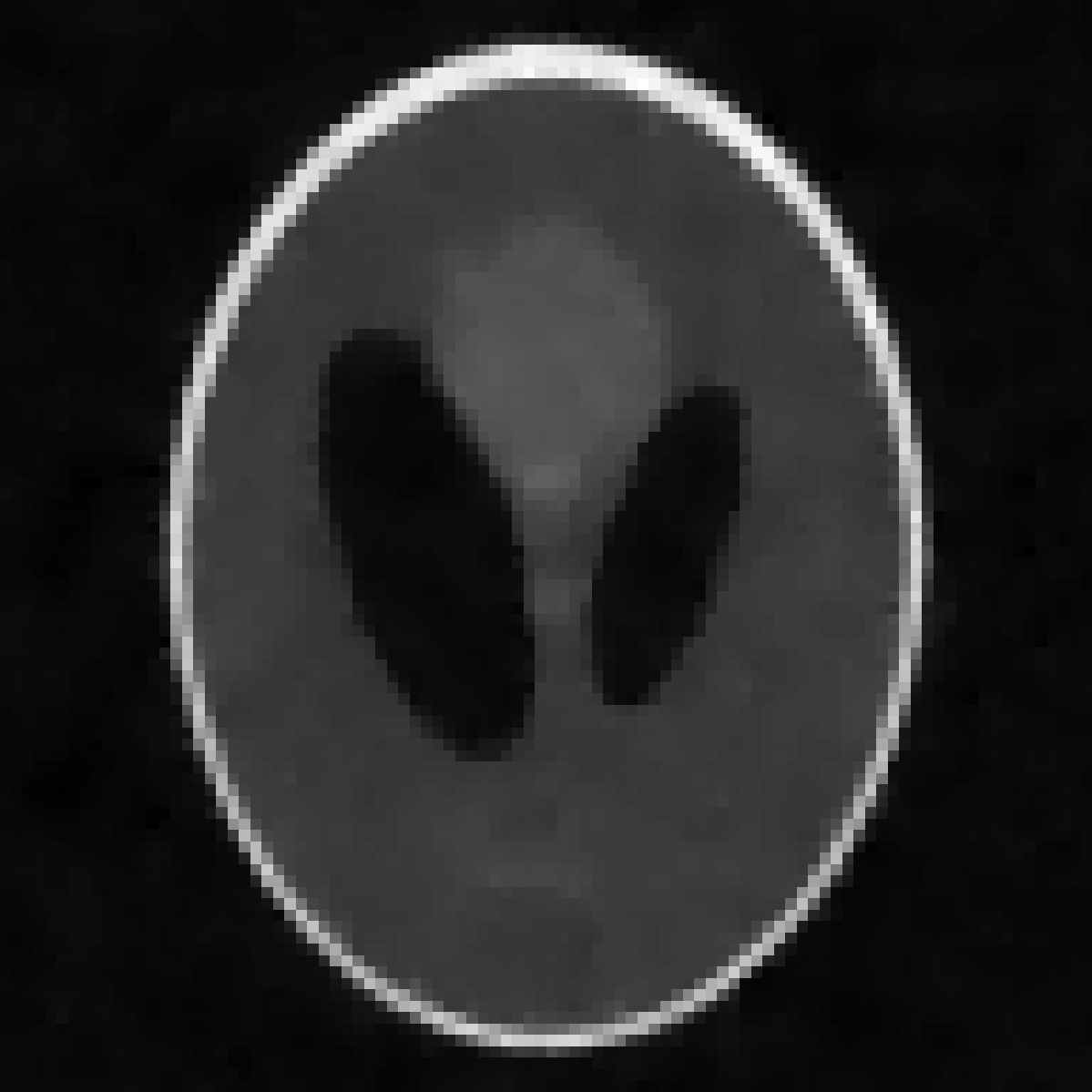}
    \end{subfigure}
    \begin{subfigure}[b]{0.23\textwidth}
    \centering
    \includegraphics[width=0.9\textwidth]{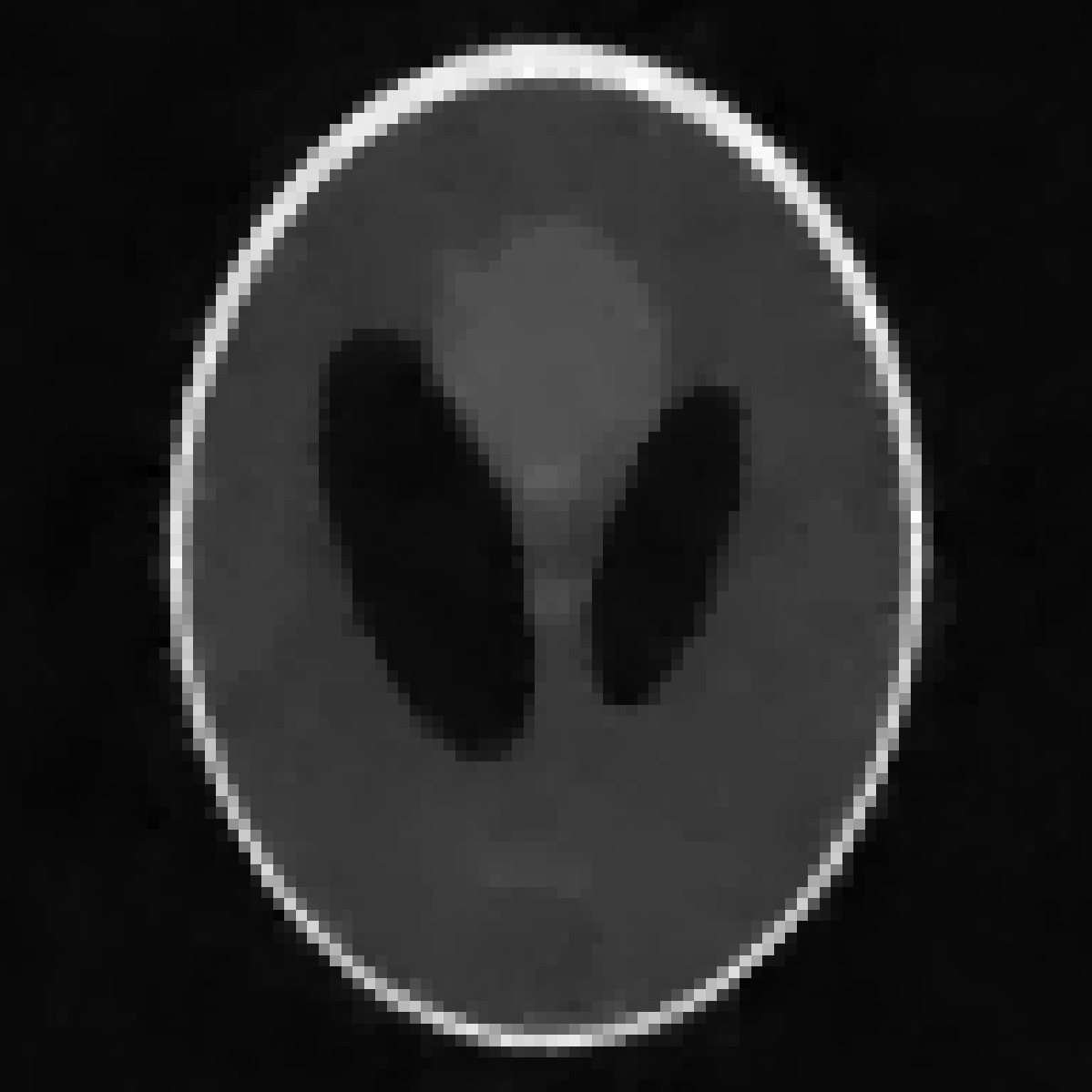}
    \end{subfigure}
    \begin{subfigure}[b]{0.23\textwidth}
    \centering
    \includegraphics[width=.9\textwidth]{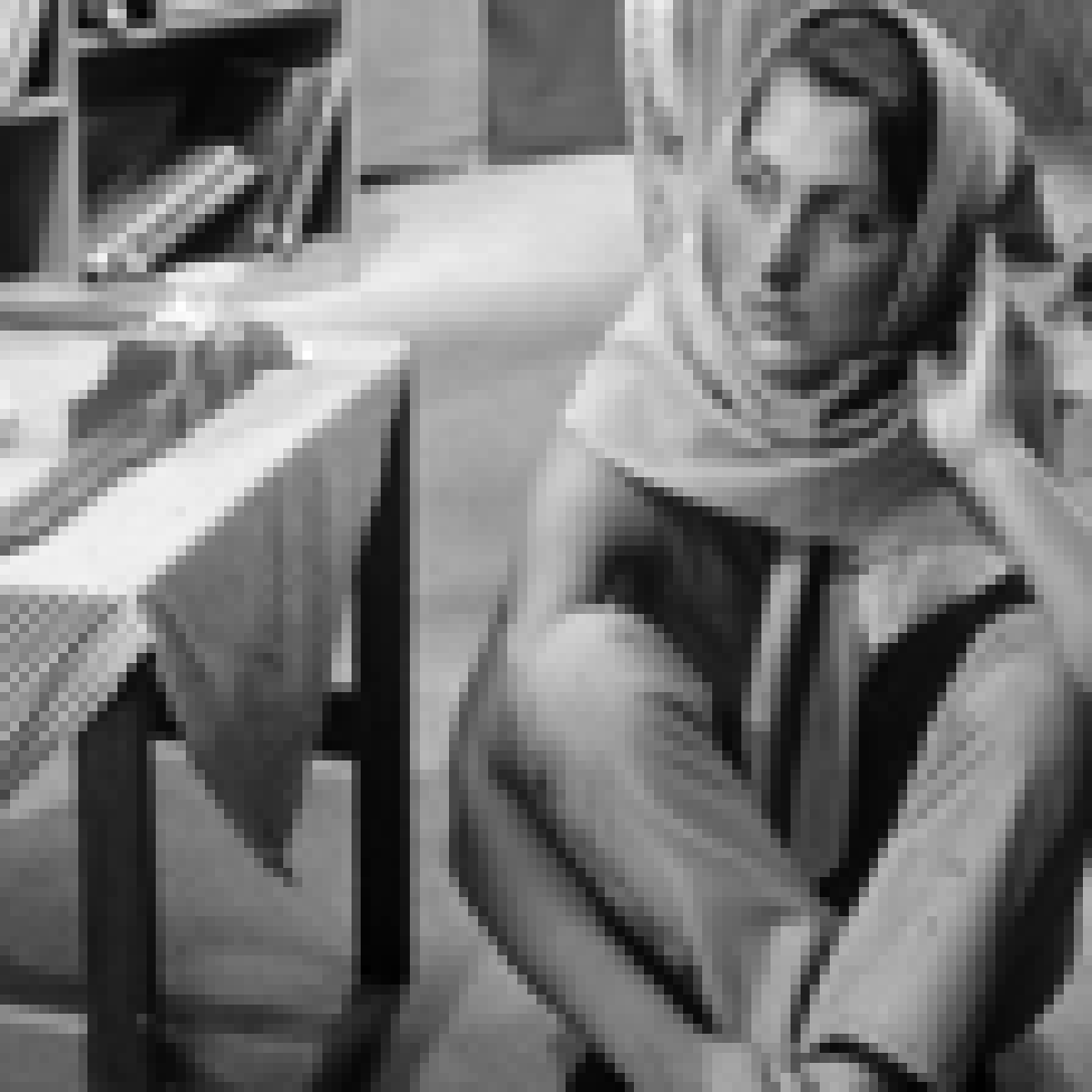}
    \end{subfigure}
    \begin{subfigure}[b]{0.23\textwidth}
    \centering
    \includegraphics[width=.9\textwidth]{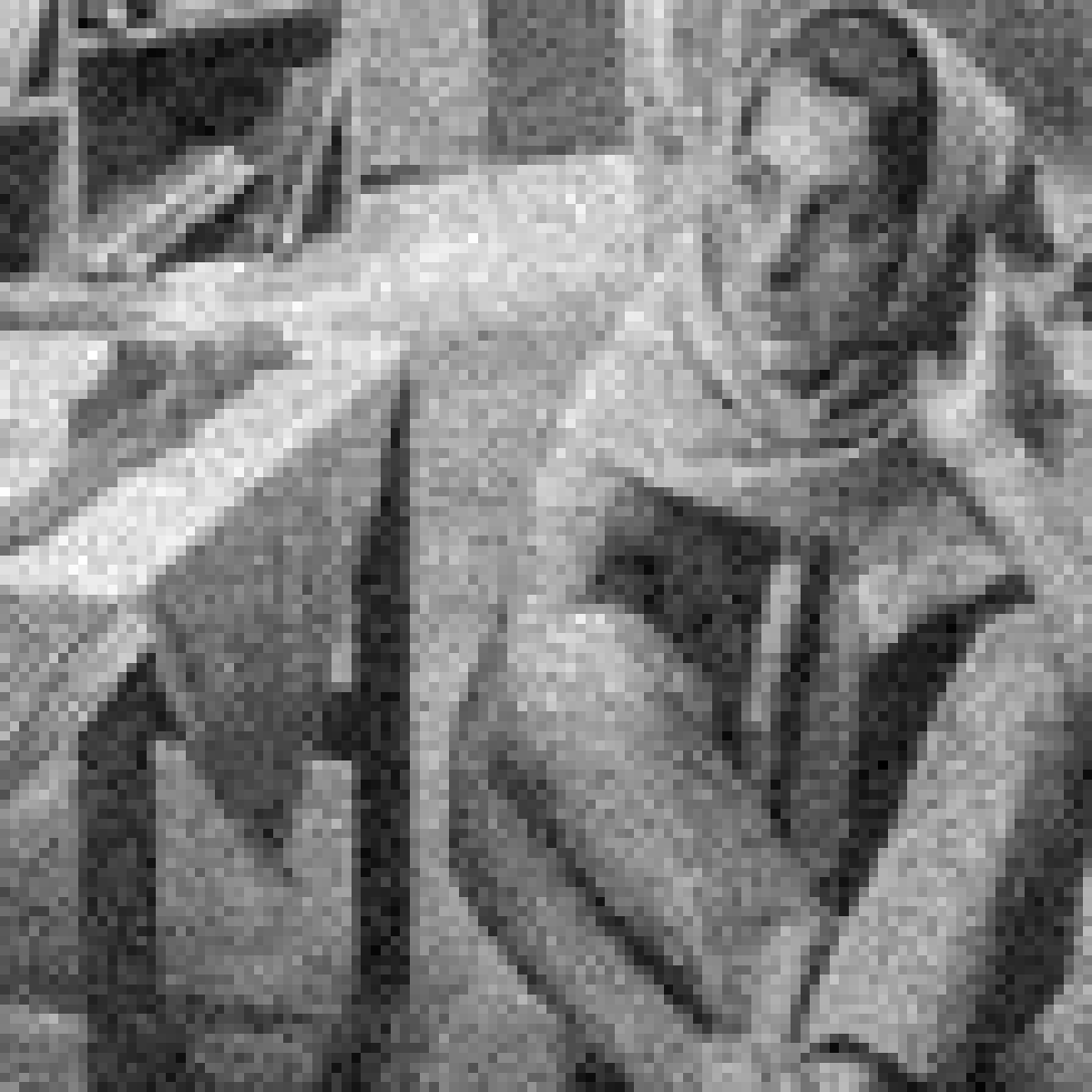}
    \end{subfigure}
    \begin{subfigure}{0.23\textwidth}
    \centering
    \includegraphics[width=.9\textwidth]{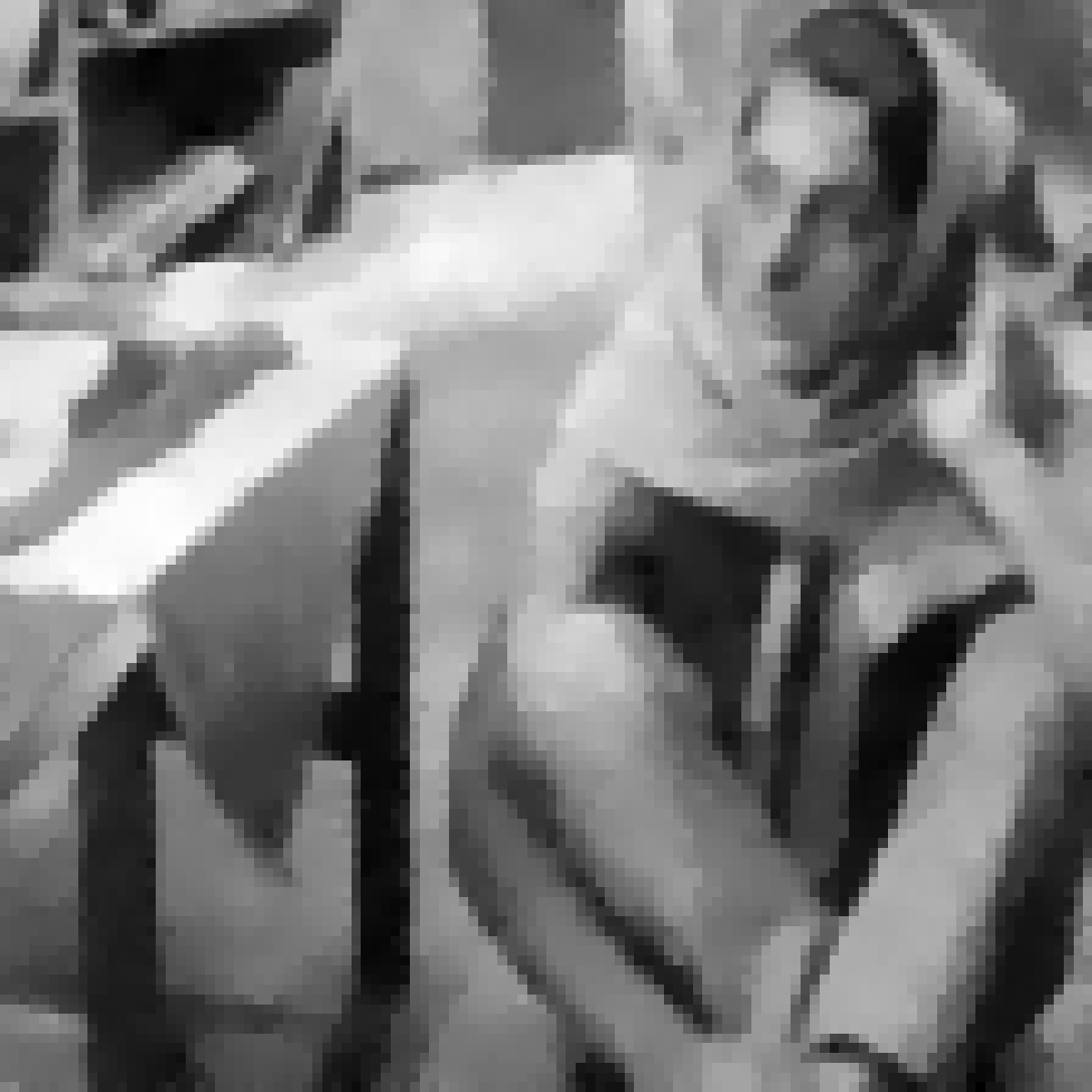}
    \end{subfigure}
    \begin{subfigure}[b]{0.23\textwidth}
    \centering
    \includegraphics[width=.9\textwidth]{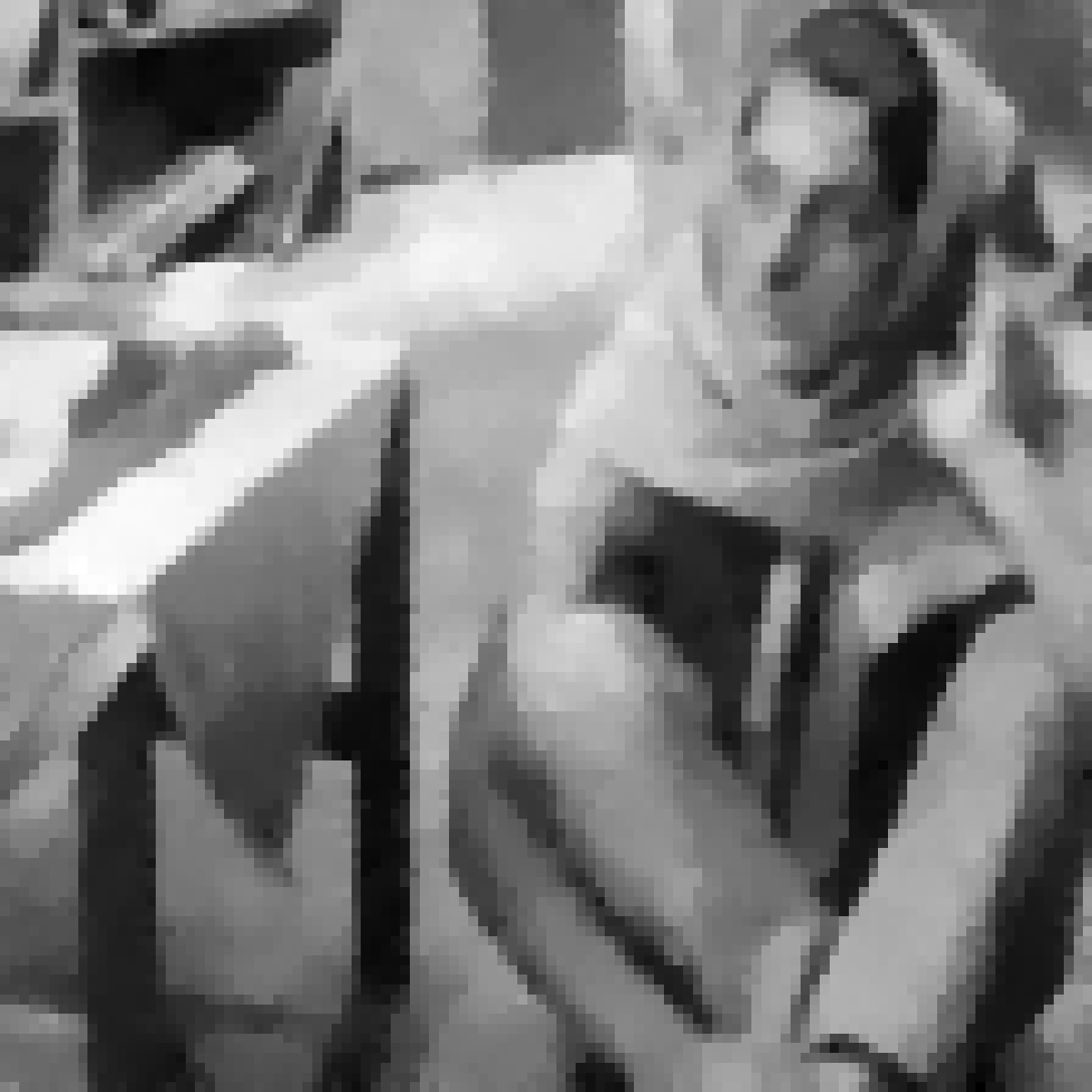}
    \end{subfigure}
    \begin{subfigure}[b]{0.23\textwidth}
    \centering
    \includegraphics[width=.9\textwidth]{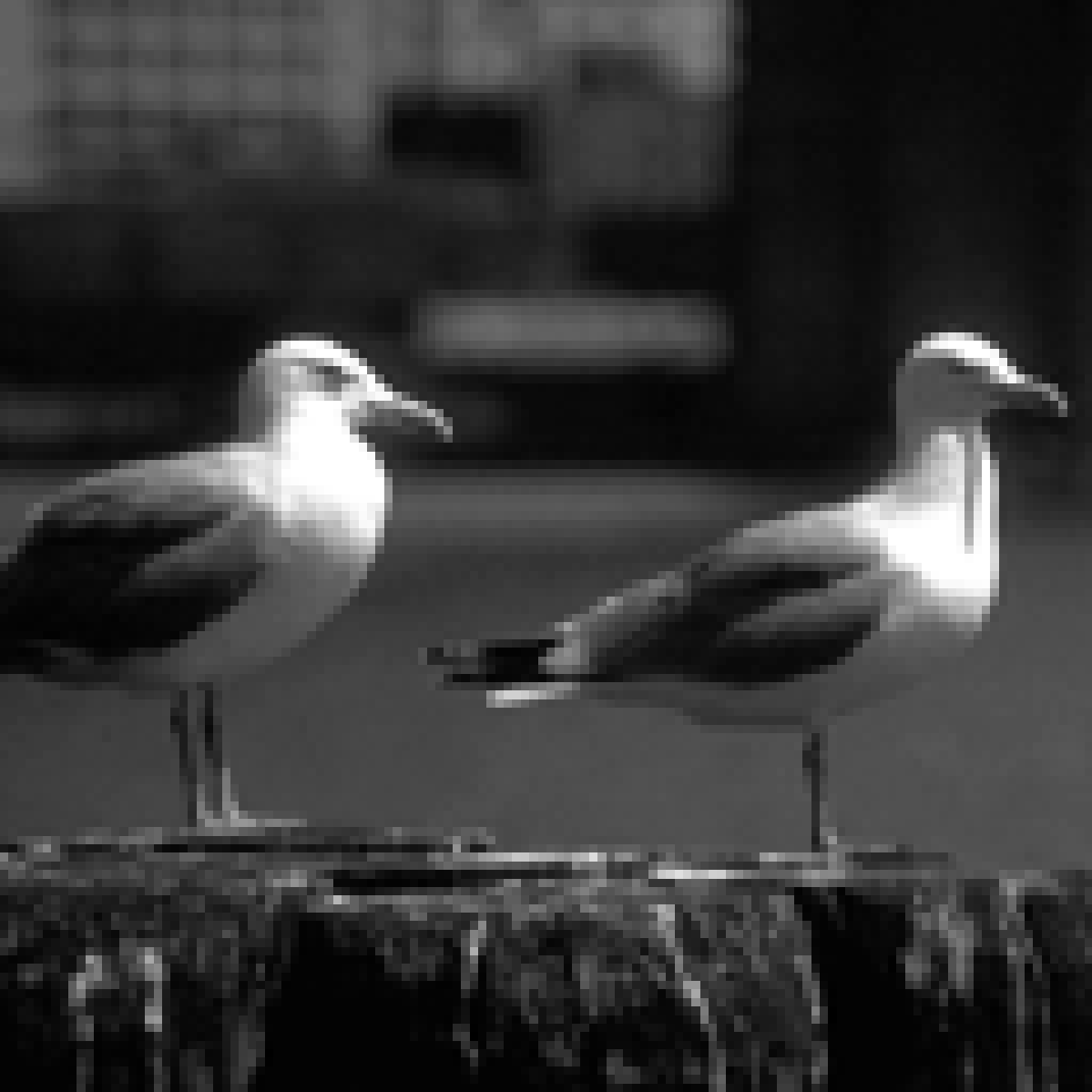}
    \end{subfigure}
    \begin{subfigure}[b]{0.23\textwidth}
    \centering
    \includegraphics[width=.9\textwidth]{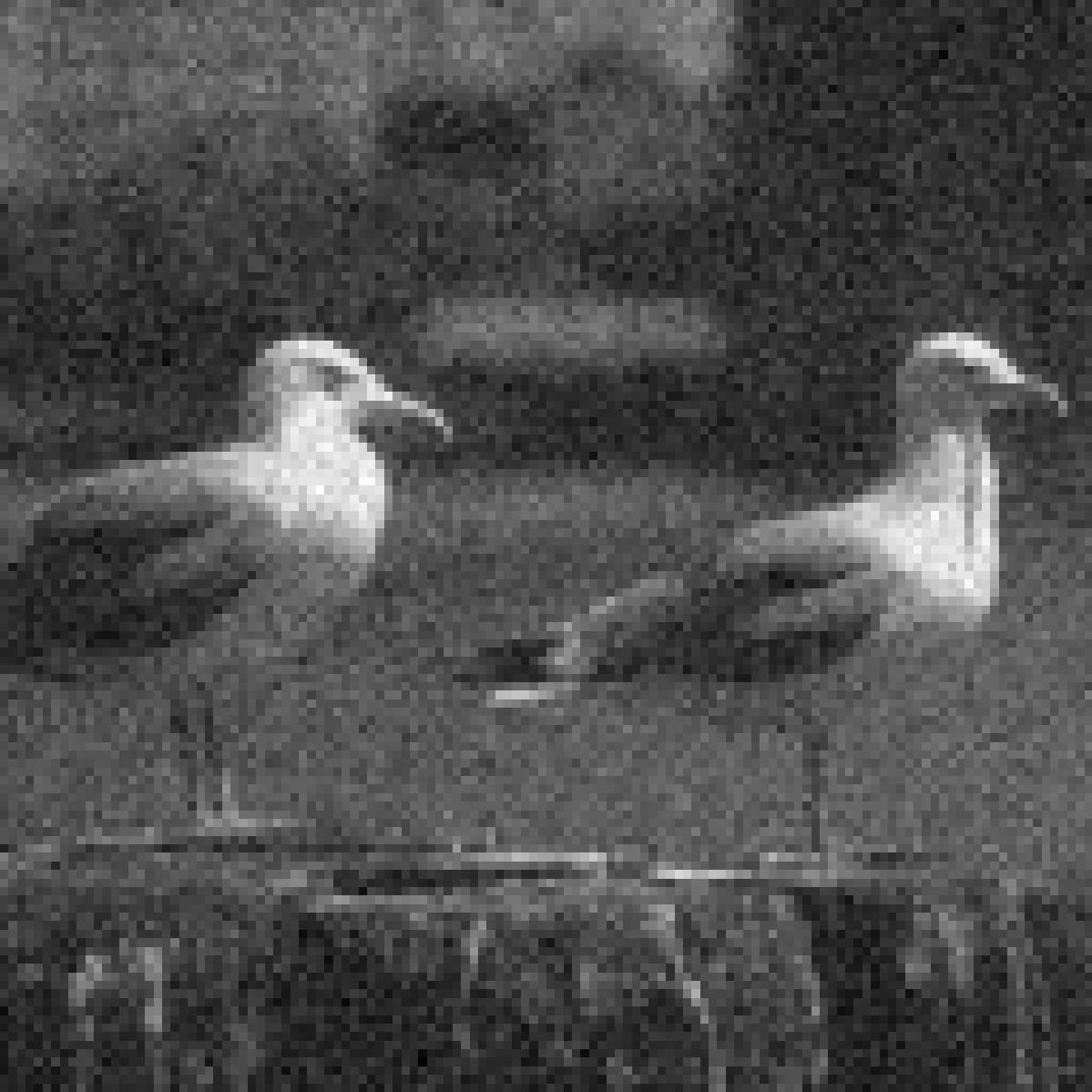}
    \end{subfigure}
    \begin{subfigure}[b]{0.23\textwidth}
    \centering
    \includegraphics[width=.9\textwidth]{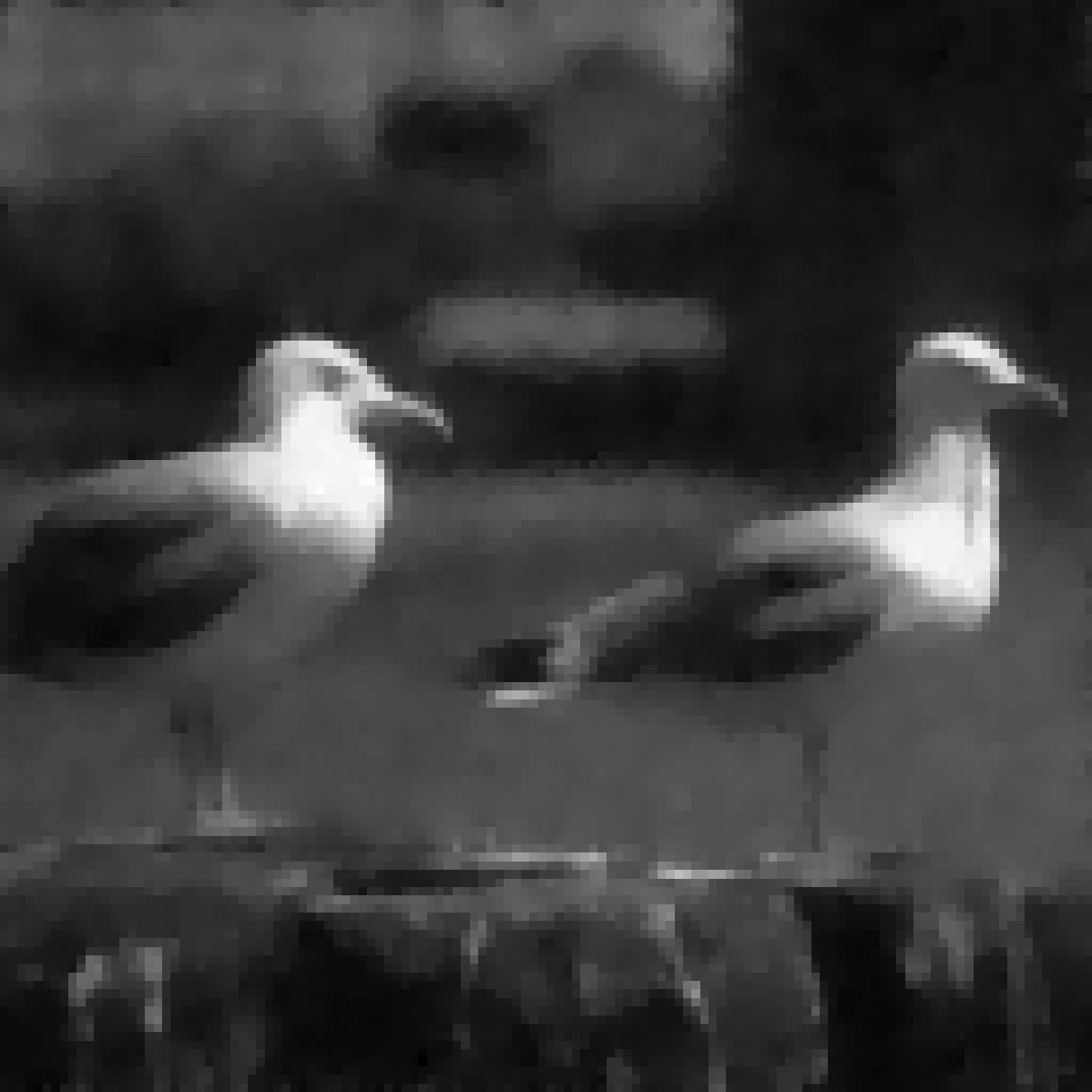}
    \end{subfigure}
    \begin{subfigure}[b]{0.23\textwidth}
    \centering
    \includegraphics[width=.9\textwidth]{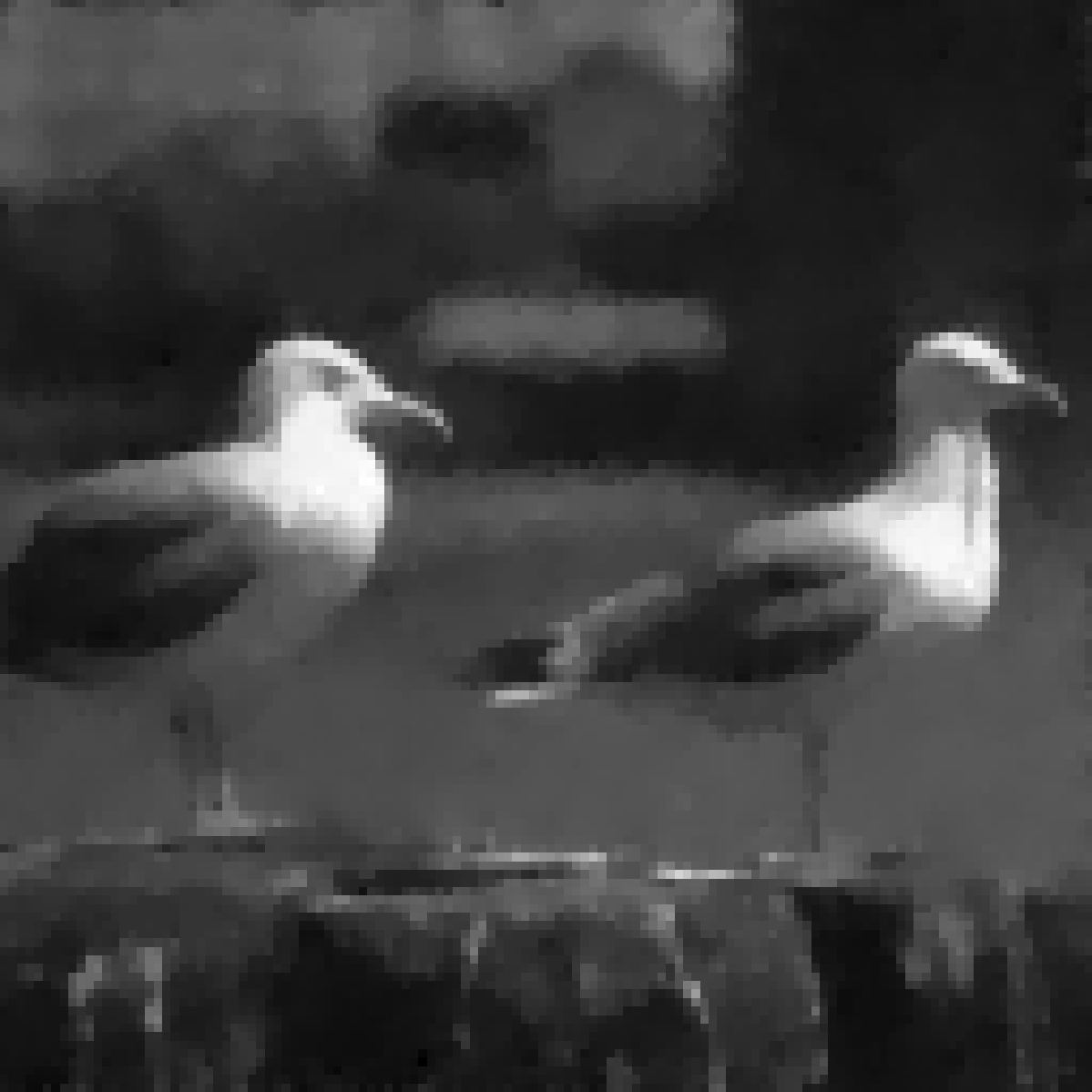}
    \end{subfigure}
    \begin{subfigure}[b]{0.23\textwidth}
    \centering
    \includegraphics[width=0.9\textwidth]{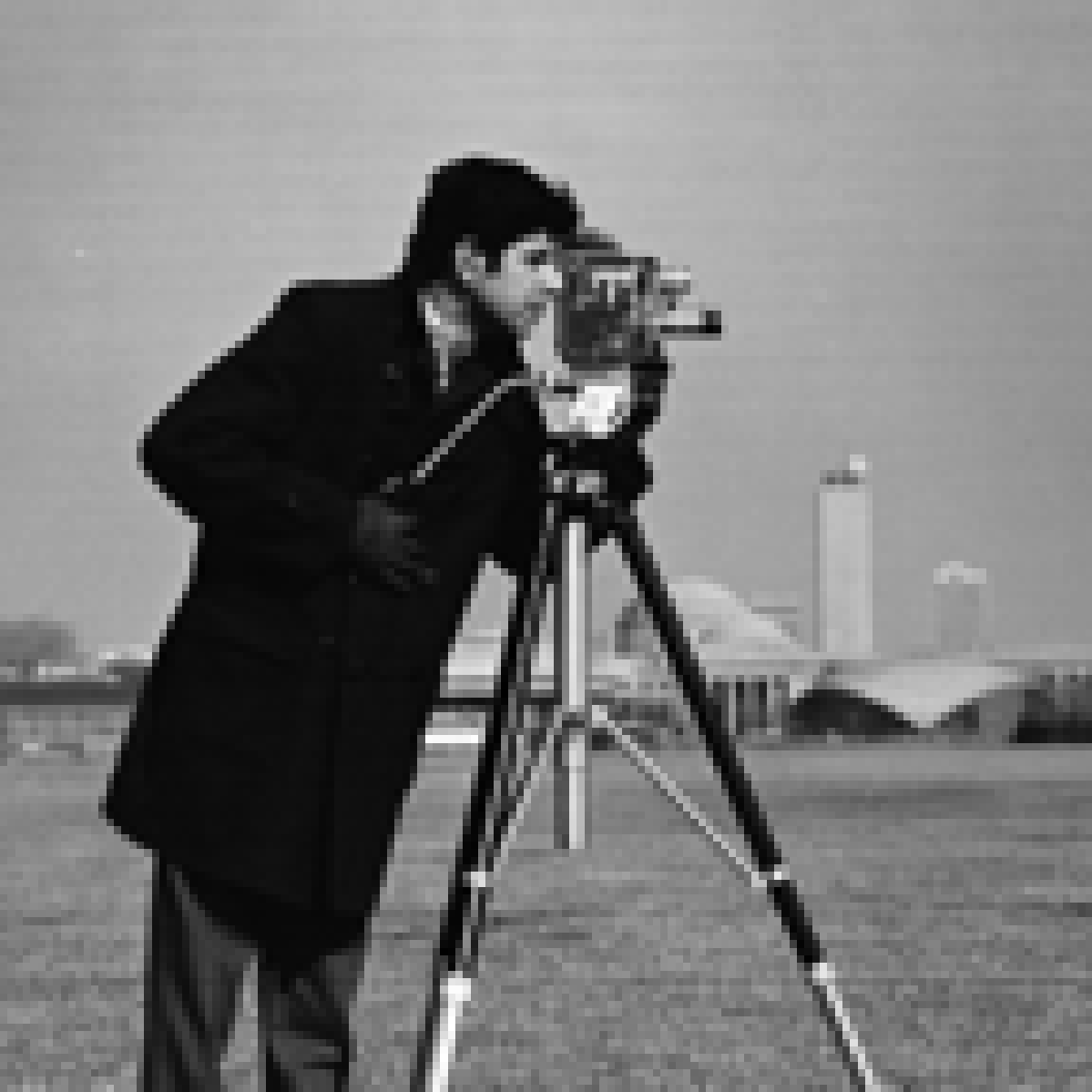}
    \caption{Original}
    \end{subfigure}
    \begin{subfigure}[b]{0.23\textwidth}
    \centering
    \includegraphics[width=0.9\textwidth]{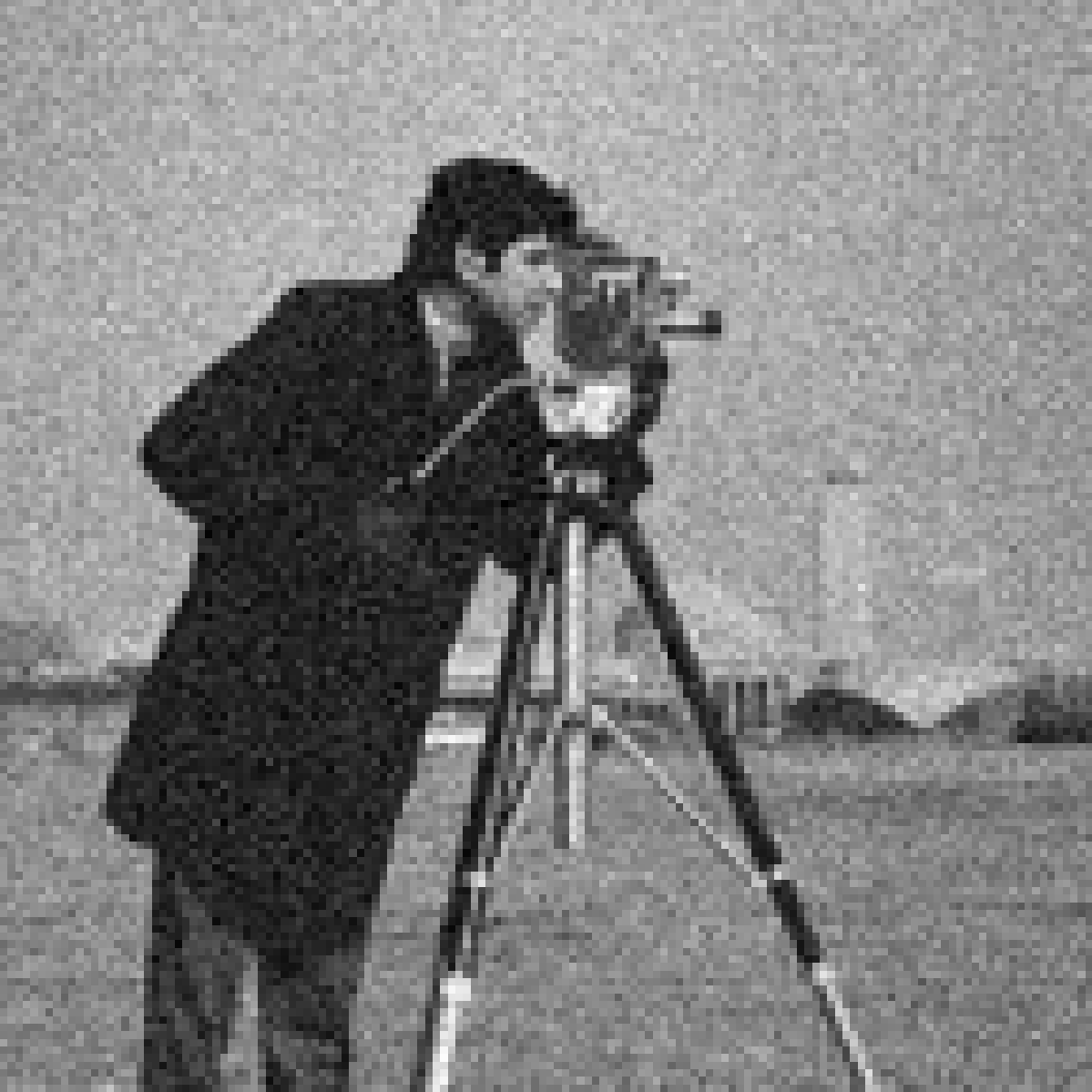}
     \caption{Noisy}
    \end{subfigure}
    \begin{subfigure}[b]{0.23\textwidth}
    \centering
    \includegraphics[width=0.9\textwidth]{images/cam_Denoised_image_la1_0.01.png}
    \caption{Restored (Algorithm~\ref{alg:PD})}
    \end{subfigure}
    \begin{subfigure}[b]{0.23\textwidth}
    \centering
    \includegraphics[width=0.9\textwidth]{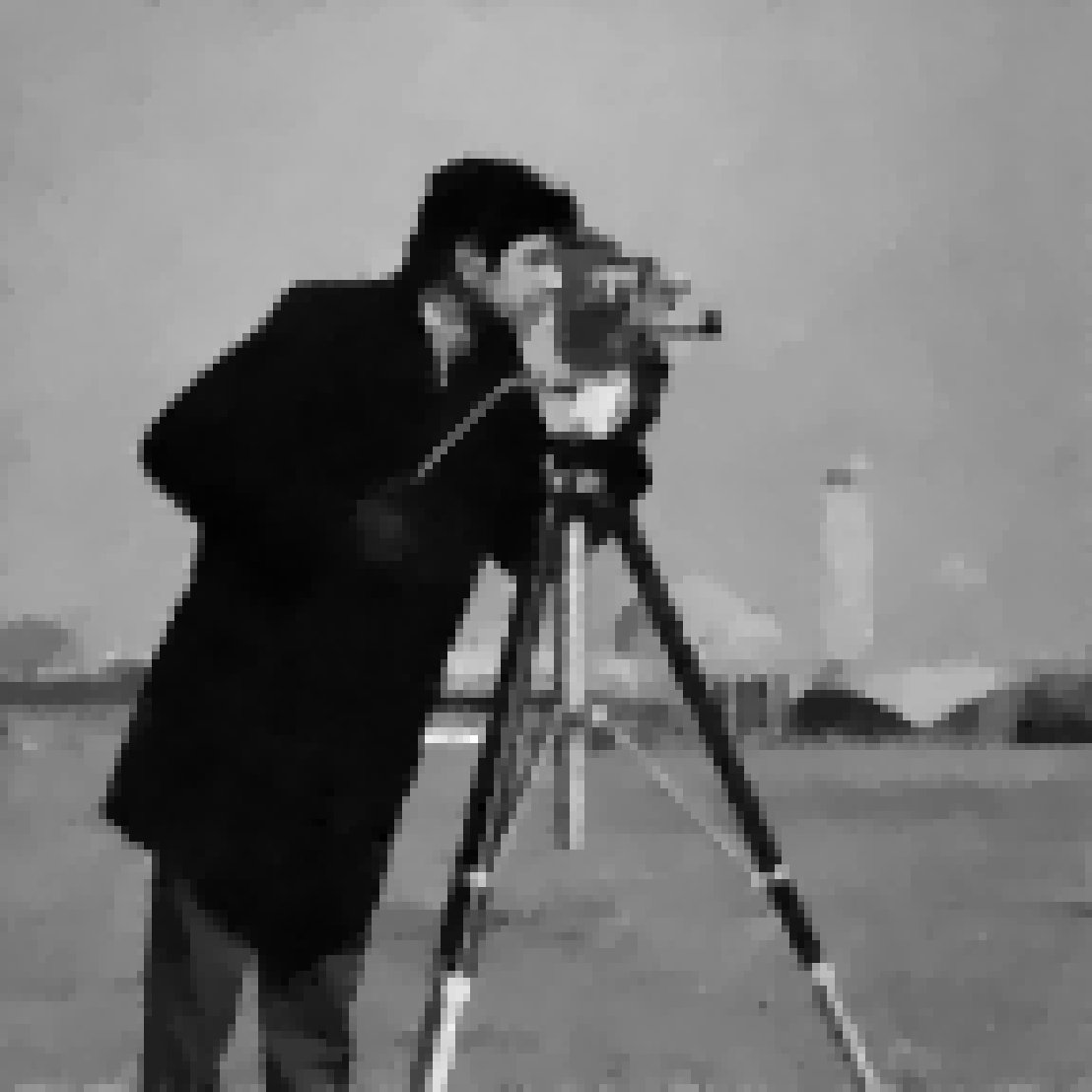}
    \caption{Restored (DR)}
    \end{subfigure}
    \caption{Original, noisy and denoised images for Algorithm~\ref{alg:PD} and Douglas--Rachford algorithm in product space.}
    \label{fig:restored lena shapman}
\end{figure}
\begin{table}[h!]
    \centering
    \renewcommand{\arraystretch}{1.2}
     \caption{Summary of the computational results for ten random instances.}
    \begin{tabular}{ |c|c|c|c|c|c|c|c|c| } 
\hline
 Image& Size ($M$) &\multicolumn{2}{c|}{\textbf{$k$}} & \multicolumn{2}{c|}{\textbf{$\frac{\|(\mathbf{p}^k,\mathbf{q}^k)-(\mathbf{p}^*,\mathbf{q}^*)\|}{m}$}}& \multicolumn{2}{c|}{Time (s)} \\ \cline{3-8}
          &          & Alg~\ref{alg:PD} & DR & Alg~\ref{alg:PD} & DR  &Alg~\ref{alg:PD} & DR\\
\hline
\multirow{4}{0.5em}\\Shepp-Logan phantom& 96 &9&26&$3.83\times 10^{-5}$&$2.71\times 10^{-5}$&5.10&13.98\\ 
Barbara & 112 &10&27& $2.49\times 10^{-5}$&$2.08\times 10^{-5}$&5.84&14.68\\ 
Bird & 128  &8&22& $2.82\times 10^{-5}$&$2.05\times 10^{-5}$&4.81&12.09\\
Cameraman& 144  &10&22&$2.11\times 10^{-5}$&$2.13\times 10^{-5}$&5.88&12.54\\
\hline
\end{tabular}
    \label{tab:comparison table 2}
\end{table}

The normalized distance to the solution $\frac{\|(\mathbf{p}^k,\mathbf{q}^k)-(\mathbf{p}^{*},\mathbf{q}^{*})\|}{m}$ was plotted against $100$ iterations for both algorithms across the four test images showed in Figure \ref{fig:error graph}. For Algorithm~\ref{alg:PD}, Corollary~\ref{corollary primal dual linear convergence} guarantees the linear convergence which have shown in Figure~\ref{fig:error graph}. However, convergence behavior of DR algorithm does not clearly show linear convergence, as seen in the same figures. Also, Algorithm~\ref{alg:PD} performs better than Douglas--Rachford algorithm on these instances in terms of iterations, normalized distance to the solution and time based on Table~\ref{tab:comparison table 2} and Figure~\ref{fig:error graph}.

\begin{figure}[h!]
    \begin{subfigure}[b]{0.5\textwidth}
    
    \includegraphics[width=0.9\textwidth]{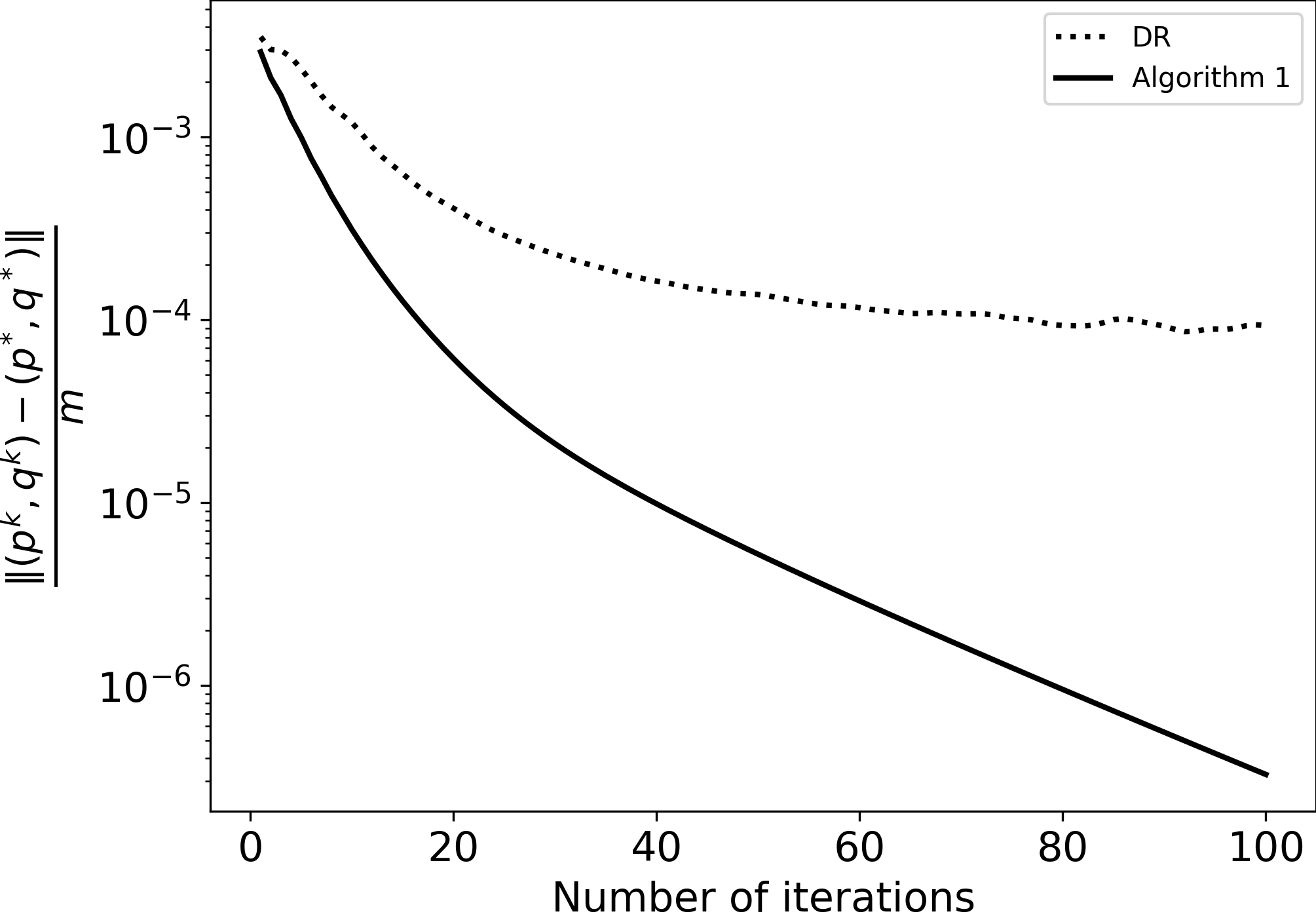}
    \caption{Shepp-Logan phantom}
    \end{subfigure}
    \begin{subfigure}[b]{0.5\textwidth}
    \includegraphics[width=0.9\textwidth]{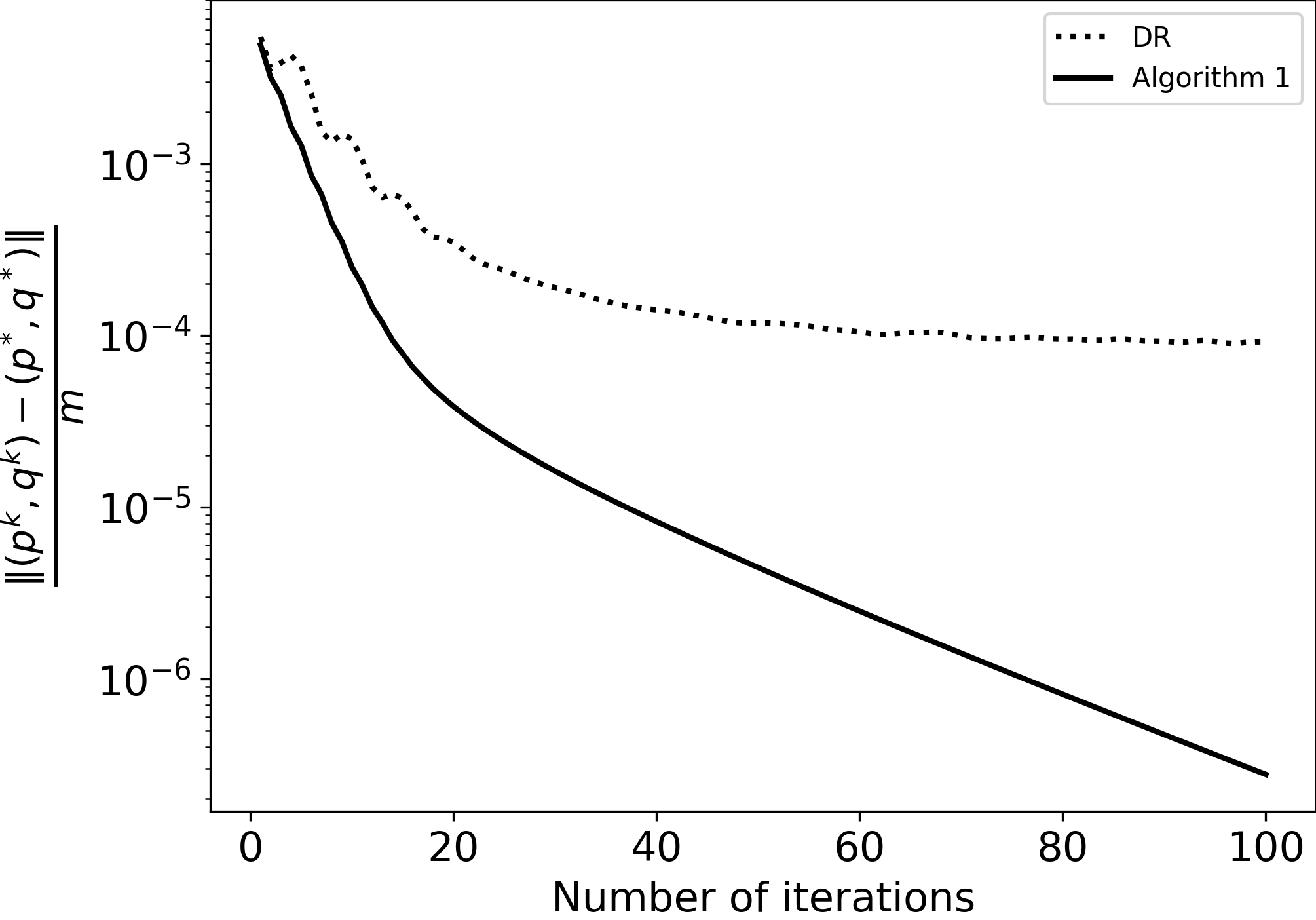}
    \caption{Barbara}
    \end{subfigure}
    \begin{subfigure}[b]{0.5\textwidth}
    \includegraphics[width=0.9\textwidth]{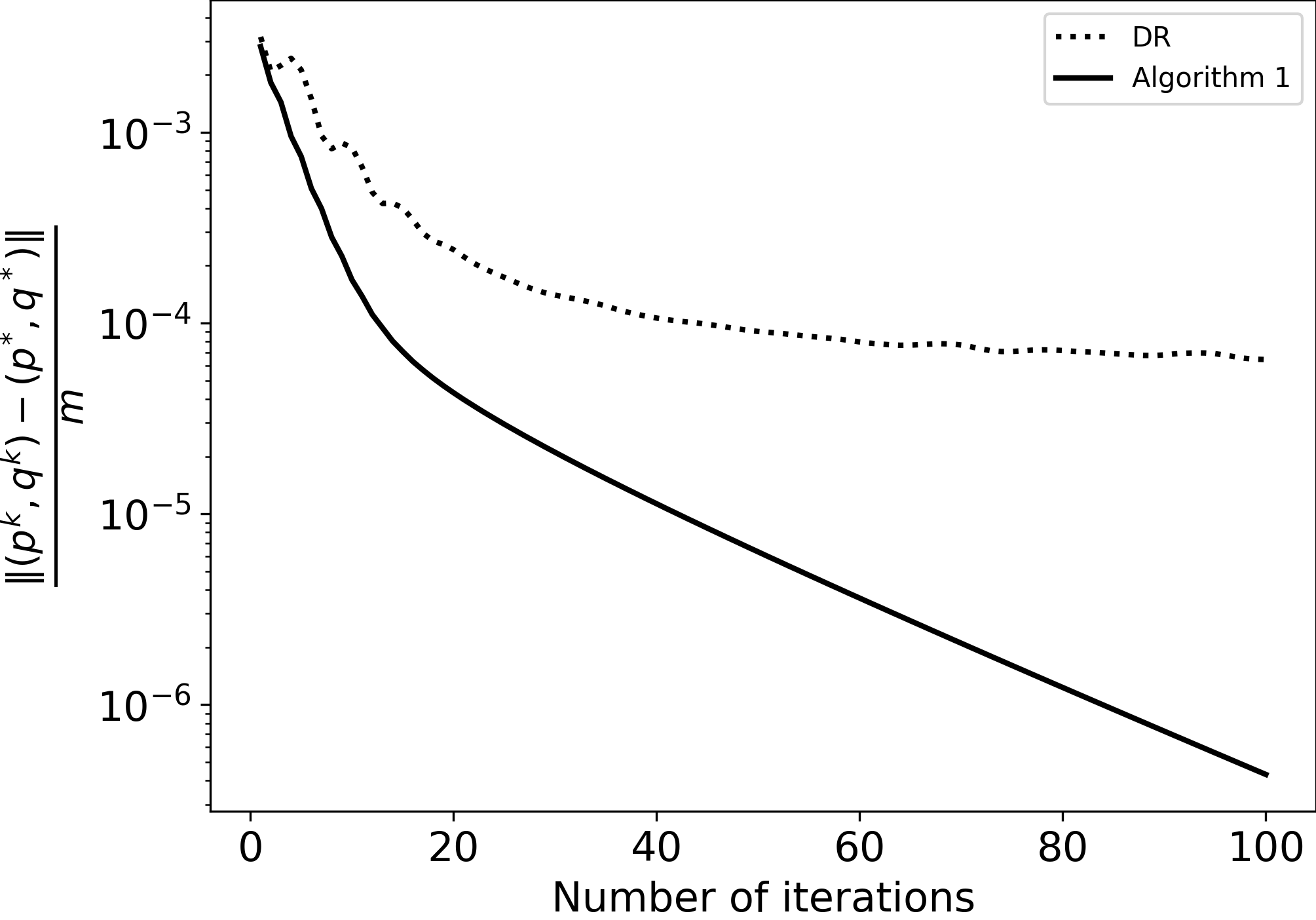}
    \caption{Bird}
    \end{subfigure}
    \begin{subfigure}[b]{0.5\textwidth}
    \includegraphics[width=0.9\textwidth]{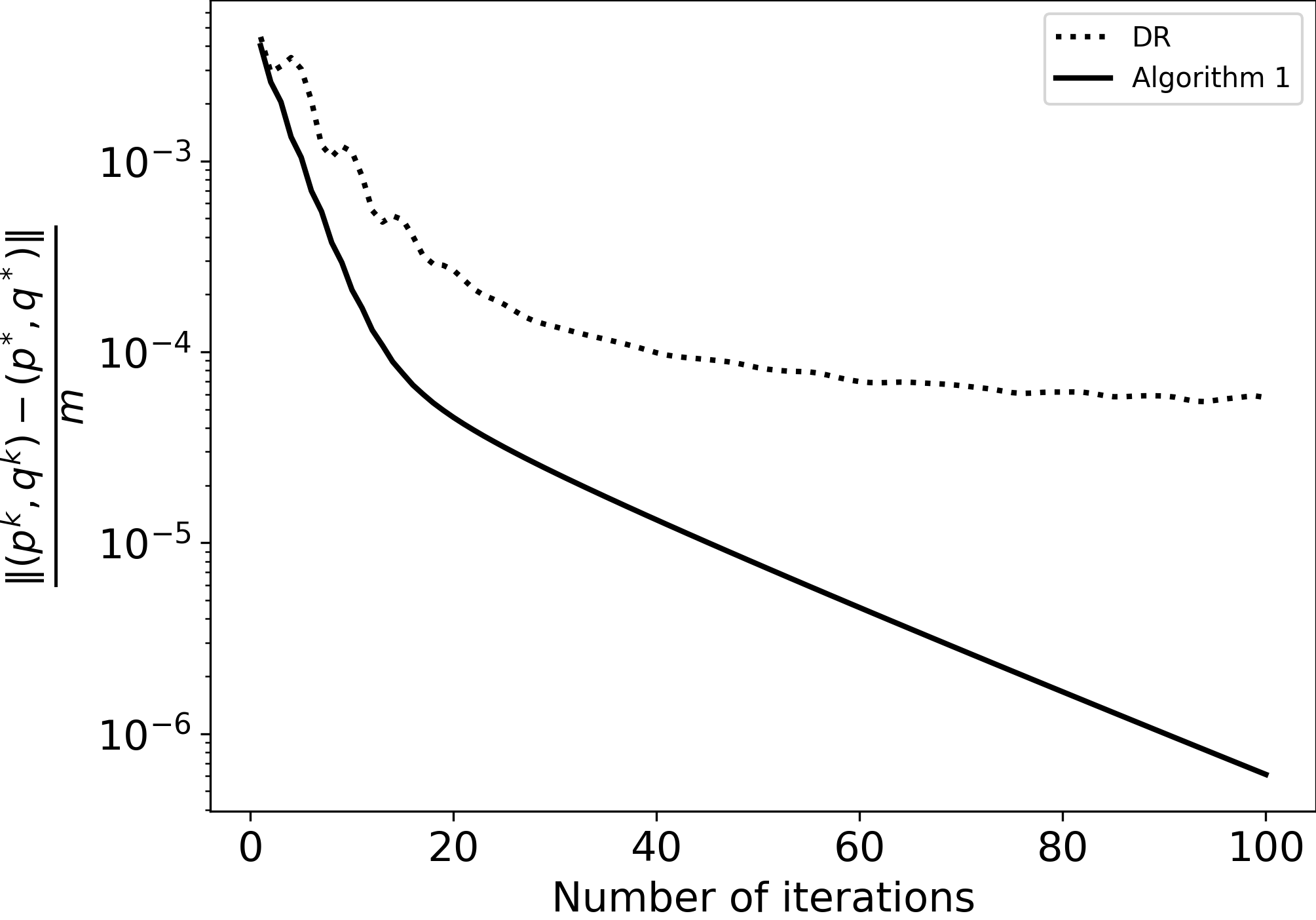}
     \caption{Cameraman}
    \end{subfigure}
    \caption{Relative error vs iterations of Algorithm~\ref{alg:PD} and DR algorithm for (a) ``Shepp-Logan phantom" (b) ``Barbara", (c) ``Bird", and (d) ``Cameraman" images.}
    \label{fig:error graph}
\end{figure}
\section{Conclusions}\label{s: conclusions}
 In this work, we established linear convergence of the resolvent splitting algorithm due to Malitsky and Tam~\cite{malitsky2023resolvent} for finding a zero in the sum of $n$ maximally monotone operators under two different sets of assumptions. In the first setting, the first $(n-1)$ operators are maximally monotone and Lipschitz, and the last operator is maximally strongly monotone. In the other setting, the first $(n-1)$ operators are maximally strongly monotone and Lipschitz, and the last one is maximally monotone. We then applied our result under the first set of assumptions to derive linear convergence of a primal--dual algorithm (Algorithm~\ref{alg:PD}) for a convex minimization problem involving infimal convolution, and presented experimental results in the context of image denoising. Our experiments demonstrate that Algorithm~\ref{alg:PD} achieves linear convergence. Furthermore, we conducted a comparative analysis with the Douglas--Rachford (DR) algorithm applied in the product space. For this problem, our results suggest that Algorithm~\ref{alg:PD} performs better than Douglas--Rachford algorithm in terms of both accuracy and time. 
 
 One possible direction for further research is to investigate tight linear convergence rates of Algorithm~\ref{alg:PD} similar to in \cite{giselsson2017tight}. In our proof of \ref{theorem for linear convergence}, we rely on the inequality \eqref{inequality}, which is unlikely to be tight.

\section*{Acknowledgments}
FAS was supported in part by a Research Training Program Scholarship from the Australian Commonwealth Government and the University of Melbourne.
MKT was supported in part by Australian Research Council grant DP230101749.

\section*{Disclosure of Interest}
The authors report there are no competing interests to declare.

\bibliographystyle{plain}
\bibliography{ref}

\end{document}